\documentclass[11pt,oneside,english,reqno]{article}
\usepackage[UKenglish]{babel}
\usepackage{amsthm,amsmath,amsfonts,amssymb}
\usepackage{mathtools,stmaryrd,mathrsfs}
\usepackage{newtxtext}
\usepackage{microtype} 
\usepackage[dvipsnames]{xcolor}
\usepackage[numbers]{natbib}
\usepackage[colorlinks=false]{hyperref}
\hypersetup{
    colorlinks,
		linktocpage,
    linkcolor={blue!80!black},
    citecolor={blue!80!black},
    urlcolor={blue!80!black}
}
\usepackage{orcidlink}
\usepackage{authblk} 
\usepackage[a4paper]{geometry}
\usepackage{nicefrac}
\usepackage{mathtools}
\usepackage{cleveref}
\usepackage{tikz, pgfplots} 

\theoremstyle{plain}
\newtheorem{theorem}{Theorem}[section]
\newtheorem{lemma}[theorem]{Lemma}
\newtheorem{proposition}[theorem]{Proposition}

\theoremstyle{remark}
\newtheorem{definition}[theorem]{Definition}

\newtheorem{remark}[theorem]{Remark}

\newtheorem*{assumption*}{Assumption}

\newcommand{\mcC}{\mathcal{C}}

\newcommand{\mcF}{\mathcal{F}}

\newcommand{\mcM}{\mathcal{M}}

\newcommand{\mcP}{\mathcal{P}}

\newcommand{\mbB}{\mathbb{B}}

\newcommand{\mbE}{\mathbb{E}}

\newcommand{\mbN}{\mathbb{N}}

\newcommand{\mbP}{\mathbb{P}}

\newcommand{\mbR}{\mathbb{R}}

\newcommand{\mfp}{\mathfrak{p}}

\newcommand{\msD}{\mathscr{D}}

\newcommand{\NN}{\mathbb{N}}

\newcommand{\PP}{\mathbb{P}}

\newcommand{\RR}{\mathbb{R}}

\newcommand{\ZZ}{\mathbb{Z}}

\newcommand{\cC}{\mathcal{C}}

\newcommand{\cF}{\mathcal{F}}

\newcommand{\cK}{\mathcal{K}}

\newcommand{\cM}{\mathcal{M}}

\newcommand{\cP}{\mathcal{P}}

\newcommand{\cS}{\mathcal{S}}

\newcommand{\second}{2\textsuperscript{nd} }
\newcommand{\third}{3\textsuperscript{rd} }

\DeclarePairedDelimiter{\norm}{\lVert}{\rVert}
\DeclarePairedDelimiter{\abs}{\lvert}{\rvert}
\makeatletter
\newcommand{\opnorm}{\@ifstar\@opnorms\@opnorm}
\newcommand{\@opnorms}[1]{%
  \left|\mkern-1.5mu\left|\mkern-1.5mu\left|
   #1
  \right|\mkern-1.5mu\right|\mkern-1.5mu\right|
}
\newcommand{\@opnorm}[2][]{%
  \mathopen{#1|\mkern-1.5mu#1|\mkern-1.5mu#1|}
  #2
  \mathclose{#1|\mkern-1.5mu#1|\mkern-1.5mu#1|}
}
\makeatother
\DeclarePairedDelimiterX{\inp}[2]{\langle}{\rangle}{#1, #2}
\DeclarePairedDelimiterX\set[1]\lbrace\rbrace{\, #1 \,}
\newcommand{\dequal}{\overset{\operatorname{d}}{=}}
\def \defby{\vcentcolon =}
\newcommand{\dd}{\,\mathrm{d}}
\renewcommand{\epsilon}{\varepsilon}
\newcommand{\indic}{\mathbf{1}}

\newcommand{\conv}{\ast}
\newcommand{\from}{\colon}

\def \N{\mathbb{N}}

\def \R{\mathbb{R}}

\newcommand{\floor}[1]{\lfloor #1 \rfloor}

\DeclareMathOperator*{\esssup}{ess\,sup}

\begin{document}

\title{Pathwise Uniqueness for Multiplicative Young and Rough Differential Equations Driven by Fractional Brownian Motion}

\author[1]{Toyomu Matsuda \orcidlink{0000-0002-2422-0863}}
\affil[1]{Independent.  \protect\\
    \href{mailto:toyomumatsudawork@gmail.com}{toyomumatsudawork@gmail.com}}
\author[2]{Avi Mayorcas \orcidlink{0000-0003-4133-9740}}
\affil[2]{Department of Mathematical Sciences, University of Bath. \protect\\
    \href{mailto:am2735@bath.ac.uk}{am2735@bath.ac.uk}}

\date{}

\maketitle

\begin{abstract}
 We show \emph{pathwise uniqueness} of multiplicative SDEs, in arbitrary dimensions, driven by fractional Brownian motion with Hurst parameter $H\in (1/3,1)$ with volatility coefficient $\sigma$ that is at least $\gamma$-H\"older continuous for $\gamma > \frac{1}{2H} \vee \frac{1-H}{H}$. This improves upon the long-standing results of \cite{lyons94,lyons1998,davie08} which cover the same regime but require $\sigma$ to be at least $\frac{1}{H}$-H\"older continuous. Our central innovation is to combine stochastic averaging estimates with refined versions of the stochastic sewing lemma, due to \cite{le20,gerencser2022,matsuda22}.

\bigskip

\noindent
{\sc Keywords and phrases.}
Stochastic differential equations, fractional Brownian motion, rough paths,  stochastic sewing, regularisation by noise.

\noindent
{\sc MSC 2020}. 60H10, 60G22, 60L20, 60H50.

\end{abstract}

\tableofcontents

\section{Introduction}\label{sec:intro}
This paper is concerned with \emph{pathwise} uniqueness of solutions to multiplicative SDEs driven by fractional Brownian motion (fBm),
\begin{align}\label{eq:main_sde}
	\dd X_t &= \sigma(X_t) \,\dd B_t^H,  \quad X_0 = x \in \mathbb{R}^{d_1},
\end{align}
where $\mbR^{d_1} \ni x\mapsto \sigma(x)$ takes values in the space of $d_1 \times d_2$ matrices and 
$B^H$ is a $d_2$-dimensional fBm with Hurst parameter $H \in (1/3, 1)$. We interpret \eqref{eq:main_sde} as either a Young differential equation, when $H\in (1/2,1)$ or a rough differential equation when $H\in (1/3,1/2]$. The tools developed herein naturally extend to the even rougher case $H\in (1/4,1/3)$, however, for the sake of concision they are not presented, see Rem.~\ref{rem:H_less_than_1/3} for a further discussion. Our main result is to obtain \emph{pathwise} uniqueness of Young (resp. rough) solutions to \eqref{eq:main_sde} for uniformly elliptic, volatility functions $\sigma$ which are $\gamma$-H\"older continuous for all
\begin{equation*}
	\gamma > \max\left\{\frac{1}{2H}, \frac{1-H}{H} \right\} = 
	\begin{cases}
		\frac{1}{2H}, \quad & \text{if } H \in (1/2,1), \\
		\frac{1-H}{H}, &  \text{if } H \in (1/3,1/2].
	\end{cases}
\end{equation*}
The ellipticity condition on $\sigma$ is natural since both the Young and rough interpretation of \eqref{eq:main_sde} are natural extensions of the Stratonovich interpretation, applicable when $H=1/2$, which also requires ellipticity on $\sigma$ in order to generically guarantee well-posedness. See Rem.~\ref{rem:elliptic_cond} for a further discussion of this assumption. We also refer the reader to Rem.~\ref{rem:path_by_path_uniq}, Def.~\ref{def:pathwise_uniq_young} and Def.~\ref{def:pathwise_uniq_rough} for further discussion and precise definition of \emph{pathwise} uniqueness in this context, for $H\neq 1/2$. 
At this stage, we simply note that \emph{pathwise} uniqueness under the above condition is the first major improvement on well-posedness theory for multiplicative SDEs driven by fBm since \cite{lyons94,lyons1998,davie08}, which together required $\gamma\geq 1/H$. Hence, our result significantly reduces the regularity requirement on $\sigma$ in arbitrary dimensions. In fact, our present work improves the recent result \cite[Thm.~5.2]{matsuda22}, which treats the  regime $H \in (1/2, 1)$ and assumes the technical condition $\gamma > \frac{(1-H)(2-H)}{H(3-H)}$.

When $H=1/2$ there are two natural interpretations of \eqref{eq:main_sde}, either as a stochastic, It\^o equation or Stranovich equation. These correspond to two interpretations of the deterministically, ill-posed integration of $\sigma(X_t)$ against $\dd W_t \coloneqq \dd B^{1/2}_t$. Crucially, It\^o solutions are martingales while Stratonovich integration respects the chain rule but does not produce martingale solutions. The notion of \emph{pathwise} uniqueness is equally applicable to both, however, since both interpretations produce solutions adapted to the natural filtration of the Brownian motion. In It\^o's theory it is readily seen that \emph{pathwise} uniqueness holds for \eqref{eq:main_sde} (with $H=1/2$) provided $\sigma$ is Lipschitz continuous, see for example \cite{revuz_99_continuous}. This condition provides a natural sanity check on our condition as they agree at $H=1/2$. We mention here that the result for It\^o solutions does not require any ellipticity assumption on $\sigma$, while the same statement holds for Stratonovich solutions, under the additional requirement of ellipticity on $\sigma$, see Rem.~\ref{rem:elliptic_cond}. As will be remarked below, this structural condition in the Stratonovich case can be relaxed at the price of requiring  $\sigma$ to be twice differentiable and bounded.

When $H\neq 1/2$ It\^o's theory no longer applies to \eqref{eq:main_sde} since the fBm is not a martingale process, see Sec.~\ref{subsec:fbm}. However, the Stratonovich interpretation naturally extends beyond the martingale case. Note that for $H>1/2$ no probabilistic theory  is actually required to make sense of the equation, since, under sufficient regularity assumptions on $\sigma$, one expects $t\mapsto \sigma(X_t)$ to carry the same H\"older regularity as the fBm, in this case larger than $1/2$ and so Young integration theory classically applies; see Sec.~\ref{subsec:young_integration}. In our case the regularity assumption on $\sigma$ is too weak to naively apply this theory, however, our main contribution in this regime is to carefully exploit a combination of deterministic analysis and probabilistic averaging tools to indeed construct a sufficiently regular integral, see Sec.~\ref{subsec:young_int}. For $H <1/2$ the natural extension of the Stratonovich integral is provided by \cite{lyons1998,qian_coutin_02_stochastic}. 

\paragraph*{Applications of our Results}	 Our work fits into the growing literature on the topic of \emph{regularisation by noise}. This is the phenomenon by which rapidly oscillating and typically random data lead to improved existence, uniqueness or other desirable properties of equations which would otherwise be ill-posed or behave pathologically in some manner, \cite{flandoli_11,tzvetkov_16_random,flandoli_17_random,gess_tsatsouli_20_synchro,flandoli_luo_21_high,anzeletti_le_ling_23_pathbypath}. 
In recent years there has been significant progress in the study of SDEs with irregular coefficients and driven by noise lying outside of the Brownian (or even martingale) class, for which we give a non-exhaustive list of references \cite{catelliergubinelli,butkovsky_mytnik_19_she,athreya_butkovsky_mytnik_20_strong_sde,le20,harang_perkowski_21_Cinfinity,matsuda22,gerencser2022,dareiotisgerencser22,galeatigerencser22,anzeletti_richard_tanre_23_regularisation}. More precisely, our work fits within this latter programme and considers an, as yet, less studied regime for regularisation by noise; namely the interaction between noise and irregular multiplicative coefficients.
To the best of the authors' knowledge, the only other work in this direction is \cite{hinz2022variability} in which equations with more general noise are studied but under more stringent dimensional and structural assumptions on the volatility coefficient. Although not as directly related, we also highlight the work \cite{galeati_harang_22_regularisation} which obtained regularisation by noise for SDEs with multiplicative Brownian noise under the inclusion of an additive fBm.

Equations of the type \eqref{eq:main_sde} with fractional Brownian motion are of interest in a number of applications, most notably in the study of stochastic volatility \cite{comte_renault_96_long,comte_coutin_renault_02_affine,friz_salkeld_wagenhofer_22_weak}, turbulent fluids \cite{flandoli_russo_23_reduced,crisan_holm_leahy_nilssen_22_variational,crisan_holm_leahy_nilssen_22_solution,apolinario_chevillard_mourrat_22_dynamical} and more generally stochastic processes with memory \cite{beran_94_statistics,benelli_weiss_21_sub_super,Hairer:2020aa,hairer_li_22_generating,li_panlou_seiber_23_non_stationary}.

More concretely, motivated by works such as \cite{lilly_sykulski_early_ohlede_17_fbm_matern,faranda_pons_falvio_etal_14_modelling,franzke_okane_terence_berner_williams_lucarini_14_stochastic} there has been a concerted effort in recent to study Lagrangian flows for SDEs drive by fBm with applications to intrinsic modelling of turbulence in geophysical fluids, \cite{roveri_triggiano_24_rough,crisan_holm_leahy_nilssen_22_variational,crisan_holm_leahy_nilssen_22_solution}. Focussing, for example on \cite{galeati_leahy_nilssen_25_rough_continuity}, the authors therein consider the dynamics of a typical fluid particle, modelled as a solution the rough SDE
\begin{equation*}
	\dd y_t = u_t(y_t)\dd t + \sum_{k=1}^m \xi_k(y_t) \dd Z^k_t,
\end{equation*}
where $Z$ is any path which can be lifted to a $\mfp$-variation rough path for $\mfp \in (2,3)$. Concretely, see \cite[Thm.~1.1]{galeati_leahy_nilssen_25_rough_continuity} as well as \cite[Thm.~4.3]{roveri_triggiano_24_rough} and \cite{crisan_holm_leahy_nilssen_22_variational,crisan_holm_leahy_nilssen_22_solution}. The fractional Brownian motion with $H\in (\nicefrac{1}{3},1)$ serves as a particularly tractable and well studied example in this case. Since the main focus of \cite{galeati_leahy_nilssen_25_rough_continuity} is on the admissible regularity of $u$ rather than the family $\{\xi\}_{k=1}^m$, the latter are assumed to be three times differentiable and bounded, which allows for the application of standard, general results in the analysis of rough differential equations. However, the results presented herein suggest that, in the particular case of fractional Brownian noise, a similar analysis could be performed for $\xi_k \in \mcC^{\gamma}(\mbR^d;\mbR^{2d})$ uniformly elliptic and such that $\gamma$ satisfies \eqref{eq:sigma_reg_uniq} given below. Since a major problem in the study of geophysical flows is to assimilate often sparse and noisy data into robust models, it is of importance to widen the admissible parameter range for well-posedness of models.

	\subsection{Main Result}
	We state the main result of our paper below in a somewhat informal manner, referring to precise definitions in Sec.~\ref{sec:young} and Sec.~\ref{sec:rough} for the Young and rough cases respectively. We also refer the reader to Sec.~\ref{subsec:notation} for the explanation of any unfamiliar notation.
	\begin{theorem}[Pathwise Uniqueness of \eqref{eq:main_sde}]\label{thm:main}
		Let $d_1, d_2 \in \mathbb{N}$, $B^H$ a $d_2$-dimensional fractional Brownian motion with Hurst parameter $H \in (1/3, 1)$ and $\gamma \in \RR$ satisfy 
		\begin{equation}\label{eq:sigma_reg_uniq}
			\gamma > \max\left\{\frac{1}{2H}, \frac{1-H}{H} \right\} = 
			\begin{cases}
				\frac{1}{2H}, \quad & \text{if } H > 1/2, \\
				\frac{1-H}{H}, &  \text{if } H \leq 1/2.
			\end{cases}
		\end{equation}
		Then, given a uniformly elliptic, $\sigma\in \mcC^\gamma(\mbR^{d_1};\mbR^{d_1\times d_2})$ (i.e. $\sigma \sigma^\top $ is uniformly, positive definite, see \eqref{eq:pos_def_matrix})
		interpreting the SDE \eqref{eq:main_sde} as Young differential equation (resp. rough differential equation with respect to the canonical, geometric, rough path lift of $B^H$) if $H \in (1/2,1)$ (resp. $H\in (1/3,1/2]$), pathwise uniqueness in the sense of Def.~\ref{def:pathwise_uniq_young} (resp. Def.~\ref{def:pathwise_uniq_rough}) of solutions holds for \eqref{eq:main_sde}.
	\end{theorem}
	\begin{proof}
		See Thm.~\ref{thm:pathwise_uniqueness_young} for the case $H \in (1/2,1)$ and Thm.~\ref{thm:pathwise_uniqueness_rough} 
		for $H \in (1/3,1/2]$.
	\end{proof}
	The parameter range described by \eqref{eq:sigma_reg_uniq} is explicitly visualised in Figure~\ref{fig:graph_of_H} below, along with the threshold of $\gamma\geq \frac{1}{H}$, which is higher than ours in both cases, as required by the standard Young or rough path approaches, \cite{lyons94,lyons1998,davie08}.
	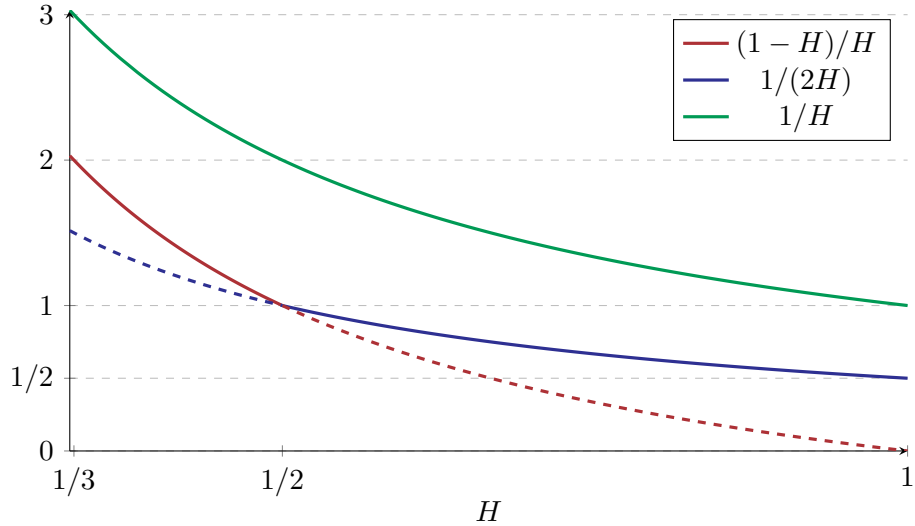
\begin{figure}[b!]
		\centering
		\begin{tikzpicture}
			\begin{axis}[
				height = 0.25\paperheight,
				width = 0.6\paperwidth,
				axis lines = left,
				xlabel = \(H\),
				xtick={1/3, 1/2, 1},
				xticklabels={$1/3$, $1/2$, $1$},
				ytick={0, 0.5, 1, 2, 3},
				yticklabels={$0$, $1/2$, $1$, $2$, $3$},
				ymajorgrids=true,
				grid style=dashed,
				every axis plot/.append style={very thick}
				]
				\addplot [
				domain=0.33:0.5, 
				samples=100, 
				color=Maroon,
				]
				{(1-x)/x};
				\addlegendentry{\((1-H)/H\)}
				\addplot [
				domain=0.5:1, 
				samples=100, 
				color=Blue,
				]
				{1/(2*x)};
				\addlegendentry{\(1/(2H)\)}
				\addplot [
				domain=0.33:1, 
				samples=100, 
				color=ForestGreen,
				]
				{1/x};
				\addlegendentry{\(1/H\)}
				\addplot [
				dashed,
				domain=0.5:1, 
				samples=100, 
				color=Maroon,
				]
				{(1-x)/x};
				\addplot [
				dashed,
				domain=0.33:0.5, 
				samples=100, 
				color=Blue,
				]
				{1/(2*x)};
			\end{axis}
		\end{tikzpicture}
		\caption{Plots of admissible parameter regimes for $\gamma$-H\"older continuous volatility to ensure pathwise uniqueness of \eqref{eq:main_sde}. Pathwise theory of \cite{lyons94,lyons1998,davie08} requires 
			$\gamma \geq 1/H$ (solid green line). Thm~\ref{thm:main} allows for $\gamma > \frac{1}{2H}\vee \frac{1-H}{H}$ (blue and red solid line).}
		\label{fig:graph_of_H}
	\end{figure}
	
	We collect some remarks and ancillary results related to Theorem~\ref{thm:main}. Firstly, as is the case for SDEs driven by Brownian motion, we have the following Yamada--Watanabe theorem, 
	\begin{proposition}\label{prop:yamada_watanabe}
		If weak existence, in the Young or rough senses (see Def.~\ref{def:weak_young_sol} or Def.~\ref{def:weak_rough_sol}) and pathwise uniqueness holds (Def.~\ref{def:pathwise_uniq_young} or \ref{def:pathwise_uniq_rough}) for \eqref{eq:main_sde}, then weak uniqueness and strong existence, in the appropriate Young or rough senses, hold as well.
	\end{proposition}
	\begin{proof}
		This follows from a genelarised Yamada-Watanabe theorem of Kurtz~\cite{kurtz07}.
		One can consult \cite[Appendix~B]{matsuda22} for more detailed discussion.
	\end{proof}
	Note that a standard compactness argument (see e.g. \cite[App.~B]{matsuda22}) 
	allows one to construct weak solutions to \eqref{eq:main_sde} (in both the Young and rough senses) for $\sigma \in \mathcal{C}^{\gamma}(\mbR^{d_1};\mbR^{d_1\times d_2})$ with $\gamma > (1-H)/H$, without any assumption of ellipticity. Hence, due to Prop.~\ref{prop:yamada_watanabe} weak uniqueness and strong existence for \eqref{eq:main_sde}, under the additional ellipticity assumption follow under the assumptions of Thm.~\ref{thm:main}. See Rem.~\ref{rem:weak_young_existence_compactness} and Rem.~\ref{rem:weak_rough_existence_compactness} for further discussion of this fact.
	\begin{remark}[Ellipticity Assumption]\label{rem:elliptic_cond}
		In Itô's theory, applicable in the case $H=1/2$, if $\sigma$ is at least Lipschitz continuous, then no ellipticity is required in order to obtain existence and uniqueness of solutions. However, if one treats the SDE \eqref{eq:main_sde} as a Stratonovich equation, the elliptic condition becomes necessary to obtain existence and uniqueness under the same regularity assumption. Alternatively, one can remove ellipticity at the price of increased regularity on $\sigma$, this corresponds to the $\gamma>\frac{1}{H}$ threshold discussed earlier. Since our treatment, beyond the Brownian case, relies on either Young or rough integration both of which generalise the Stratonovich integral and aims for improved regularity assumptions,  we also rely on an ellipticity assumption for $\sigma$.
		
		For simplicity of presentation, in Thm.~\ref{thm:main} (similarly  Thm.~\ref{thm:pathwise_uniqueness_young} and Thm.~\ref{thm:pathwise_uniqueness_rough}) this is stated as uniform ellipticity and our proofs only explicitly treat this case. However, a simple stopping time argument easily allows one to relax this assumption to obtain well-posedness on any fixed time interval to simple ellipticity. We refer to the proof of \cite[Thm.~5.2]{matsuda22} for an exposition of this idea in a similar setting. Since our main focus is the regularity of $\sigma$ we do not repeat the argument in our paper.
	\end{remark}
	\begin{remark}[Rougher and Smoother Noise Regimes]\label{rem:H_less_than_1/3}
		It is natural to expect that Thm~\ref{thm:main} can be extended to the case $H \in (1/4, 1/3]$. However, since this would require the treatment of \third  order iterated integrals (as opposed to \second order as required in Sec.~\ref{sec:rough}) it would significantly increase the length  technicality of our proofs. Therefore, in the interests of keeping the paper somewhat concise we do not present this further extension. It is worth noting that $1/4$ is likely to be a significant barrier, since even the canonical lift of an fBm to a rough path is not defined below this threshold, see \cite{qian_coutin_02_stochastic}.   
		
		Similarly, one could consider the extended Young regime, i.e. $H\in [1,\infty) \setminus \mbN$ (see for example \cite{dareiotisgerencser22,galeatigerencser22}). Again, in order to clarify presentation and avoid additional technicalities we have chosen not to explicitly consider this case here. However, we do not expect significant challenge, in this regime the fBm is classically differentiable so that some arguments may even simplify. We conjecture that in this regime, $H>1$, the threshold, coming from our tools, would become $\gamma>1-\frac{1}{2H}$.
	\end{remark}
	\begin{remark}\label{rem:existence_vs_uniqueness}
		The condition of $\gamma$-H\"older regularity for $\sigma$ with $\gamma > \frac{1-H}{H}$ is necessary to interpret 
		the SDE \eqref{eq:main_sde} as a Young differential equation in the case $H > 1/2$ or as rough differential equation when $H \leq 1/2$. 
		This condition is sufficient for weak existence when $H \in (1/3, 1)$, and also for pathwise uniqueness when $H \in (1/3, 1/2]$. 
		When $H > 1/2$, there is a gap between the condition of weak existence and that of pathwise uniqueness. 
		We conjecture that pathwise uniqueness does not hold in general for $\gamma < \frac{1}{2H}$, although
		we are not currently able to establish a proof.
		
		Due to comparison with Brownian case, we expect that pathwise uniqueness also holds for $\gamma = \frac{1}{2H}$ when $H > 1/2$, however, 
		our current methods cannot cover this case either.
	\end{remark}
	\begin{remark}[Time Dependent $\sigma$]
		As in \cite{galeatigerencser22}, one could naturally consider time dependent $\sigma$, for which we expect that the shifted stochastic sewing lemma (Lem.~\ref{lem:ssl}) with controls would be necessary.
	\end{remark}
	\begin{remark}[\emph{Path-by-Path} and \emph{Weak} Uniqueness]\label{rem:path_by_path_uniq}
		Despite its name the notion of \emph{pathwise} uniqueness is a purely probabilistic one, which in this work we combine with truly pathwise notions of existence (see Def.~\ref{def:pathwise_uniq_young} and Def.~\ref{def:pathwise_uniq_rough}). A stronger and more genuinely pathwise notion of uniqueness was coined by Davie, \cite{davie08}, as \emph{path-by-path} uniqueness. In our case, this property would hold if, on a given probability space $(\Omega,\mcF,\mbP)$ carrying an fBm $B^H$, there exists a $\mbP$-null set in $\Omega$, which may depend on $\sigma$ and the initial data such that the Young differential equation (resp. rough differential equation) for $H > 1/2$ (resp. $H \leq 1/2$),
		\begin{align*}
			\mathrm{d} y_t = \sigma (y_t) \, \mathrm{d} B_t^H(\omega), \quad y_0 = x \in \mbR^{d_1}
		\end{align*}
		has a unique solution. In particular, these solutions are not required to be adapted to any given filtration. We refer the reader to  \cite{shaposhnikov_wresch_22_path} and the references therein for a more thorough discussion of the distinction between pathwise and path-by-path uniqueness.
		
		It is not currently clear to the authors what the natural assumptions on $\sigma$ are to ensure \emph{path-by-path} uniqueness for solutions to \eqref{eq:main_sde}. In \cite[Thm.~4.8]{davie08} an example of \eqref{eq:main_sde} is studied for $H=1/2$ with $\sigma$ that is $2-\epsilon$-H\"older continuous, for any $0<\epsilon\ll 1 $ which has infinitely many solutions. However, this is not a counter example to \emph{path-by-path} uniqueness since the $\sigma$ here is in fact allowed to be random (almost surely of the required H\"older continuity) and may depend on the realisation of the Brownian motion.
		
		Similarly, it is not clear to the authors which assumptions should be required of $\sigma$ in order to prove weak uniqueness of solutions to \eqref{eq:main_sde} (see Def.~\ref{def:weak_young_sol} and Def.~\ref{def:weak_rough_sol}).
	\end{remark}
	\begin{remark}[SDEs with Drift]\label{rem:with_drift}
		A natural question following from Thm.~\ref{thm:main} is to obtain similar results for fBm driven SDEs with drift and multiplicative noise, of the form
		\begin{align}\label{eq:drift_mult_sde}
			\mathrm{d} X_t = b(X_t) \, \mathrm{d} t + \sigma (X_t) \, \mathrm{d} B^H_t.
		\end{align}
		If $b$ is sufficiently regular, for example Lipschitz continuous, it seems clear that the arguments of the current paper should adapt without difficulty to prove pathwise uniqueness, under the assumptions of Thm.~\ref{thm:main}. 
		Conversely, in \cite{dareiotisgerencser22} the authors consider \eqref{eq:drift_mult_sde} with regular $\sigma$ but irregular drift term, $b$.
		Building on the current work, it would be of interest to study the case when both $\sigma$ and $b$ are irregular. See also the more recent additions to this line of research \cite{dareiotis_gerencser_le_ling_24_regularisation_gaussian,catellier2024regularizationnoiseroughdifferential}. The analogous situation has been studied in the Brownian case, e.g. \cite{krylov_23_strong}.
	\end{remark}
	\begin{remark}
		It seems reasonable to expect that using similar arguments as ours (see Sec.~\ref{subsec:proof_overview} below) one could hope to show the existence of a continuous, adapted, semi-flow $\Phi : [0,T]^{2}_{\leq} \times \mbR^{d_1}\to \mbR^{d_1}$ defined such that
		\begin{equation}
			\Phi_{s,t}(x) \coloneqq X^{s;x}_t = x + \int_s^t \sigma(X^{s;x}_r)\dd B^H_r.
		\end{equation}
		Continuity here means weak convergence (as stochastic processes) of $\Phi_{s,\,\cdot\,}(x)- \Phi_{s,\,\cdot\,}(y)$ to zero as $x\to y$. One challenge here would be to obtain uniqueness of solutions for all initial data simultaneously, while relevant null sets in our arguments may (in principle) depend on the initial data $x$. A common resolution to this issue is to first obtain stability of the solution in the initial data, see for example \cite{Cass:2013aa,galeatigerencser22}. It is not currently clear if our tools allow us to obtain similar stability estimates. This question is also related to stability of the solution in the coefficient $\sigma$, which in turn would directly allow one to consider McKean--Vlasov SDEs of the form,
		\begin{equation}
			\dd X_t = (\sigma\ast \mu_t)(X_t) \dd B^H_t,\quad \mu_t = \text{Law}(X_t),
		\end{equation}
		under similar conditions on $\sigma$ as Thm.~\ref{thm:main}. See for example \cite{galeati_harang_mayorcas_22_distribution_fbm,galeatigerencser22} for treatments of McKean--Vlasov equations with additive noise and singular drifts as well as \cite{bailleul_catellier_delarue_20_solving,delarue_salkeld_23_example,delarue_salkeld_22_prob_rp_II} for treatments of multiplicative, rough, McKean--Vlasov equations in the case of smooth volatilities. Stability estimates with respect to $\sigma$ are also related to convergence rates of numerical schemes for these and similar equations, see \cite{friz_salkeld_wagenhofer_22_weak,leling21,galeati_ling_23_stability}.
	\end{remark}
	\subsection{Overview of the Proof}\label{subsec:proof_overview}
	We give a brief and formal description of our proof strategy for Thm.~\ref{thm:main}. In both cases for $H\in (1/3,1/2]$ and $H\in (1/2,1)$ our argument is derived from a strategy given in \cite{le20}, which relates pathwise uniqueness to construction of an irregular stochastic integral. In our case, let us consider two (weak) solutions $X,\,Y$ to \eqref{eq:main_sde} defined on the same, filtered, probability space $(\Omega,\mcF,(\mcF_t)_{t\geq 0},\mbP)$ carrying an fBm $B^H$ with Hurst parameter $H\in (1/4,1)$. We refer the reader to Def.~\ref{def:pathwise_uniq_young} and Def.~\ref{def:pathwise_uniq_rough} for the precise notions of pathwise uniqueness that we work with. Then, letting $d_1,\,d_2 \in \mbN$ and assuming $\sigma:\mbR^{d_1}\to \mbR^{d_1\times d_2}$ is smooth for now, the difference $X-Y$ satisfies,
	\begin{equation}
		X_t-Y_t = \int_0^t \left(\sigma(X_r)-\sigma(Y_r) \right)\dd B_r^H.
	\end{equation}
	Applying the mean value theorem we see that $\sigma(X_r)-\sigma(Y_r)$ is a $d_1\times d_2$ matrix whose components are given by
	\begin{equation*}
		\sigma_{ij}(X_r)-\sigma_{ij}(Y_r) = (X_r-Y_r) \cdot \int_0^1 \nabla \sigma_{ij} (\theta X_r +(1-\theta)Y_r) \dd \theta,\quad i =1,\ldots d_1,\,j=1,\ldots,d_2.
	\end{equation*}
	Thus, we can formally write,
	\begin{equation}\label{eq:diff_linearisation}
		X_t-Y_t = \sum_{k=1}^{d_1} \int_0^t (X_r^k-Y_r^k) \dd G^k_r,
	\end{equation}
	where the components $(G^k)_{k=1}^{d_1}$ are defined by setting
	\begin{equation}\label{eq:g_lin_def}
		G^k_t \coloneqq \int_0^t \int_0^1 \partial_k \sigma(\theta X_r + (1-\theta)Y_r) \dd \theta \dd B^H_r.
	\end{equation}
	Here the product between $\partial_k \sigma$ and $\dd B^H$ is understood in the usual manner as a $d_1\times d_2$ vector acting on a vector of length $d_2$. Thus, we aim to view the difference $X-Y$ as a solution to the linear, integral equation
	\begin{equation}\label{eq:z_eq_linear}
		Z_t = \int_0^t Z_t \cdot \dd G_t.
	\end{equation}
	Since \eqref{eq:z_eq_linear} clearly has unique solutions (if they exist), the question of pathwise uniqueness of solutions to \eqref{eq:main_sde} boils down to rigorous construction of the integral on the right hand side of \eqref{eq:diff_linearisation}. The challenge is to do so for irregular $\sigma$ and this requires careful use of a priori properties which we obtain for solutions to \eqref{eq:main_sde}. The key result we make use of is a stochastic averaging lemma for stochastic integrals of singular integrands evaluated along a rapidly oscillating component of the fBm path, see Lem.~\ref{lem:gaussian_integral}. This highlights why our assumption of ellipticity of $\sigma$ is crucial since it ensures that the weak solutions can be robustly expanded in terms of the noise process. Beyond this key tool, the technical steps of our proof differ in the regimes $H\in (1/2,1)$ and $H\in (1/3,1/2]$, however, the spirit is the same in both cases.
	\subsubsection*{Young case: $H\in (1/2,1)$}
	In this regime we expect to have solutions $X,\,Y$ which are $\alpha$-H\"older continuous for any $\alpha \in (1/2,H)$ and in turn we expect the map $t\mapsto G_t^k$ (as defined by \eqref{eq:g_lin_def}) to be at least $1/2$-H\"older continuous. Thus we expect to be able to make sense of the integral on the right hand side of \eqref{eq:diff_linearisation} as a Young integral, see Sec.~\ref{subsec:young_int}  below. However, the limiting factor is the regularity of the composition $\partial_k \sigma(\theta X_r +(1-\theta)Y_r)$ inside $G$. If we allow $\sigma$ to be $\gamma$-H\"older continuous then one expects $r\mapsto \partial_k \sigma(\theta X_r +(1-\theta)Y_r)$ to be $\alpha(\gamma-1)$-H\"older continuous. Thus, in order to naively define $G$ we would require $\alpha (\gamma-1)+\alpha >1 \,\, \iff\,\, \gamma > \frac{1}{\alpha}>\frac{1}{H}$, (recall this is the condition of \cite{lyons94,davie08}). To go below this threshold we make more detailed use of the fact that one expects $\theta X +(1-\theta)Y$ to be \emph{controlled} in a suitable probabilistic sense by $B^H$ (see Def.~\ref{def:control_prob}). 
	This allows us to exploit averaging properties of the fBm in order to obtain sharper stochastic estimates on the composition $\partial_k \sigma(\theta X_r +(1-\theta)Y_r)$. These estimates are obtained by Prop.~\ref{prop:fbm_young} giving control (essentially) of the form
	\begin{equation}\label{eq:intro_young_fbm}
		\Big\| \int_s^t f (X_r) \mathrm{d} B_r^H \Big\|_{L^p_{\omega}} 
		\lesssim_{\beta, l, p, \gamma, X,T} \| f \|_{\mathcal{C}_x^{\gamma-1}} \big((t - s)^{\gamma H}
		+  (t-s)^{\left(1+\frac{1}{2H} +(\gamma-1)\right)H}\big),
	\end{equation}
	for $f:\mbR^{d_1}\to \mbR^{d_1\times d_2}$ that is at least $(\gamma-1)$-H\"older continuous and any $X$ \emph{probabilistically controlled} (see Def.~\ref{def:control_prob}) by $B^H$ with remainder that is at least $1/2$-H\"older continuous. In order to apply the sewing lemma (see Lem.~\ref{lem:det_sewing})   we therefore require $\gamma H \wedge \left(\frac{1}{2H}+\gamma\right)H \geq \frac{1}{2} \iff \gamma >\frac{1}{2H}$. The proof of this estimate is carried out by considering the germ,
	\begin{equation*}
		A_{s, t} \defby \mathbb{E}\left[ \int_s^t f(X_{s-(t-s)} + X'_{s-(t-s)} B_{s-(t-s), r}^H) \, \mathrm{d} B^H_r 
		\Big \vert \mathcal{F}_{s - (t-s)} \right],
	\end{equation*}
	where $X'$ denotes the Gubinelli derivative of $X$ and applying Gerencsér's shifted, stochastic, sewing lemma \cite[Lem.~2.2]{gerencser2022} (see Sec.~\ref{subsec:sewing}). The significant challenge here is to treat the stochastic integral inside the germ $A$, this is handled by somewhat involved Gaussian computations, carried out in App.~\ref{sec:gaussian}, which result in the stochastic averaging given  by Lem.~\ref{lem:gaussian_integral, crucial to our approach}. The remainder of Sec.~\ref{sec:young} essentially consists of applying this abstract control to construct $G$ as defined by \eqref{eq:g_lin_def} by considering a smooth approximation $\sigma_n\to \sigma$ converging in any $\gamma'$-H\"older topology for $\frac{1}{2H}<\gamma'<\gamma$, see the proof of Thm.~\ref{thm:pathwise_uniqueness_young} in Sec.~\ref{subsec:pw_uniq}. Prior to this final step we also show that all weak solutions to \eqref{eq:main_sde} are \emph{probabilistically controlled} by $B^H$ in a suitable sense, see Lem.~\ref{lem:sol_is_controlled}, which allows us to apply Prop.~\ref{prop:fbm_young}.
	\subsubsection*{Rough case: $H\in (1/3,1/2]$}
	In the rough case, our strategy is essentially the same, in that we reduce our problem to constructing the integral \eqref{eq:g_lin_def} and solving the linearised equation \eqref{eq:diff_linearisation}. However, in this case we can no longer appeal to Young integration theory for either task. Concerning the second, since we cannot expect $X,\,Y$ to be more regular than the noise, i.e. $\alpha$-H\"older continuous for any $\alpha \in (1/3,H)$, even if we assume $\sigma$ to be smooth, we cannot expect $r \mapsto G_r$ (with $G$ as in \eqref{eq:g_lin_def}) to be more regular than the noise either and so the integral of $X-Y$ against $\dd G$ is not defined as a Young integral since we no longer have $2\alpha>1$. This difficult is relatively easily circumvented by instead appealing to rough integration. Assuming for now that we can construct $G$, the natural solution is to define the joint lift of $(B^H,G)$ to a rough path 
	\begin{equation}\label{eq:B_G_lift}
		\left(B^H, G^k, \int B^H \otimes \, \mathrm{d} B^H, \int B^H \otimes \, \mathrm{d} G^k, \int G^k \otimes \, \mathrm{d} B^H, \int G^k \otimes \, \mathrm{d} G^l\right)_{k,\,l\, =1}^{d_1}.
	\end{equation} 
	Since $X,\,Y$ can both be shown to be \emph{controlled} by $B^H$ (in the sense of Gubinelli's controlled rough paths, see \cite{gubinelli04} and Def.~\ref{def:controlled_path}) the integration of $X-Y$ against this joint rough path follows from standard results. The construction of the joint lift for smooth $\sigma$ is also relatively standard, the central challenge is to obtain this construction for $\sigma$ satisfying the assumptions of Thm.~\ref{thm:main}.
	
	It turns out that this construction relies on the same estimates that allow us to construct $G$ in the first place. As in the Young case the core of the proof is to obtain precise, probabilistic estimates on integrals of singular functions against fBm, in this case these are of the form
	\begin{equation*}
		\int_{s}^t f(X_r)\dd B^H_r,\quad \int_s^t f(X_r) B^H_{s,r}\otimes \dd B^H_r\quad \text{and}\quad \int_s^t \int_{s}^{r_2} g(Y_{r_1}) \dd B^{H}_{r_1}\otimes f(X_{r_2})\dd B^H_{r_2}.
	\end{equation*}
	As before, we take $f,\,g$ to be $\gamma-1$-H\"older continuous, playing the role of $\nabla \sigma$, this time, however, we assume $X,\,Y$ are deterministically \emph{controlled} by $B^H$, see Def.~\ref{def:controlled_path}. Control on these integrals, similar in spirit to \eqref{eq:intro_young_fbm} are obtained progressively through Prop.~\ref{prop:fbm_rough}, Prop.~\ref{prop:rough_iterated_integral} and Prop.~\ref{prop:rough_itr_int_2}. These estimates are also more technical than in the Young case and we are required to apply the stochastic sewing lemma twice, first in the fully shifted sense of \cite{matsuda22} and then in the original form given by \cite{le20}. We refer the reader to the proofs of the above propositions for details, simply remarking here that the main innovation is to use the knowledge that $X$ and $Y$ are controlled by $B^H$ to again obtain improved regularising effects through averaging.
	
	Once sufficient control on the above integrals is obtained we are again able to pass along a smooth approximation $\sigma_n$ to $\sigma$, converging in any $\gamma'$-H\"older topology for $\frac{1-H}{H}<\gamma' <\gamma$. As before, we are then in position to appeal to uniqueness of solutions to the linear equation \eqref{eq:z_eq_linear} and conclude.
	\subsection{Notation}\label{subsec:notation}
	Throughout the paper we use the following notations and conventions.
	\begin{itemize}
		\item When stating an inequality we either write $A \lesssim _{\alpha,\,\beta,\,\ldots} B$ to indicate that the inequality holds up to a constant depending on the parameters $\alpha,\,\beta,\,\ldots$ or write that there exists a $C\coloneqq C(\alpha,\beta,\ldots)>0$ such that $A\leq CB$. If we write $A\lesssim B$ without explicit dependents we mean that the inequality holds up to a constant depending on unimportant parameters that we do not keep track of.
		\item We  $\NN \coloneqq \{1,\,2,\,\ldots\}$ for the strictly positive integers and set $\NN_0 \coloneqq \NN \cup \{0\}$. As usual we use $\RR^d$, for $d\in \mbN$, to denote Cartesian products of the real numbers. 
		Given a positive integer $d$, we write $l \in \NN^d$ (resp. $l\in \NN_0^d$) to denote a multi-index made up of positive (resp. non-negative) integers of length $d$ and set $|l| \coloneqq \sum_{i=1}^d l_i$. As is standard, we write $\partial_x$ (resp. $\partial_t$) for the derivatives on $\RR^d$ (resp. $[0,\infty)$) and given $l \in \NN_0^{d}$ we let $\partial_x^l \coloneqq (\partial^{l_1}_{x_1},\ldots,\partial^{l_d}_{x_d})$. Where it will not use confusion we sometimes write $\dot{f}(t,x)$ for the derivative in time of a function on $\mbR_+\times \mbR^d$.
		\item We fix two positive integers $d_1,\,d_2 \in \NN$ which will always denote, respectively, the dimension of the solution process $t\mapsto X_t\in \RR^{d_1}$ and the dimension of the noise $t\mapsto B^H_t \in \RR^{d_2}$. Note that all results hold even in the case $d_2<d_1$.
		\item Given $S,\, T\in \RR$ such that $S<T$, $d\in \NN$ and a function $f:[S,T]\to \RR^d$, we define the increment,
		$f_{s,t} \coloneqq f_t-f_s$.
		Furthermore, defining the sets
		\begin{align*}
			[S,T]^n_{\leq} &:= \{ (s_1,\ldots,s_n) \in [S,T]^n \,:\, s_1\leq s_2\leq \cdots \leq  s_n\}, \quad \text{for}\,\, n\in \mbN
		\end{align*}
		we specify a map $\delta : C([S,T]^2_{\leq};\RR^d)\rightarrow C([S,T]^3_{\leq};\RR^d)$ such that
		$\delta A_{s,u,t} = A_{s,t}-A_{s,u}-A_{u,t}.$
		\item Given, $d_1,\,d_2 \in \mbN$, we write $\cM(d_1,d_2)$ for the space of $d_1 \times d_2$ matrices and for $d_1=d_2\eqqcolon d$ we define the set of positive definite, square matrices,
		\begin{equation}\label{eq:pos_def_matrix}
			\cM_+(d) \coloneqq \left\{ \Gamma \in \cM(d,d) \,:\, z^{\top} \Gamma z >0,\, \text{for all } z \in \RR^{d} \setminus \{0\}\right\}.
		\end{equation}
		Letting $I_{d} \in \cM(d,d)$ denote the identity matrix, given $\Gamma \in \cM(d,d)$ and $c\in \RR$ we say $\Gamma> c$
		if and only if $ \Gamma - cI_{d} \in \cM_+(d)$ and $\Gamma< c$ if and only if $cI_{d}-\Gamma\in \cM_+(d)$.
		By identifying $\cM(d_1,d_2)$ with $\R^{d_1\times d_2}$, we equip $\cM(d_1,d_2)$ with the usual Euclidean norm 
		$\abs{\,\cdot\,}$. Finally, given $\Gamma \in \cM(d_1,d_2)$ we use $\Gamma^{\top}\in \mcM(d_2,d_1)$ to denote the transposed matrix.
		\item Given $\Gamma \in \cM_+(d)$ we define the generalised heat kernel $p_\Gamma :\RR^{d}\to \RR$,
		\begin{equation*}
			p_{\Gamma}(x) \defby \frac{1}{(2 \pi)^{d/2} (\det \Gamma)^{1/2}} \exp \Big(-\frac{\inp{x}{\Gamma^{-1} x}}{2}\Big)
		\end{equation*}
		and for a function $f:\RR^{d} \to \RR$ we set $\cP_{\Gamma} f \defby p_{\Gamma} \conv f$.
		When $f$ is vector-valued, $\mathcal{P}_{\Gamma}$ acts component-wise.
		If $\Gamma = t I_{d}$ for some $t \in (0, \infty)$ we simply write
		$\cP_t \defby \cP_{t I_{d}}$.
		\item Given $k \in \NN_0$ we write $C^k(\RR^{d_1};\RR^{d_2})$ (resp. $C^k(\RR^{d_1};\cM(d_1,d_2))$) for the space of $k$-times, continuously differentiable mappings from $\RR^{d_1}$ to $\RR^{d_2}$ (resp. $\RR^{d_1}$ to $\cM(d_1,d_2)$). From now on, we typically only specify the domain of a given function space, leaving the range unstated where it is clear from context, or does not affect the statement.  For $\gamma \in \RR\setminus \NN_0$ we let $\cC^{\gamma}(\RR^{d_1})$ denote the set of functions finite under the norm,
		\begin{equation*}
			\norm{f}_{\cC^{\gamma}_x} \defby \sum_{l \in \N_0^{d_1},\, \abs{l} < \floor{\gamma}} \norm{\partial^l_x f}_{L^{\infty}_x}
			+ \sum_{l \in \N_0^{d_1},\, \abs{l}= \lfloor \gamma \rfloor} \sup_{x\,\neq y\, \in \R^d} \frac{\abs{\partial^l_x f(x) - \partial^l_x f(y)}}{\abs{x- y}^{\gamma - \lfloor \gamma \rfloor}}.
		\end{equation*}
		Similarly, given $0\leq S\leq T<\infty$ and $d\in \NN$ we write $\cC^{\gamma}([S,T];\RR^{d})$ for the paths $f:[S,T]\to \RR^{d}$ finite under the norm,
		\begin{equation*}
			\norm{f}_{\cC^{\gamma}_{[S,T]}} \defby \sum_{l=0}^{\floor{\gamma}-1 } \norm{\partial^l_t f}_{L^{\infty}_{[S,T]}}
			+  \sup_{(s,t) \in [S,T]^2_{\leq}} \frac{\abs{\partial^l_t f(t) - \partial^l_t f(s)}}{\abs{t-s}^{\gamma - \lfloor \gamma \rfloor}}.
		\end{equation*}
		When $S=0$ we use the abbreviations $\cC^\gamma_T \coloneqq \cC^\gamma([0,T];\RR^d)$ and $\|f\|_{\cC^\gamma_T}$.
		As also introduce the family of seminorms
		\begin{align}\label{eq:holder_seminorm}
			\llbracket f \rrbracket_{\mathcal{C}^{\gamma}_{[S, T]}} \defby \sup_{(s, t) \in [S, T]^2_{\leq}} \frac{\abs{f_t - f_s}}{\abs{t - s}^{\gamma}}.
		\end{align}
		For $k\in \NN_0$ we let $\cC^{k}(\RR^{d_1};\RR^{d_1})$ (in turn $\cC^k(\RR^{d_1};\cM)$ and $\cC^k([S,T];\RR^d)$) denote the set of functions whose first $k$, weak, derivatives are bounded.
		
		For, $\gamma <0$ we define the norm
		\begin{equation}\label{eq:negative_gamma_norm}
			\|f\|_{\cC^{\gamma}_x} \coloneqq \sup_{t \in (0,1]} t^{-\frac{\gamma}{2}} \|\cP_t f \|_{L^\infty}
		\end{equation}
		and declare $\cC^{\gamma}(\RR^{d_1};\RR^{d_2})$ (resp. $\cC^{\gamma}(\RR^{d_1};\cM(d_1,d_2))$) to be the set of distributions $f \in \cS'(\RR^{d_1};\RR^{d_2})$ (resp. $f\in \cS'(\RR^{d_1};\cM(d_1,d_2))$) which are finite under the norm of \eqref{eq:negative_gamma_norm}. Note that for $\gamma \in \RR\setminus \ZZ$ the space $\cC^\gamma(\RR^{d_1})$ coincides with the standard H\"older--Besov spaces defined via Littlewood--Payley  blocks, see \cite[Sec.~2.6.4]{triebel92}.
		\item Given a probability space $(\Omega,\mcF,\mbP)$, a Banach space valued random variable $X :\Omega \to (E,\|\,\cdot\,\|_E)$ and $p\in [1,+\infty] $ we define the norms
		\begin{equation*}
			\|X\|_{L^p_\omega} \,\coloneqq \, \begin{cases}
				\mbE[\|X\|_{E}^p]^{\frac{1}{p}}, & p\in [1,+\infty),\\
				\esssup_{\omega \in \Omega} \|X(\omega)\|_E,
			\end{cases}
		\end{equation*}
		and set
		\begin{equation*}
			L^p(\Omega;E) \coloneqq \{X:\Omega \to E\,:\, \|X\|_{L^p_\omega} <\infty\}.
		\end{equation*}
		In these definitions we leave the probability measure implicit and where appropriate, the Banach space $E$.
	\end{itemize}
	\subsubsection*{Acknowledgement}
	\noindent Both authors wish to thank K. Dareiotis, L. Galeati, M. Geresc\'er and N. Perkowski for illuminating discussions during the course of this project.
	The main portion of this research was conducted when TM was a PhD student at Freie Universität Berlin with financial support by the German Science Foundation (DFG) via the IRTG 2544. 
	AM gratefully acknowledges financial support from DFG Research Unit FOR2402. 
	%
	%
	\section{Preliminaries}\label{sec:preliminaries}
	\noindent We recall some useful facts concerning the heat semi-group, Young and rough integration (stochastic) sewing as well as the fractional Brownian motion.
	\subsection{Properties of the Heat Semi-Group}
	Let $\alpha,\,\beta \in \RR$ with $\beta \geq 0 \vee \alpha$, $t > 0$, $d\in \mbN$ and $c>1$. Then the following hold;
	\begin{itemize}
		\item given $\Gamma\in \cM_+(d)$ such that $c^{-1}< \Gamma < c$
		\begin{equation}\label{eq:semigroup_estimate}
			\|\cP_{t \Gamma} \|_{\cC^\beta_x} \lesssim_{c,\alpha,\beta} (t\wedge 1)^{-\frac{\beta-\alpha}{2}} \|f\|_{\cC^\alpha_x}.
		\end{equation}
		\item given $\Gamma_1,\,\Gamma_2\in \cM_+(d)$ such $c^{-1} < \Gamma_i < c$ for $i=1,\,2$
		\begin{equation}\label{eq:semigroup_difference}
			\norm{(\cP_{t\Gamma_1} - \cP_{t\Gamma_2 }) f}_{\cC^{\beta}_x}
			\lesssim_{c, \alpha, \beta} (t\wedge 1)^{-\frac{\beta-\alpha}{2}} \abs{\Gamma_1 - \Gamma_2}
			\norm{f}_{\cC^{\alpha}_x}.
		\end{equation}
	\end{itemize}
	Inequality \eqref{eq:semigroup_estimate} is formally stated and proved as Lem.~\ref{lem:semigroup_regularisation} while \eqref{eq:semigroup_difference} is formally stated and proved as Lem.~\ref{lem:semigroup_difference}.
	\subsection{Young Integration}\label{subsec:young_integration}
	\noindent We briefly review the core aspects of \emph{Young integration} \cite{young}, which we appeal to for $H \in( 1/2,1)$.
	We refer the reader to \cite{galeati_23_nonlinear,friz_victoir_2010} for modern accounts of the theory. Let $T > 0$, $\alpha , \beta \in (0, 1)$, and let $g \in \mathcal{C}^{\beta}_T(\mathbb{R}^{d_2})$. 
	Suppose that $f \in \mathcal{C}^{\alpha}_T(\mbR^{d_1};\mcM(d_1,d_2))$. If $\alpha + \beta > 1$, then, as shown for example by \cite[Thm.~6.8]{friz_victoir_2010}, the limit
	\begin{align*}
		\int_0^T f_r \, \mathrm{d} g_r \defby \lim_{\substack{\pi \text{ is a partition of }[0, T]; \\ \abs{\pi} \to 0} }
		\sum_{[s, t] \in \pi} f_s g_{s, t}
	\end{align*}
	exists, and furthermore, one has the following estimate,
	\begin{align}\label{eq:young_fund_est}
		\abs[\Big]{\int_s^t f_r \, \mathrm{d} g_r - f_s g_{s, t}} \lesssim_{\alpha, \beta} 
		\llbracket f \rrbracket_{\mathcal{C}^{\alpha}_{[s, t]}} \llbracket g \rrbracket_{\mathcal{C}^{\beta}_{[s, t]}}
		(t-s)^{\alpha + \beta},
	\end{align}
	where the seminorms $\llbracket \,\cdot\, \rrbracket_{\mcC^{\gamma}_{[s,t]}}$ are defined by \eqref{eq:holder_seminorm} (for $\gamma \in \{\alpha,\beta\}$ here).
	\subsection{Rough Integration}\label{sec:rough_integration}
	When the regularity of the driver $g$ is below this threshold, one may still construct a suitable notion of integral, provided a suitable enhancement to the path $g$, and the additional constraint that the integrand $f$ must \emph{look like} $g$, in a suitable sense, at small scales. These notions are made precise in the following two definitions.
	\begin{definition}[Rough Paths and Geometric Rough Paths -  {\cite{lyons1998}}]\label{def:rough_path}
		Let $d\in \mbN$, $\alpha \in (1/3, 1/2]$, $g \in \mathcal{C}_T^{\alpha}(\mathbb{R}^{d})$ and $\mathbb{G} \from [0, T]_{\leq}^2 \to \mathbb{R}^{d} \otimes \mathbb{R}^{d}$. 
		The pair $(g, \mathbb{G})$ is called an $\alpha$-Hölder \emph{rough path} if 
		\begin{itemize}
			\item the pair satisfies the following H\"older regularity estimate
			\begin{align*}
				\norm{(g, \mathbb{G})}_{\alpha} \defby 
				\sup_{[s, t] \in [0, T]_{\leq}} \frac{\abs{g_t - g_s}}{\abs{t-s}^{\alpha}} + \frac{\abs{\mathbb{G}_{s, t}}}{\abs{t-s}^{2 \alpha}} < \infty,
			\end{align*}
			\item and the following algebraic relationship; for $(s, u, t) \in [0, T]_{\leq}^3$,
			\begin{align}\label{eq:chen_relation}
				\mathbb{G}_{s, t} = \mathbb{G}_{s, u} + \mathbb{G}_{u, t} + 
				(g_s-g_u) \otimes (g_t -g_u).      
			\end{align}
		\end{itemize}
		We say that a rough path is \emph{geometric} if there exists a sequence of smooth functions $(g^n)_{n=1}^{\infty}$
		such that 
		\begin{align*}
			\lim_{n \to \infty} \norm{g - g^n}_{\mathcal{C}^{\alpha}_T} = 0 \quad \text{and} \quad 
			\lim_{n \to \infty} \sup_{[s, t] \in [0, T]_{\leq}^2} \frac{\abs{\mathbb{G}_{s, t} - \mathbb{G}^n_{s, t}}}{\abs{t-s}^{2 \alpha}} = 0,
		\end{align*}
		where $\mathbb{G}^n_{s, t} \defby \int_s^t g^n_{s, r} \otimes \dot{g}^n_r \, \mathrm{d}r$. 
	\end{definition}
	\begin{definition}[Space of Controlled Paths  - \cite{gubinelli04}]\label{def:controlled_path}
		Let $d_1,\,d_2 \in \mbN$, $\alpha \in (0, 1)$, $\beta \in (\alpha \vee (1-\alpha), 1)$ and $g \in \mathcal{C}^{\alpha}_T(\mathbb{R}^{d_2})$.
		A path $f \from [0, T] \to \mathbb{R}^{d_1}$ is said to be \emph{controlled} by $g$, if 
		there exists an $f' \from [0, T] \to \mcM(d_1,d_2)$ such that 
		\begin{align}\label{eq:cont_path_norm}
			\norm{(f,f')}_{\mathscr{D}_g^{\beta}} \defby \norm{ f' }_{\mathcal{C}^{\beta - \alpha}_T}
			+ \sup_{(s, t) \in [0, T]_{\leq}^2} \frac{\abs{f_{s, t} - f'_s g_{s, t}}}{\abs{t-s}^{\beta}} < \infty.
		\end{align}
		The map $f'$ is called the \emph{Gubinelli derivative} of $f$. 
		We write $\mathscr{D}^{\beta}_g$ for the space of $\beta$-regular paths controlled by $g$. Note that for $\beta'>\beta$ one has $\msD^{\beta'}_g \subset \msD^{\beta}_g$.
	\end{definition}
	The definition of the space of controlled rough paths allows us to give meaning to the rough integral against a given path. For $\alpha \in (1/3, 1)$ and $\beta \in (1-\alpha, 1)$ let $(g, \mathbb{G})$ be an $\alpha$-Hölder rough path and $f \in \mathscr{D}^{\beta}_g$. Then, referring for example to \cite[Thm.~4.10]{FH20}, one sees that the limit 
	\begin{equation*}
		\int_0^T f_r \, \mathrm{d} g_r \defby \lim_{\substack{\pi \text{ is a partition of }[0, T]; \\ \abs{\pi} \to 0} }
		\sum_{[s, t] \in \pi} f_s g_{s, t} + f'_s \mathbb{G}_{s, t}
	\end{equation*}
	exists and defines the rough integral of $f$ against $g$. Moreover, similarly to \eqref{eq:young_fund_est}, we have 
	\begin{equation}\label{eq:rough_fund_est}
		\abs[\Big]{\int_s^t f_r \, \mathrm{d} g_r - f_s g_{s, t} - f'_s \mathbb{G}_{s, t}} \lesssim_{\alpha, \beta} 
		\norm{(f, f')}_{\mathscr{D}_g^{\beta}} \norm{(g, \mathbb{G})}_{\alpha} 
		(t-s)^{\alpha + \beta} .
	\end{equation}
	\subsection{Sewing}\label{subsec:sewing}
	\noindent We require a number of versions of the sewing lemma, first introduced by \cite{gubinelli04} and \cite{feyel06}. 
	\begin{lemma}[Sewing Lemma -  \cite{gubinelli04,feyel06}]\label{lem:det_sewing}
		Let $E$ be a Banach space, $T > 0$ and \mbox{$A: [0, T]_{\leq}^2 \to E$} be an $E$-valued, two parameter process. Then, suppose that for some $\varepsilon, C \in (0, \infty)$ it holds that 
		\begin{align*}
			\norm{\delta A_{s, u, t}}_E \leq C \abs{t-s}^{1+\varepsilon}, \quad \forall \,(s, u, t) \in [0, T]^3_{\leq},
		\end{align*}
		then there exists a unique map $\mathcal{A} \from [0, T] \to E$ such that 
		\begin{align}\label{eq:cA_uniq}
			\norm{\mathcal{A}_t - \mathcal{A}_s - A_{s, t}}_E \lesssim_{\varepsilon} C \abs{t-s}^{1 + \varepsilon}, 
			\quad \forall \,(s, t) \in [0, T]_{\leq}^2.
		\end{align}
		Furthermore, it holds that
		\begin{align*}
			\mathcal{A}_t - \mathcal{A}_s = 
			\lim_{\substack{\pi \text{ is a partition of }[s, t], \\ \abs{\pi} \to 0}} 
			\sum_{[s', t'] \in \pi} A_{s', t'}.
		\end{align*}
	\end{lemma}
	A proof of Lem.~\ref{lem:det_sewing} can be found in \cite{gubinelli04} as well as the proof of \cite[Lem.~4.2]{FH20}.

	\noindent As is now standard in the field of regularisation by noise, 
	we will need the stochastic version of Lem.~\ref{lem:det_sewing}, obtained first 
	by Lê \cite[Thm.~2.1]{le20}. For our purpose, we will require the following \emph{shifted} version.
	\begin{lemma}[Shifted Stochastic Sewing Lemma - {\cite[Thm.~1.1]{matsuda22}}]\label{lem:ssl}
		Let $T \in (0, \infty)$, $(\Omega,\mcF,(\mcF_t)_{t\in [0,T]},\mbP)$ be a given filtered probability space and $(A_{s, t})_{(s, t) \in [0, T]_{\leq}^2}$ be 
		an $\mathbb{R}^d$-valued two-parameter stochastic process such that for all $(s,t)\in [0,T]^2_\leq$
		$A_{s, t}$ is $\cF_t$-measurable and there exist constants $\Gamma_1,\,\Gamma_2 \geq 0$, $\epsilon_1,\,\epsilon_2 \in (0,1)$ and $\alpha \in [0,1/2+\epsilon_2)$ such that for all $(v,s,u,t)\in [0,T]^4_\leq$ (with $t-s<s-v$) and $p\in [2,\infty)$, it holds that
		\begin{equation}\label{eq:ssl_assump}
			\begin{split}
				\norm{\delta A_{s, u, t}}_{L^p_{\omega}} &\leq \Gamma_1 
				(t-s)^{\frac{1}{2} + \epsilon_1}, \\
				\norm{\mathbb{E}[\delta A_{s, u, t} \vert \mathcal{F}_{v}]}_{L^p_{\omega}} 
				&\leq \Gamma_2 (s-v)^{-\alpha} (t-s)^{1 + \epsilon_2}.
			\end{split}
		\end{equation}
		Then, there exists a unique (up to modification) stochastic process $\mathcal{A}$ 
		such that $\mathcal{A}_0 = 0$ and 
		\begin{align*}
			\norm{\mathcal{A}_t - \mathcal{A}_s - A_{s, t}}_{L^p_{\omega}} &\lesssim_{d, p, \epsilon_1, \epsilon_2, \alpha} 
			\Gamma_1  (t-s)^{\frac{1}{2} + \epsilon_1} 
			+ \Gamma_2 (t-s)^{1 + \epsilon_2 - \alpha}, \\
			\norm{\mathbb{E}[\mathcal{A}_{t} - \mathcal{A}_s - A_{s, t} \vert 
				\mathcal{F}_{v}]}_{L^p_{\omega}} 
			&\lesssim_{d, p, \epsilon_1, \epsilon_2, \alpha} \Gamma_2 (s-v)^{-\alpha} (t-s)^{1 + \epsilon_2} 
		\end{align*}
		for every $(v,s,t)\in [0,T]^3_\leq$ (with $t-s \leq s-v$ assumed in the second estimate).
		Furthermore, there exists an $\epsilon \in (0, 1)$ such that,
		for all partitions $\pi$ of $[s, t] \subseteq [0, T]$,
		\begin{align*}
			\norm[\Big]{\mathcal{A}_t - \mathcal{A}_s - 
				\sum_{[s', t'] \in \pi} A_{s', t'}}_{L^p_{\omega}} 
			\lesssim_{d, p, \epsilon_1, \epsilon_2, \alpha, T} (\Gamma_1 + \Gamma_2) \abs{\pi}^{\epsilon}.
		\end{align*}
	\end{lemma}
	We make a number of remarks related to the stochastic sewing lemma.
	\begin{remark}\label{rem:shifted_ssl_1}
		The original version of the stochastic sewing lemma of \cite{le20} is stated for $v = s$ and $\alpha = 0$.
		The shifted stochastic sewing lemma was first obtained by \cite[Lem.~2.2]{gerencser2022}, 
		where the case $v = s - M (t-s)$ for some $M > 0$ and $\alpha = 0$ are treated. Therefore, Lem.~\ref{lem:ssl} generalises both the sewing lemmas of \cite{le20} and \cite{gerencser2022}, since it allows for arbitrary shifts.
	\end{remark}
	\begin{remark}\label{rem:shifted_ssl_2}
		We note that in the setting of \cite{le20,gerencser2022}, as remarked in the paragraph preceding  \cite[Lem.~2.2]{gerencser2022}, 
		it is in fact enough to check that the condition analogous to \eqref{eq:ssl_assump}, in that setting, holds for $s,u,t \in [0,T]^3_{\leq}$ also satisfying
		\begin{align*}
			\min\{u-s, t-u\} \geq (t-s)/3. 
		\end{align*}
		This assumption on $u$ is necessary in the proof of Prop.~\ref{prop:fbm_young} and Prop.~\ref{prop:fbm_rough}
		\footnote{It is natural to expect that such assumption on $u$ is also valid to impose for the fully shifted version \cite[Thm.~1.1]{matsuda22}; however, this fact is not explicitly shown therein and furthermore, we do not require the restriction in the cases where we apply the version of \cite{matsuda22}.}.
	\end{remark}
	\begin{remark}\label{rem:sewing_uniq}
		We repeatedly make use of the uniqueness statement of the (stochastic) sewing lemma in the following manner.  Given $A^1:[0,T]^{2}_\leq \to E$ and $A^2:[0,T]^2_\leq \to E$ such that, for some $\varepsilon > 0$, 
		\begin{align}\label{eq:SSL_uniq_check_1}
			\norm{A^1_{s, t} - A^2_{s, t}}_E \lesssim \abs{t-s}^{1+\varepsilon} \quad \forall (s, t) \in [0, T]_{\leq}^2.
		\end{align}
		Then, we may set $A_{s, t} \defby A^1_{s, t} - A^2_{s, t}$ which satisfies 
		\begin{equation*}
			\norm{\delta A_{s, u, t}}_E \leq \norm{A_{s, t}^1 - A^2_{s, t}}_E + \norm{A^1_{s, u} - A^2_{s, u}}_E 
			+ \norm{A^1_{u, t} - A^2_{u, t}}_E 
			\lesssim \abs{t-s}^{1 + \varepsilon}.
		\end{equation*}
		Hence, the assumptions of Lem.~\ref{lem:det_sewing} are satisfied by $A:[0,T]^2_\leq \to E$. However, since the constant $\mathcal{A} \equiv 0$ obviously satisfies the condition \eqref{eq:cA_uniq}, it must hold that
		\begin{equation}\label{eq:A_1_A_2_equal}
			\lim_{\substack{\pi \text{ is a partition of }[s, t], \\ \abs{\pi} \to 0}} 
			\sum_{[s', t'] \in \pi} A^1_{s', t'}
			= 
			\lim_{\substack{\pi \text{ is a partition of }[s, t], \\ \abs{\pi} \to 0}} 
			\sum_{[s', t'] \in \pi} A^2_{s', t'}
		\end{equation}
		under the additional condition that at least one of the Riemann sum approximations converges. 
		
		Similarly, in the setting of the stochastic sewing lemma, if  
		\begin{equation*}
			\norm{A^1_{s, t} - A^2_{s, t}}_{L^p_{\omega}} \lesssim (t-s)^{\frac{1}{2} + \varepsilon_1}, 
			\quad \norm{\mathbb{E}[A^1_{s, t} - A^2_{s, t} \vert \mathcal{F}_v]}_{L^p_{\omega}} \lesssim (s-v)^{-\alpha} (t-s)^{1 + \varepsilon_2}
		\end{equation*}
		with $\alpha \in [0, 1/2 + \varepsilon_2)$, then the identity \eqref{eq:A_1_A_2_equal} holds in $L^p(\Omega;E)$.
		If the latter estimate also holds with $\alpha = 0$, then one only needs to check the bound for $v = s$ or $v = s - (t-s)$.
	\end{remark}
	We will apply an appropriate version of the result depending on the problem at hand. 
	In our paper, we use the three versions of the stochastic sewing lemma: in Sec.~\ref{sec:young} we use Gerencsér's version, \cite[Lem.~2.2]{gerencser2022} (see proof of Prop.~\ref{prop:fbm_young}), 
	while in Sec.~\ref{sec:rough} we use both Lê's version \cite[Thm.~2.1]{le20} (see proof of Prop.~\ref{prop:fbm_rough}) 
	and the fully shifted version (see proof of Lem.~\ref{lem:integral_along_tilde_B}).
	
	\subsection{Fractional Brownian Motion}\label{subsec:fbm}
	We collect some necessary background on the fractional Brownian motion, referring to \cite[Sec.~5]{nualart06} and \cite[App.~A.1]{galeati_22_thesis} for more details. As specified in the notation sections we fix a $d_2\in \NN$ and will consider noise taking values in $\RR^{d_2}$. We also fix a probability space $(\Omega,\cF,\PP)$.
	\begin{definition}\label{def:fbm}
		Let $H \in (0, 1)$. The \emph{fractional Brownian motion} (fBm) with Hurst parameter $H$, started at $0$, is a centred Gaussian process 
		$$B^H = (B^{H, i})_{i=1}^{d_2}: [0, \infty)\times \Omega\to \RR^{d_2}$$
		such that $B^H_0 = 0$ and for all $i,\,j = 1,\ldots, d_2$, it holds that 
		\begin{align*}
			\mathbb{E}[B^{H, i}_t  B^{H, j}_s ] = \frac{\delta_{ij}}{2} (s^{2H} + t^{2H} - (t-s)^{2H}), 
			\quad \forall \, 0\leq s<t <\infty,
		\end{align*}
		where $\delta_{i j}$ is the Kronecker delta.
	\end{definition}
	It will be useful to have a more concrete representation for the fBm. 
	Given a Brownian motion $W$, the formally written process 
	\begin{align}\label{eq:fbm_volterra}
		B^H_t \defby	(\cK_H \, \mathrm{d} W)_t = \int_0^t K_H(t, s) \, \mathrm{d} W_s
	\end{align}
	defines an fBm with Hurst parameter $H$. We refer the reader to \cite[Sec.~5.1.3]{nualart06} for more details. In the case $H>1/2$ the kernel $K_H$ is explicitly given by 
	\begin{align}\label{eq:K_H_gtr_12}
		K_H(t, s) \defby c_H s^{\frac{1}{2} - H} \int_s^t (u-s)^{H-\frac{3}{2}} u^{H-\frac{1}{2}} \, \mathrm{d} u,
	\end{align}
	while for $H<1/2$ one has
	\begin{equation}\label{eq:K_H_less_12}
		K_H (t, s) \defby c_H \left( \Big( \frac{t}{s} \Big)^{H-\frac{1}{2}} 
		(t-s)^{H-\frac{1}{2}} 	- 
		\Big( H-\frac{1}{2} \Big) s^{\frac{1}{2} - H} 
		\int_s^t u^{H-\frac{3}{2}} (u-s)^{H-\frac{1}{2}} \, \mathrm{d} u
		\right).
	\end{equation}
	In both cases, the constant $c_H$ ensures that $B^H$ has the correct covariance.
	
	The main benefit of this representation is that since the operator $\cK_H$ admits an inverse onto its domain of definition, given an fBm $B^H$ we can recover the associated Brownian motion $W_t \defby (\cK_H^{-1} B^H)_t$. An explicit formula for the inverse is given, for example, by \cite[Eq.~A.7]{galeati_22_thesis}.
	Recall that a Brownian motion $W$ is called an $(\mathcal{F}_t)$-Brownian motion if for every $s < t$ the random variable $W_t - W_s$ is 
	independent of $\mathcal{F}_s$.
	\begin{definition}\label{def:adapted_fBm}
		Given a complete filtered space $(\Omega, \cF,(\mathcal{F}_t)_{t \geq 0}, \PP)$, an adapted process $(B^H_t)_{t \geq 0}$ is called
		an $(\mathcal{F}_t)$-\emph{fractional Brownian motion} if the process $\cK_H^{-1} B^H$ is 
		an $(\mathcal{F}_t)$-Brownian motion.
	\end{definition}
	When $H \in (1/3, 1/2]$, the process $B^H$ can be lifted to a second order geometric rough path, we state this fact without proof below, for details see \cite[Ch.~10]{FH20}.
	\begin{lemma}[Canonical Lift of fBm, {\cite[Ch.~10]{FH20}}]\label{lem:canonical_lift}
		Let $H \in (1/3, 1/2]$
		and $(B^{H, (n)})_{n=1}^{\infty}$ be a piecewise linear approximation of the fBm $B^H$. Then
		the canonical lift of $B^{H, (n)}$
		\begin{equation*}
			\left(B^{H, (n)}, \mathbb{B}^{H, (n)}\right) \defby \left(B^{H, (n)}, \left(\int_{s}^t B^{H, (n)}_{s, r} \otimes   
			\dot{B}^{H, (n)}_r \, \mathrm{d} r \right)_{0\leq s<t<\infty} \right)
		\end{equation*}
		converges in the space of
		$\alpha$-Hölder rough paths for any $\alpha \in (1/3, H)$
		to some geometric rough path $(B^H, \mathbb{B}^H)$, called the \emph{canonical lift of the fBm $B^H$}.
	\end{lemma}
	As the stochastic sewing lemma is a key tool in our proofs, we will need to understand the behaviour of the fBm under conditioning.
	Let $B^H$ be an $(\mathcal{F}_t)$-fBm.
	Recalling the Volterra representation \eqref{eq:fbm_volterra}, for $s < t$ we observe the decomposition 
	\begin{equation*}
		B^H_t = \mathbb{E}\left[B^H_t \vert \mathcal{F}_s\right] + \left(B^H_t - \mathbb{E}\left[B^H_t \vert \mathcal{F}_s\right]\right) 
		= \int_0^s K_H(t, r) \, \mathrm{d} W_r + \int_s^t K_H(t, r) \, \mathrm{d} W_r,
	\end{equation*}
	where the first term is measurable with respect to $\mathcal{F}_s$ and the second term is independent of $\mathcal{F}_s$. 
	In view of this decomposition,  
	we define the processes $Y^{H, s} = (Y^{H, s, i})_{i=1}^{d_2}$ and  $\tilde{B}^{H, s} = (\tilde{B}^{H, s, i})_{i=1}^{d_2}$ by 
	\begin{align}
		Y^{H, s}_t &\defby \int_0^s K_H(t, r) \, \mathrm{d} W_r \label{eq:def_Y},\\
		\tilde{B}^{H, s}_t &\defby \int_s^t K_H(t, r) \, \mathrm{d} W_r \label{eq:def_tilde_B}.
	\end{align}
	Informally, we refer to $Y$ as the \emph{history process} and $\tilde{B}$ as the \emph{update process}.  This decomposition is motivated by the requirement to have a rapidly fluctuating update process $\tilde{B}^{H,s}$ which is independent of $\mcF_s$ and provides regularisation through stochastic averaging, see Lem.~\ref{lem:gaussian_integral} while keeping a history process which can be suitably controlled.
	When, $H = 1/2$, by convention we set $Y^{H, s} \defby 0$ and $\tilde{B}^{H, s} \defby B^{H}$. 
	We additionally set
	\begin{align}\label{eq:def_rho}
		\rho^2_H(s, t) \defby \mathbb{E}[(\tilde{B}^{H, s, i}_t)^2] = \int_s^t K_H(t, r)^2 \, \mathrm{d} r
	\end{align}
	and have the following, two sided bound, for all $H\in (0,1)$
	\begin{align}\label{eq:rho_lower_bd}
		(t-s)^{2H} \lesssim \rho^2_H(s, t) \leq (t-s)^{2H}.
	\end{align}
	The upper bound of \eqref{eq:rho_lower_bd} follows trivially from \eqref{eq:K_H_gtr_12} and \eqref{eq:K_H_less_12}. In particular one readily checks
	\begin{equation*}
		\int_s^t K_H(t, r)^2 \, \mathrm{d} r \leq \int_0^s \left(K_H(t, r) - K_H(s, r)\right)^2  \, \mathrm{d} r + \int_s^t K_H(t, r)^2 \, \mathrm{d} r 
		= (t-s)^{2H}.
	\end{equation*}
	The lower bound of \eqref{eq:rho_lower_bd}, whose simple proof can be found for example as \cite[App.~A.1]{galeati_22_thesis}, expresses the so-called \emph{local nondeterminism} property of fBm, see also \cite{Berman:1973aa}. 
	
	The function $t \mapsto \rho_H(s, t)$ is increasing, whence we can consider the Riemann-Stieltjes integral $\int_a^b f(r) \rho_H(s, \mathrm{d} r)$. 
	Since $\rho_H(s, \cdot)$ is differentiable, one actually has $\rho_H(s, \mathrm{d} r) = \frac{\mathrm{d} \rho_H(s, r)}{\mathrm{d} r} \, \mathrm{d} r$. In particular, 
	\begin{equation}\label{eq:rho_RS}
		\int_u^t \rho_H^{\gamma}(s, r) \rho_H(s, \mathrm{d} r) = \frac{1}{\gamma+1}\int_u^t \frac{\mathrm{d} \rho_H^{\gamma+1}(s, r)}{\mathrm{d} r} \, \mathrm{d} r
		\leq \frac{\max\{\rho_H^{\gamma+1}(s, t), \rho_H^{\gamma+1}(s, u)\}}{\abs{\gamma+1}}.
	\end{equation}
	A similar remark holds for $\rho^2_H(s, \mathrm{d} r)$.

	We record the following lemmas on $Y^{H, s}$ and $\tilde{B}^{H, s}$, whose proofs will be given in App.~\ref{sec:gaussian}.
	\begin{lemma}\label{lem:Y_dot}
		Let $H \in (1/3, 1)$ and for $s\geq 0$ we define $Y^{H, s}$ by \eqref{eq:def_Y}.
		The process $Y^{H, s}$ is smooth in $(s, \infty)$. Furthermore, for every $p \in [1, \infty)$ and $0\leq s<u <t <\infty$, we have 
		\begin{align*}
			\norm{\dot{Y}^{H, s}_t}_{L^p_{\omega}} \lesssim_{p} (t-s)^{H-1},
			\quad \norm{Y^{H, s}_{u, t}}_{L^p_{\omega}} \lesssim_p (t-s)^H - (u-s)^H.
		\end{align*}
	\end{lemma}
	\begin{proof}
		The proof is given in App,~\ref{subsec:gaussian_sec2}.
	\end{proof}
	The following crucial lemma, shows that for $H\in (1/3,1/2)$, the \emph{update process} of the Volterra, fBm can be lifted to a second order rough path.
	\begin{lemma}\label{lem:tilde_B_rp}
		Given $H\in (1/3,1/2)$, then the pair $(\tilde{B}^{H, s}, \tilde{\mathbb{B}}^{H, s})$ is a geometric rough path of order $\alpha$ for any $\alpha < H$, where
		\begin{equation}\label{eq:tilde_bold_B}
			\tilde{\mathbb{B}}^{H, s}_{u, t} 
			\defby \mathbb{B}^H_{u, t} - \int_u^t Y^{H, s}_{u, r} \otimes \, \mathrm{d} B^H_r 
			-  \int_u^t B^{H}_{u, r} \otimes \dot{Y}^{H, s}_r \, \mathrm{d} r + \int_u^t Y^{H, s}_{u, r} \otimes \dot{Y}^{H, s}_r \, \mathrm{d} r.
		\end{equation}
		The second integral is well-defined as Riemann--Stieltjes integral, since $Y^{H, s}$ is smooth in $(s, \infty)$.
	\end{lemma}
	\begin{proof}
		The proof is given in App.~\ref{subsec:gaussian_sec2}.
	\end{proof}
	Below and throughout the text, the integral with respect to $\tilde{B}^{H, s}$ is understood as 
	Young integral when $H > 1/2$ and as
	rough integral with respect to the geometric rough path $(\tilde{B}^{H,s}, \tilde{\mathbb{B}}^{H, s})$ when $H \leq 1/2$. The following lemma can be understood as a Clark--Ocone type formula and we expect that a similar result should hold for a more general class of locally non-deterministic Gaussian processes.

	\begin{lemma}\label{lem:gaussian_integral}
		Let $H \in (1/3, 1)$, $0 \leq s \leq u < t$, $g \in C^3(\mathbb{R}^{d_1}, \mathbb{R})$ and $a \in \mcM(d_1,d_2)$. 
		Then, given $\tilde{B}^{H, s}$ and $\rho_H(s, r)$ as defined by \eqref{eq:def_tilde_B} and \eqref{eq:def_rho} respectively, one has the identity,
		\begin{equation*}
			\mathbb{E} \left[ \int_u^t g (a \tilde{B}^{H, s}_r)
			\, \mathrm{d} \tilde{B}_r^{H, s, i} \right]
			= \frac{1}{2} \sum_{j=1}^{d_1} a^{j i} \int_u^t [\partial_j \mathcal{P}_{\rho^2_H(s, r) a a^{\top}} g](0)  
			\, \rho^2_H(s,\mathrm{d} r),
		\end{equation*}
		where the integral on the left-hand side is understood as Young integral when $H > 1/2$ and as
		rough integral with respect to the geometric rough path $(\tilde{B}^{H,s}, \tilde{\mathbb{B}}^{H, s})$ 
		(see Lem.~\ref{lem:tilde_B_rp}) when $H \leq 1/2$.
		The integral in the right-hand side is understood as Riemann--Stieltjes integral with respect to the increasing function 
		$r \mapsto \rho^2_H(s, r)$.
	\end{lemma}
	\begin{proof}
		The proof is given in App.~\ref{subsec:int_tilde_B}.
	\end{proof}
	%
	%
	%
	%
	%
	%
	%
	\section{Young Case: $H\in (1/2,1)$}\label{sec:young}
	In this section we prove Thm.~\ref{thm:main} for the regime $H\in (1/2,1)$. In this parameter range we are able to appeal to Young integration theory in order to give a pathwise definition of solutions to the equation.
	\begin{equation}\label{eq:main_young}
		\dd X_t = \sigma(X_t) \, \mathrm{d} B^H_t, \quad X_0 = x \in \mbR^{d_1}
	\end{equation}
	To make this precise we recall the following definition. 
	\begin{definition}\label{def:young_de_sol}
		Given $T>0$, $d_1,\,d_2 \in \mbN$ a path $g\in \mcC^{\alpha}([0,T];\mbR^{d_2})$ and a map $\sigma :\mbR^{d_1} \to \mcM(d_1,d_2)$. We say that a path $y :[0,T]\to \mbR^{d_1}$ is a \emph{solution to the Young differential equation}
		\begin{align}\label{eq:young_de_pw}
			\mathrm{d} y_t = \sigma(y_t) \, \mathrm{d} g_t, \quad y_0 = x \in \mbR^{d_1}
		\end{align}
		on $[0,T]$ if $y$ is such that $\sigma(y) \in \mcC^{\beta}([0,T];\mcM(d_1,d_2))$ for some $\beta>0$ such that $\beta>1-\alpha$ and for all $t\in [0,T]$ one has
		\begin{align*}
			y_t = x + \int_0^t \sigma (y_r) \, \mathrm{d} g_r,
		\end{align*}
		where the integral on the right hand side is understood as a Young integral as described in Sec.~\ref{subsec:young_integration}.
	\end{definition}
	Since we are interested in probabilistic notions of uniqueness we develop the following series of probabilistic definitions for solutions to \eqref{eq:main_young}.
	\begin{definition}[Existence and Uniqueness of Weak-Young Solutions]\label{def:weak_young_sol}
		We say that a tuple
		\begin{equation*}
			(\Omega,\mcF,(\mcF_t)_{t\geq 0}, \mbP,B^H,X),
		\end{equation*}
		of a filtered probability space, an fBm of Hurst parameter $H$ and an $(\mathcal{F}_t)$-adapted process $X$ is a weak-Young solution to \eqref{eq:main_young} if $B^H$ is an $(\mcF_t)_{t\geq 0}$ fBm (in the sense of Def.~\ref{def:adapted_fBm}) and for $\mbP$-a.e $\omega \in \Omega$,  the process $X(\omega)$ solves the differential equation \eqref{eq:main_young} in the sense of Young (Def.~\ref{def:young_de_sol}). 
		
		We say that weak uniqueness holds for \eqref{eq:main_young} in the Young sense, if for any two weak solutions $(\Omega,\mcF,(\mcF_t)_{t\geq 0}, \mbP,B^H,X)$ and $(\tilde{\Omega},\tilde{\mcF},(\tilde{\mcF_t})_{t\geq 0}, \tilde{\mbP},\tilde{B}^H,\tilde{X})$ as above, the processes $X,\, \tilde{X}$ are equal in law, i.e.
		\begin{equation*}
			X_{\#} \mbP = \tilde{X}_{\#} \tilde{\mbP} \in \mcP(C([0,+\infty);\mbR^{})).
		\end{equation*}
	\end{definition}
	\begin{remark}\label{rem:weak_young_existence_compactness}
		As mentioned in the text following Prop.~\ref{prop:yamada_watanabe}, a standard compactness argument (given in \cite[App.~B]{matsuda22}) allows for the construction of weak solutions in the case $H\in (1/2,1)$ and $\sigma\in \mcC^{\gamma}(\mbR^{d_1};\mcM(d_1,d_2))$ for $\gamma>\frac{1-H}{H}$.
		
		Note that since a posteriori one has $y\in \mcC^{\alpha}([0,T];\mbR^{d_1})$ for any $\alpha \in (0,H)$, we have $\sigma(y) \in \mcC^{\gamma \alpha}([0,T];\mcM(d_1,d_2))$. Hence, the conditions $\alpha<H$ and $\gamma>\frac{1-H}{H}$ imply that $\alpha \gamma +\alpha = (1+\gamma)\alpha >\frac{\alpha}{H}>1$, hence the Young integral in \eqref{eq:young_de_pw} is well-defined.

		In the regime $H\in (1/2,1)$ our main result (Thm.~\ref{thm:main}) obtains pathwise uniqueness for $\sigma$  that is $\gamma$-H\"older continuous for $\gamma >\frac{1}{2H}$ and are uniformly elliptic; in this case note that $\frac{1}{2H}>\frac{1-H}{H}$, so that a fortiori we obtain pathwise uniqueness (under the assumption of uniform ellipticity) for these weak solutions and hence strong existence due to Prop.~\ref{prop:yamada_watanabe}.
	\end{remark}
	\begin{definition}[Existence of Strong-Young Solutions]\label{def:strong_young_sol}
		Given a probability space, $(\Omega,\mcF,\mbP)$ carrying an fBm $B^H$, we say that $X$ is a strong-Young solution to \eqref{eq:main_young} if for $\mbP$-a.e. $\omega\in \Omega$, the process $X(\omega)$ solves \eqref{eq:main_young} in the sense of Young (Def.~\ref{def:young_de_sol}) and $X$ is adapted to the natural filtration generated by $B^H$.
	\end{definition}
	\begin{definition}[Pathwise Uniqueness of Young Solutions]\label{def:pathwise_uniq_young}
		We say that \emph{pathwise-uniqueness} holds for \eqref{eq:main_young}, in the Young sense, if for any two, weak-Young solutions, $X$ and $\tilde{X}$ (in the sense of Def~\ref{def:weak_young_sol}), on the same filtered probability space $(\Omega,\mcF,(\mcF_t)_{t\in [0,T]},\mbP)$, it holds that $X(\omega)=\tilde{X}(\omega)$ for $\mbP$-a.e. $\omega \in \Omega$.
	\end{definition}
	Pathwise uniqueness for \eqref{eq:main_young} is proven in Sec.~\ref{subsec:pw_uniq}.

	\subsection{Young Integral with Irregular Integrand}\label{subsec:young_int}
	As explained in Sec.~\ref{subsec:proof_overview}, the key to obtaining pathwise uniqueness is to obtain suitable, stochastic estimates on integrals of the form, 
	\begin{align}\label{eq:int_f_young}
		\int_s^t f(X_r) \, \mathrm{d} B^H_r
	\end{align}
	for an irregular $f$ and a path $X$ in a suitable class. 	Since we will eventually take $X$ to be any weak solution to the SDE \eqref{eq:main_young}, a natural class to take are those paths which are stochastically controlled by $B^H$. It turns out that this class is also sufficient for our purposes. The following definition is given in the full range $H\in (0,1)$ since we also make use of it in Sec.~\ref{sec:rough} where we treat the case $H\in (1/3,1/2]$.
	\begin{definition}\label{def:control_prob}[Probabilistically Controlled Paths]
		Let $T>0$, $d_1,\,d_2\in \mbN$, $(\Omega,\mcF,(\mcF_t)_{t\in [0,T]},\mbP)$ be a filtered probability space and $B^H$ be an $\mbR^{d_2}$ dimensional, $(\mathcal{F}_t)$-fBm with Hurst parameter $H\in (1/3,1)$.
		Then, for $\beta \in (0, 1]$, $l \in (1, \infty)$ and a pair of ($\cF_t$)-adapted processes $(X, X'):[0,T]\to \mbR^{d_1}\times \mcM(d_1,d_2)$, we say that $X$ is \emph{probabilistically controlled} by $B^H$ and write $(X, X') \in
		\mathfrak{D}^{\beta, l}_{B^H}$ if 
		\begin{itemize}
			\item the \emph{derivative process}, $X'$ satisfies the analytic bound,
			\begin{equation*}
				\sup_{0 \leq s < t \leq T} \frac{| | X'_{s, t} | |_{L^p
						_{\omega}}}{(t - s)^{\beta H}} < \infty \quad \text{for all }\, p \in [1,\infty); 
			\end{equation*}
			\item the remainder  
			\begin{align}\label{eq:def_R}
				R_{s, t} \defby X_{s, t} - X'_s B^H_{s, t},
			\end{align}
			defined for all $0\leq s< t \leq T$, satisfies the bound
			\begin{equation*}
				\sup_{0 \leq s < t \leq T} \frac{| | R_{s, t} | |_{L^p
						_{\omega}}}{(t - s)^{(1 + \beta) H}} < \infty, \quad \text{for all }\, p \in [1,\infty); 
			\end{equation*}
			\item and for all $t\geq 0$, $X'(X')^\top \in \mcM_+(d_2)$, i.e. there exists an $l>0$ such that
			$l^{-1} \leq X'_t (X'_t)^{\top} \leq l$.
		\end{itemize}
		Given $(X,X')$ satisfying these conditions and $p\in [1,\infty)$, we set
		\begin{equation}\label{eq:def_X_norm}
			\opnorm{X}_{p,\beta} \defby \sup_{0 \leq t \leq T} \| X'_t
			\|_{L^p_{\omega}} + \sup_{0 \leq s < t \leq T} \frac{| | X'_{s, t} |
				|_{L^p_{\omega}}}{(t - s)^{\beta H}} + \sup_{0 \leq s < t \leq T} \frac{| |
				R_{s, t} | |_{L^p_{\omega}}}{(t - s)^{(1 + \beta) H}} . 
		\end{equation}
	\end{definition}
	The main result of this subsection is the following proposition, which gives an analogue of the fundamental estimate of Young integration for integrals of the form \eqref{eq:int_f_young}.
	\begin{proposition}\label{prop:fbm_young}
		Let $T>0$, $d_1,\,d_2 \in \mbN$, $B^H$ be an $\mbR^{d_2}$-fBm, with Hurst parameter $H \in ( 1/2, 1)$, adapted to a given filtration $(\mcF_t)_{t\in [0,T]}$, $\beta \in (\frac{1}{2H}, 1]$ and $l \in [1, \infty)$. 
		Then, given $(X, X') \in \mathfrak{D}^{\beta, l}_{B^H}$, 
		$f \in C^1(\R^{d_1}, \R)$, $p \in [2, \infty)$, $\tilde{\gamma} \in (\frac{1}{2H} - 1, 0)$ and $(s,t)\in [0,T]^2_{\leq}$,  it holds that
		\begin{equation*}
			\Big\| \int_s^t f (X_r) \mathrm{d} B_r^H \Big\|_{L^p_{\omega}} 
			\lesssim_{\beta, l, p, \tilde{\gamma}, T} \| f \|_{\mathcal{C}_x^{\tilde{\gamma}}} \big((t - s)^{(1+\tilde{\gamma})H}
			+ \opnorm{X}_{2p,\beta} (t-s)^{(1+ \beta +\tilde{\gamma})H}\big).
		\end{equation*}
	\end{proposition}
	\begin{proof}
		For simplicity, in the remainder of the proof we write $B \defby B^H$ and do not keep track of dependencies on $\beta, l, p, \tilde{\gamma}, T$.
		Our strategy of proof is to apply Gerencsér's shifted stochastic sewing, \cite[Lem.~2.2]{gerencser2022}  (equivalently Lem.~\ref{lem:ssl} with $\alpha = 0$ and $v = s - (t-s)$) to the germ
		\begin{equation}\label{eq:germ_young}
			A_{s, t} \defby \mathbb{E}_{s - (t-s)} \left[\int_s^t f (X_{s-(t-s)} + X_{s-(t-s)}' B_{s-(t-s), r}) \, \mathrm{d} B_r\right] .
		\end{equation}
		As we demonstrate below, the shifted conditional, expectation and coefficients, along with the shifted stochastic sewing lemma, allow us to obtain the claimed estimate. We note that the shift is necessary, for instance to bound \eqref{eq:sing_int} below.
		We decompose the proof into six steps. First demonstrate that the integral $\int_s^t f(X_r)\dd B^H_r$ can be realised as the (shifted) stochastic sewing of \eqref{eq:germ_young}. 
		
		\textbf{Step 1.} 
		This is equivalent to the claim that the following limit holds in $L^p(\Omega;\mbR)$,
		\begin{align}\label{eq:integral_by_A}
			\int_s^t f(X_r) \, \mathrm{d} B_r = \lim_{\substack{\pi \text{ is a partition of }[s,t],\\\abs{\pi} \to 0}}
			\sum_{[s', t'] \in \pi} A_{s', t'}.
		\end{align}
		Namely, this shows that the germ $(A_{s, t})$ provides a correct approximation to the integral, \eqref{eq:int_f_young}. 
		To prove this identity, we first set $v \defby s - (t-s)$ and define the auxiliary germ
		\begin{align}\label{eq:germ_young_a0}
			A^0_{s, t} \defby
			\int_s^t f (X_{v} + X_{v}' B_{v, r}) \,\mathrm{d} B_r.
		\end{align}
		It is directly clear that the Riemann sums of $A^0$ and $A$ coincide. To wit, appealing to the uniqueness element of the stochastic sewing lemma, as in Rem.~\ref{rem:sewing_uniq}, it suffices to check that 
		\begin{align}\label{eq:A_A_0_one_half}
			\norm{A_{s, t} - A^0_{s, t}}_{L^p_\omega} \lesssim \abs{t-s}^{\frac{1}{2}+\varepsilon_1} \quad \forall (s, t) \in [0, T]_{\leq}^2,
		\end{align}
		and
		\begin{equation*}
			\left\|\mbE_{s-(t-s)}\left[A_{s,t}-A^0_{s,t}\right]\right\|_{L^p_\omega} \lesssim |t-s|^{1+\varepsilon_2},\quad \forall (s, t) \in [0, T]_{\leq}^2.
		\end{equation*}
		The latter is immediate since $\mbE_{s-(t-s)} \left[A_{s,t}-A^0_{s,t}\right] =0$. The former also follows rather straightforwardly, we apply the triangle inequality, the fundamental estimate of Young integration \eqref{eq:young_fund_est} and the assumptions on $f,\, (X,X')$ to find,
		\begin{equation}
			\|A_{s,t}-A^0_{s,t}\|_{L^p_\omega} \leq \|A_{s,t}\|_{L^p_\omega} +\|A^0_{s,t}\|_{L^p_\omega} \lesssim_{p, f, (X, X'), T} |t-s|^H.
		\end{equation}
		Hence, \eqref{eq:A_A_0_one_half} follows since here we only consider $H>1/2$. Therefore, in order to obtain \eqref{eq:integral_by_A} it is enough to show that
		\begin{align}\label{eq:integral_by_A_0}
			\int_s^t f(X_r) \, \mathrm{d} B_r = \lim_{\substack{\pi \text{ is a partition of }[s,t],\\\abs{\pi} \to 0}}
			\sum_{[s', t'] \in \pi} A^0_{s', t'}.
		\end{align}
		We achieve this by applying the vanilla sewing lemma, (Lem.~\ref{lem:det_sewing}) with $E=L^p(\Omega;\mbR^{d_2})$.
		Indeed, with $v = s - (t-s)$, we have
		\begin{multline*}
			\abs[\Big]{\int_s^t f(X_r) \, \mathrm{d} B_r- A^0_{s, t}}
			\leq \abs[\Big]{\int_s^t f(X_r) \, \mathrm{d} B_r- f(X_s) B_{s, t}} 
			+\abs{A^0_{s,t} - f(X_{v} + X'_{v}B_{v, s}) B_{s,t}} \\
			+\abs{f(X_s) B_{s, t} - f(X_{v} + X'_{v}B_{v, s}) B_{s,t}}.
		\end{multline*}
		By the remainder estimate of Young integration \eqref{eq:young_fund_est},
		we obtain
		\begin{align*}
			\norm[\Big]{\int_s^t f(X_r) \, \mathrm{d} B_r- f(X_s) B_{s, t}}_{L^p_{\omega}}
			&\lesssim \norm{f}_{C^1} \opnorm{X}_{2p, \beta} (t-s)^{(1+\beta) H}, \\
			\norm{A^0_{s,t} - f(X_{v} + X'_{v}B_{v, s}) B_{s,t}}_{L^p_{\omega}}
			&\lesssim \norm{f}_{C^1} \opnorm{X}_{2p, \beta} (t-s)^{(1+\beta) H}.
		\end{align*}
		In addition, since
		\begin{align*}
			\abs{f(X_s)  - f(X_{v} + X'_{v}B_{v, s})}\leq
			\norm{f}_{C^1} \abs{R_{v, s}},
		\end{align*}
		where $R$ is defined by \eqref{eq:def_R},
		we get
		\begin{align*}
			\norm{f(X_s) B_{s, t} - f(X_{v} + X'_{v}B_{v, s}) B_{s,t}}_{L^p_{\omega}}
			\lesssim \norm{f}_{C^1} \opnorm{X}_{2p, \beta} (t-s)^{(2+\beta) H}.
		\end{align*}
		Combining these estimates, we obtain 
		\begin{align*}
			\norm[\Big]{\int_s^t f(X_r) \, \mathrm{d} B_r- A^0_{s, t}}_{L^p_{\omega}}
			\lesssim \norm{f}_{C^1} \opnorm{X}_{2p, \beta} (t-s)^{(1+\beta) H}.
		\end{align*}
		As $(1+\beta) H > 1$, appealing this time to the uniqueness element of the deterministic sewing lemma (see Rem.~\ref{rem:sewing_uniq}), the identity \eqref{eq:integral_by_A_0} is established.

		We are now in a position to exploit our choice of germ, given by \eqref{eq:germ_young} to obtain the claimed estimate. We first recall the Volterra representation of the fBm, given by \eqref{eq:fbm_volterra}; i.e. for $v \geq 0$ we set 
		$Y^v \defby Y^{H, v}$ and $\tilde{B}^v \defby \tilde{B}^{H, v}$ where the right hand sides are given by \eqref{eq:def_Y} and \eqref{eq:def_tilde_B} 
		respectively. 
		%
		By Lem.~\ref{lem:gaussian_integral}, given any matrix $a \in \mcM(d_1,d_2)$ and $v \leq s \leq t$,  $g \in C^1(\R^{d_2}, \R)$
		it holds that
		\begin{equation}
			\mathbb{E} \left[ \int_s^t g (a \tilde{B}_r^v)
			\, \mathrm{d} \tilde{B}_r^{v, i} \right]
			=
			\frac{1}{2} \sum_{j=1}^{d_1} a^{j i} \int_s^t 
			\left(\partial_j \cP_{\rho^2(v, r) a a^{\top}} g\right) (0) \, \rho^2(v, \mathrm{d} r),
			\label{eq:g_along_B_expectation}
		\end{equation}
		with $\rho^2(v,r) \defby \rho_H^2(v, r)$ defined by \eqref{eq:def_rho}. We also note that 
		\begin{align}\label{eq:Y_tilde_B_prop}
			\text{$Y^v$ is measurable with respect to $\mathcal{F}_v$ while $\tilde{B}^v$ is independent of $\mathcal{F}_v$}.
		\end{align}
		Recall the germ $A_{s,t}$ from \eqref{eq:germ_young} we are now in a position to check the analytic bounds \eqref{eq:ssl_assump}. The first is checked by showing that
		\begin{align}\label{eq:young_ssl_cond_1}
			\norm{A_{s,t}}_{L^p_{\omega}} \lesssim \norm{f}_{\mathcal{C}_x^{\tilde{\gamma}}}  (t-s)^{(1+\tilde{\gamma}) H},
		\end{align}
		which is done in Step 2 below. The second estimate of \eqref{eq:ssl_assump} is shown through Steps 2 - 5.
		
		\textbf{Step 2.} 
		For $v \defby s - (t-s)$, by \eqref{eq:Y_tilde_B_prop} and the trivial fact that $X_{v},\,X'_v$ are measurable with respect to $\mcF_v$ we observe that
		\begin{align*}
			A_{s,t} &= \int_s^t \mathbb{E}\left[f(x + x'(y_{v, r} + \tilde{B}^v_r))\right]\Big\vert_{x=X_{v},
				x' = X'_{v}, y = Y^{v}}  \dot{Y}_r^v \, \mathrm{d} r \\
			&\phantom{=} + \mathbb{E}\left[\int_s^t f(x + x'(y_{v, r} + \tilde{B}^v_r) )\, \mathrm{d} \tilde{B}^v_r\right] \Big\vert_{x=X_{v},
				x' = X'_{v}, y = Y^{v}}\\
			&= \vcentcolon I_1 + I_2.
		\end{align*}
		Since $x' \tilde{B}^v_r$ is a centred Gaussian random variable with variance $\rho^2(v, r) x' (x')^{\mathrm{T}}$, we find
		\begin{align*}
			I_1 = \int_s^t \cP_{\rho^2(v, r) X'_v(X'_v)^{\top}} f(X_v + X'_v Y_{v, r})
			\dot{Y}^v_r \, \mathrm{d} r,
		\end{align*}
		and by Lem.~\ref{lem:semigroup_regularisation} and \eqref{eq:rho_lower_bd} 
		\begin{align*}
			\norm{\cP_{\rho^2(v, r) X'_v(X'_v)^{\top}} f}_{L^{\infty}}
			\lesssim \norm{f}_{\mathcal{C}_x^{\tilde{\gamma}}} \rho^{\tilde{\gamma}}(v, r)
			\lesssim \norm{f}_{\mathcal{C}_x^{\tilde{\gamma}}} (r-v)^{\tilde{\gamma} H}.
		\end{align*}
		By Lem.~\ref{lem:Y_dot} we have
		$\norm{\dot{Y}^v_r}_{L^p_{\omega}} \lesssim_p  (r-v)^{H-1}$.
		Therefore, 
		\begin{align*}
			\norm{I_1}_{L^p_{\omega}}
			\lesssim \norm{f}_{\mathcal{C}_x^{\tilde{\gamma}}} \int_s^t (r-v)^{\tilde{\gamma} H + H - 1} \, \mathrm{d} r
			\lesssim \norm{f}_{\mathcal{C}_x^{\tilde{\gamma}}} (t-s)^{(1+\tilde{\gamma})H}.
		\end{align*}
		To estimate $I_2$, we note that by \eqref{eq:g_along_B_expectation} which is obtained from our averaging result Lem.~\ref{lem:gaussian_integral}, applied with $g(z) = f(x + x'y_{v, r} + x' z)$, we find 
		\begin{align*}
			I_2 &= \mathbb{E}\left[\int_s^t f(x + x'(y_{v, r} + \tilde{B}^v_r) )\, \mathrm{d} \tilde{B}^v_r\right] \Big\vert_{x=X_{v},
				x' = X'_{v}, y = Y^{v}}\\
			&	=
			\frac{1}{2} \sum_{j=1}^{d_1} (X'_v)^{j i} \int_s^t 
			\left(\partial_j \cP_{\rho^2(v, r) X'_v (X'_v)^{\top}}  f\right)(X_v + X_v' Y^v_{v, r}) 
			\, \rho^2(v, \mathrm{d} r).
		\end{align*}
		Then, appealing to Lem.~\ref{lem:semigroup_regularisation} we obtain a regularising effect by the estimate,
		\begin{align*}
			\norm{\cP_{\rho^2(v, r) X'_v (X'_v)^{\top}}  f(X_v + X_v' Y^v_{v, r})}_{C^1}
			\lesssim (r-v)^{H(\tilde{\gamma}-1)} \norm{f}_{\mathcal{C}_x^{\tilde{\gamma}}},
		\end{align*}
		so that making use of \eqref{eq:rho_RS} and \eqref{eq:rho_lower_bd}, we obtain
		\begin{align*}
			\norm{I_2}_{L^p_{\omega}} 
			&\lesssim 
			\norm{f}_{\mathcal{C}_x^{\tilde{\gamma}}} \int_s^t (r-v)^{H (\tilde{\gamma}-1)}  \, \rho^2(v, \mathrm{d} r) 
			\leq
			\norm{f}_{\mathcal{C}_x^{\tilde{\gamma}}} \int_s^t \rho(v, r)^{\tilde{\gamma} -1}  \, \rho^2(v, \mathrm{d} r)  \\
			&\lesssim \norm{f}_{\mathcal{C}_x^{\tilde{\gamma}}} \rho(v, t)^{\tilde{\gamma} + 1} 
			\lesssim \norm{f}_{\mathcal{C}_x^{\tilde{\gamma}}} (t-s)^{(\tilde{\gamma}+1)H}.
		\end{align*}
		This concludes Step 2 and the proof of \eqref{eq:young_ssl_cond_2}. 
		
		In the remaining steps we show that for $u \in (s, t)$ with 
		\begin{align}\label{eq:u_unif}
			\min\{t-u, u-s\} \geq \frac{t-s}{3},
		\end{align}
		we have
		\begin{align}\label{eq:young_ssl_cond_2}
			\norm{\mathbb{E}_{s-(t-s)} \left[\delta A_{s, u, t}\right]}_{L^p_{\omega}}
			\lesssim \norm{f}_{\mathcal{C}_x^{\tilde{\gamma}}} \opnorm{X}_{2p, \beta} (t-s)^{(1+\beta+\tilde{\gamma})H},
		\end{align}
		which is the final estimate needed to show that $A_{s,t}$ satisfies the analytic bounds of the stochastic sewing lemma, \eqref{eq:ssl_assump}. We recall that imposing the additional condition of \eqref{eq:u_unif} does not alter the applicability of the stochastic sewing lemma, see Rem.~\ref{rem:shifted_ssl_2}.

		\textbf{Step 3.}
		We first perform some simple algebraic manipulations. Let us set
		\begin{align}\label{eq:v_s}
			v_1 \defby s - (t-s), \quad v_2 \defby s- (u-s), \quad v_3 \defby u - (t-u).
		\end{align}
		It is clear that, 
		\begin{align*}
			v_1 < v_2 < \min\{v_3, s\} \leq \max\{v_3, s\} < u < t.
		\end{align*}
		One may then check that
		$\mathbb{E}_{v_1} \left[\delta A_{s, u, t}\right] = \mathbb{E}_{v_1}[I_3 + I_4]$,
		where
		\begin{align*}
			I_3 &\defby \mathbb{E}_{v_2}\left[ \int_s^u \left(f(X_{v_1}+X'_{v_1} B_{v_1, r}) - f(X_{v_2}+X'_{v_2} B_{v_2, r}) \right)
			\, \mathrm{d} B_r\right], \\
			I_4 &\defby
			\mathbb{E}_{v_3}\left[ \int_u^t \left(f(X_{v_1}+X'_{v_1} B_{v_1, r}) - f(X_{v_3}+X'_{v_3} B_{v_3, r}) \right)
			\, \mathrm{d} B_r\right].
		\end{align*}
		Since the estimate for $I_3$ and $I_4$ are similar (the latter requires exchanging $(s, u)$ for $(u, t)$), we only present the estimates for $I_3$.
		Given $j=1,2$,  we observe
		\begin{align*}
			\mathbb{E}_{v_2}\left[ \int_s^u f(X_{v_j}+X'_{v_j} B_{v_j, r})
			\, \mathrm{d} B_r\right]
			= &\,\mathbb{E}_{v_2}\left[\int_s^u f(X_{v_j}+X'_{v_j} B_{v_j, r}) \dot{Y}_r^{v_2} \, \mathrm{d} r \right]\\
			&+ \mathbb{E}_{v_2}\left[\int_s^u f(X_{v_j}+X'_{v_j} B_{v_j, r}) \, \mathrm{d} \tilde{B}_r^{v_2}\right].
		\end{align*}
		By \eqref{eq:Y_tilde_B_prop}, the first term is equal to
		\begin{align*}
			\int_s^u \left(\cP_{\rho^2(v_2, r) X'_{v_j}(X'_{v_j})^{\top}} f\right)
			(X_{v_j} + X'_{v_j}(B_{v_j, v_2} + Y^{v_2}_{v_2, r})) \dot{Y}^{v_2}_r \, \mathrm{d} r
		\end{align*}
		and by \eqref{eq:g_along_B_expectation} the second term is equal to
		\begin{align*}
			\frac{1}{2} X'_{v_j} \int_s^u 
			\left(\nabla \cP_{\rho^2(v_2, r) X'_{v_j}(X'_{v_j})^{\top}} f\right)
			(X_{v_j} + X'_{v_j}(B_{v_j, v_2} + Y^{v_2}_{v_2, r})) \, \rho^2(v_2, \mathrm{d} r).
		\end{align*}
		Hence, we write $I_3 = I_{3,1} + I_{3,2}$, where 
		\begin{align*}
			I_{3,1} \defby
			\int_s^u &\bigg( \cP_{\rho^2(v_2, r) X'_{v_1}(X'_{v_1})^{\top}} f
			\left(X_{v_1} + X'_{v_1}(B_{v_1, v_2} + Y^{v_2}_{v_2, r})\right) \\
			&- \cP_{\rho^2(v_2, r) X'_{v_2}(X'_{v_2})^{\top}} f
			\left(X_{v_2} + X'_{v_2}Y^{v_2}_{v_2, r}\right) 
			\dot{Y}^{v_2}_r \bigg)\, \mathrm{d} r,
		\end{align*}
		\begin{align*}
			I_{3,2} \defby 
			\frac{1}{2}  \int_s^u &
			\bigg( X'_{v_1} \nabla \cP_{\rho^2(v_2, r) X'_{v_1}(X'_{v_1})^{\top}} f
			\left(X_{v_1} + X'_{v_1}(B_{v_1, v_2} + Y^{v_2}_{v_2, r})\right) \\
			&- X'_{v_2} \nabla \cP_{\rho^2(v_2, r) X'_{v_2}(X'_{v_2})^{\top}} f
			\left(X_{v_2} + X'_{v_2} Y^{v_2}_{v_2, r}\right) \bigg)
			\, \rho^2(v_2, \mathrm{d} r).
		\end{align*}

		\textbf{Step 4.} In view of Step 3, in order to prove \eqref{eq:young_ssl_cond_2}, it remains to show 
		\begin{align}\label{eq:est_I_3_1}
			\norm{I_{3,1}}_{L^p_{\omega}} &\lesssim \norm{f}_{\mathcal{C}_x^{\tilde{\gamma}}} \opnorm{X}_{2p, \beta} (t-s)^{(1+\beta+\tilde{\gamma})H}, \\
			\label{eq:est_I_3_2}
			\norm{I_{3,2}}_{L^p_{\omega}} &\lesssim \norm{f}_{\mathcal{C}_x^{\tilde{\gamma}}} \opnorm{X}_{2p, \beta} (t-s)^{(1+\beta+\tilde{\gamma})H},
		\end{align}
		along with similar arguments for $I_4$ which we do not present.
		In this step, we estimate $I_{3,1}$.
		
		We decompose $I_{3,1} = I_{3,1,1} + I_{3,1,2}$, where
		\begin{align*}
			I_{3,1,1} \defby &
			\int_s^u \big( \cP_{\rho^2(v_2, r) X'_{v_1}(X'_{v_1})^{\top}} f
			(X_{v_1} + X'_{v_1}(B_{v_1, v_2} + Y^{v_2}_{v_2, r})) \\
			&- \cP_{\rho^2(v_2, r) X'_{v_1}(X'_{v_1})^{\top}} f
			(X_{v_2} + X'_{v_2}Y^{v_2}_{v_2, r}) \big)
			\dot{Y}^{v_2}_r \, \mathrm{d} r
		\end{align*}
		and
		\begin{equation*}
			I_{3,1,2} \defby
			\int_s^u \big( \cP_{\rho^2(v_2, r) X'_{v_1}(X'_{v_1})^{\top}}
			- \cP_{\rho^2(v_2, r) X'_{v_2}(X'_{v_2})^{\top}} \big)
			f(X_{v_2} + X'_{v_2}Y^{v_2}_{v_2, r})
			\dot{Y}^{v_2}_r \, \mathrm{d} r.
		\end{equation*}
		To estimate $I_{3,1,1}$, we apply Lem.~\ref{lem:semigroup_regularisation} to obtain the Lipschitz estimate
		\begin{align}\label{eq:P_C_1}
			\norm{\cP_{\rho^2(v_2, r) X'_{v_1}(X'_{v_1})^{\top}} f}_{C^1}
			\lesssim \rho^{\tilde{\gamma} - 1}(v_2, r) \norm{f}_{\mathcal{C}_x^{\tilde{\gamma}}},
		\end{align}
		which yields
		\begin{equation*}
			\abs{I_{3,1,1}}
			\lesssim 
			\norm{f}_{\mathcal{C}_x^{\tilde{\gamma}}}
			\int_s^u \rho^{\tilde{\gamma} - 1}(v_2, r)  
			\times \abs{X_{v_1} + X'_{v_1}(B_{v_1, v_2} + Y^{v_2}_{v_2, r}) - (X_{v_2} + X'_{v_2}Y^{v_2}_{v_2, r})} \abs{\dot{Y}^{v_2}_r}
			\, \mathrm{d} r.
		\end{equation*}
		Recalling the definition of the remainder $R$ of a controlled path from \eqref{eq:def_R}, we observe that $X_{v_1} + X'_{v_1}B_{v_1, v_2} - X_{v_2} = - R_{v_1, v_2}$. Therefore, 
		\begin{equation*}
			\norm{I_{3,1,1}}_{L^p_{\omega}} 
			\lesssim \norm{f}_{\mathcal{C}^{\tilde{\gamma}}_x}\int_s^u \rho^{\tilde{\gamma} - 1}(v_2, r)  
			\times \left(\norm{R_{v_1, v_2} \dot{Y}^{v_2}_r}_{L^p_{\omega}} + 
			\norm{X_{v_1, v_2}' Y_{v_2, r}^{v_2} \dot{Y}^{v_2}_r}_{L^p_{\omega}} \right) \, \mathrm{d} r.
		\end{equation*}
		By Hölder's inequality and recalling the norm $\opnorm{X}_{p, \beta}$, defined by \eqref{eq:def_X_norm}, we find 
		\begin{align*}
			\norm{I_{3,1,1}}_{L^p_{\omega}} 
			\lesssim &\norm{f}_{\mathcal{C}_x^{\tilde{\gamma}}} \opnorm{X}_{2p, \beta} \int_s^u \rho^{\tilde{\gamma} - 1}(v_2, r)\Big(\norm{\dot{Y}^{v_2}_r}_{L^{2p}_{\omega}} (v_2-v_1)^{(1+\beta)H}  \\
			&\hspace{15em}  + \norm{Y_{v_2, r}^{v_2}}_{L^{4p}_{\omega}} \norm{\dot{Y}^{v_2}_r}_{L^{4p}_{\omega}} (v_2 - v_1)^{\beta H} \Big) \, \mathrm{d} r.
		\end{align*}
		By Lem.~\ref{lem:Y_dot} we have $\norm{\dot{Y}^{v_2}_r}_{L^{4p}_{\omega}} \lesssim (r-v_2)^{H-1}$ and 
		$\norm{Y_{v_2, r}^{v_2}}_{L^{4p}_{\omega}}
		\lesssim (r-v_2)^H$.
		Hence, 
		also using the bound \eqref{eq:rho_lower_bd} to estimate $\rho(v_2, r)$, we have 
		\begin{align}
			\int_s^u \rho^{\tilde{\gamma} - 1}(v_2, r) \norm{\dot{Y}^{v_2}_r}_{L^{2p}_{\omega}} (v_2-v_1)^{(1+\beta)H} \, \mathrm{d} r  &\lesssim (v_2 - v_1)^{(1+\beta) H} \int_s^u (r-v_2)^{\tilde{\gamma} H - 1} \, \mathrm{d} r \label{eq:sing_int}\\
			&\lesssim (v_2 - v_1)^{(1+\beta) H} (s-v_2)^{\tilde{\gamma} H} \notag
		\end{align}
		and 
		\begin{align*}
			\int_s^u \rho^{\tilde{\gamma} - 1}(v_2, r) \norm{Y_{v_2, r}^{v_2}}_{L^{4p}_{\omega}} \norm{\dot{Y}^{v_2}_r}_{L^{4p}_{\omega}} (v_2 - v_1)^{\beta H} \, \mathrm{d} r &\lesssim (v_2 - v_1)^{\beta H} \int_s^u (r-v_2)^{\tilde{\gamma} H + H - 1} \, \mathrm{d} r \\
			&\lesssim (v_2 - v_1)^{\beta H} (u-v_2)^{(1+\tilde{\gamma}) H}. 
		\end{align*}
		Recalling $v_1$, $v_2$ from \eqref{eq:v_s} and the condition \eqref{eq:u_unif} on $u$, we obtain 
		\begin{align*}
			\norm{I_{3,1,1}}_{L^p_{\omega}} \lesssim \norm{f}_{\mathcal{C}_x^{\tilde{\gamma}}} \opnorm{X}_{2p, \beta} (t-s)^{(1+\beta+\tilde{\gamma})H},
		\end{align*}
		an estimate compatible with \eqref{eq:est_I_3_2}.
		
		We move to estimate $I_{3,1,2}$. 
		By Lem.~\ref{lem:semigroup_difference},
		\begin{multline*}
			\left\|\left( \cP_{\rho^2(v_2, r) X'_{v_1}(X'_{v_1})^{\top}}
			- \cP_{\rho^2(v_2, r) X'_{v_2}(X'_{v_2})^{\top}} \right)
			f\,\right\|_{L^{\infty}} \\
			\lesssim \left|X'_{v_1}(X'_{v_1})^{\top} - X'_{v_2}(X'_{v_2})^{\top}\right|\,
			\rho^{\tilde{\gamma}}(v_2, r) \norm{f}_{\mathcal{C}_x^{\tilde{\gamma}}}.
		\end{multline*}
		Thus, by Hölder's inequality,  
		\begin{align*}
			\norm{I_{3,1,2}}_{L^p_{\omega}} 
			\lesssim \norm{f}_{\mathcal{C}_x^{\tilde{\gamma}}} \int_s^u \rho^{\tilde{\gamma}}(v_2, r) \norm{X'_{v_1, v_2}}_{L^{2p}_{\omega}} \norm{\dot{Y}^{v_2}_r}_{L^{2p}_{\omega}} 
			\, \mathrm{d} r.
		\end{align*}
		Recalling the definition of $\opnorm{X}_{2p, \beta}$ from \eqref{eq:def_X_norm}, applying Lem.~\ref{lem:Y_dot} and \eqref{eq:rho_lower_bd} we then find
		\begin{align*}
			\norm{I_{3,1,2}}_{L^p_{\omega}} &\lesssim \norm{f}_{\mathcal{C}_x^{\tilde{\gamma}}} \opnorm{X}_{2p, \beta} (v_2 - v_1)^{\beta H} \int_s^u (r-v_2)^{\tilde{\gamma} H + H - 1} \, \mathrm{d} r \\
			&\lesssim \norm{f}_{\mathcal{C}_x^{\tilde{\gamma}}} \opnorm{X}_{2p, \beta} (t-s)^{(1+\beta + \tilde{\gamma})H},
		\end{align*}
		which, combined with the estimate of $I_{3,1,1}$ in the previous paragraph, yields \eqref{eq:est_I_3_2}.
		
		\textbf{Step 5.} In this last step, we estimate $I_{3,2}$ to obtain \eqref{eq:est_I_3_2}. 
		By adding and subtracting, we write $I_{3,2} = I_{3,2,1} + I_{3,2,2} + I_{3,2,3}$, where 
		\begin{align*}
			I_{3,2,1} \defby &\frac{1}{2}  \int_s^u 
			X'_{v_1} \big( \nabla \cP_{\rho^2(v_2, r) X'_{v_1}(X'_{v_1})^{\top}} f
			(X_{v_1} + X'_{v_1}(B_{v_1, v_2} + Y^{v_2}_{v_2, r})) \\
			&- \nabla \cP_{\rho^2(v_2, r) X'_{v_1}(X'_{v_1})^{\top}} f
			(X_{v_2} + X'_{v_2}Y^{v_2}_{v_2, r}) \big)
			\, \rho^2(v_2, \mathrm{d} r),
		\end{align*}
		\begin{multline*}
			I_{3,2,2} \defby
			\int_s^u X_{v_1}' \left( \nabla \cP_{\rho^2(v_2, r) X'_{v_1}(X'_{v_1})^{\top}}  - \nabla \cP_{\rho^2(v_2, r) X'_{v_2}(X'_{v_2})^{\top}} \right) \\
			\times f(X_{v_2} + X'_{v_2}Y^{v_2}_{v_2, r})
			\,  \rho^2(v_2, \mathrm{d} r),
		\end{multline*}
		and
		\begin{equation*}
			I_{3,2,3} \defby
			\frac{1}{2}  \int_s^u 
			(X'_{v_1} - X'_{v_2}) \nabla \cP_{\rho^2(v_2, r) X'_{v_2}(X'_{v_2})^{\top}} f
			(X_{v_2} + X'_{v_2} Y^{v_2}_{v_2, r}) 
			\, \rho^2(v_2, \mathrm{d} r).
		\end{equation*}
		The estimate of $I_{3,2,1}$ and that of $I_{3,2,2}$ are similar to that of $I_{3,1,1}$ and that of $I_{3,1,2}$ respectively; 
		we only demonstrate how to estimate $I_{3,2,1}$.
		By Lem.~\ref{lem:semigroup_regularisation}, 
		\begin{align*}
			\norm{\cP_{\rho^2(v_2, r) X'_{v_1}(X'_{v_1})^{\top}} f}_{C^2} 
			\lesssim \rho^{\tilde{\gamma} - 2}(v_2, r) \norm{f}_{\mathcal{C}_x^{\tilde{\gamma}}}.
		\end{align*}
		Therefore, 
		\begin{align*}
			\abs{I_{3,2,1}} \lesssim \norm{f}_{\mathcal{C}^{\tilde{\gamma}}_x}\int_s^u \rho^{\tilde{\gamma} - 1}(v_2, r) \left(\abs{R_{v_1, v_2}} + \abs{X_{v_1, v_2}' Y^{v_2}_{v_2, r}}\right) \, \rho(v_2, \mathrm{d} r ). 
		\end{align*}
		As above, recalling the definition of $\opnorm{X}_{p, \beta}$, the estimate on $Y^{v_2}$ from Lem.~\ref{lem:Y_dot} 
		and the bound \eqref{eq:rho_RS}, we get 
		\begin{align*}
			\norm{I_{3,2,1}} &\lesssim \norm{f}_{\mathcal{C}_x^{\tilde{\gamma}}} \opnorm{X}_{p, \beta} 
			\Big( (v_2 - v_1)^{(1+\beta)H} \int_s^u \rho^{\tilde{\gamma} - 1}(v_2, r) \, \rho(v_2, \mathrm{d} r)  \\
			&\phantom{\lesssim \norm{f}_{\mathcal{C}_x^{\tilde{\gamma}}} \opnorm{X}_{p, \beta} \Big(}\,\, +(v_2 - v_1)^{\beta H} \int_s^u \rho^{\tilde{\gamma}}(v_2, r) \, \rho(v_2, \mathrm{d} r)\Big) \\
			&\lesssim \norm{f}_{\mathcal{C}_x^{\tilde{\gamma}}} \opnorm{X}_{p, \beta} 
			\left[(v_2 - v_1)^{(1+\beta) H} \rho^{\tilde{\gamma}}(v_2, s) + (v_2 - v_1)^{\beta H} \rho^{\tilde{\gamma}+1}(v_2, u)\right]
		\end{align*}
		Applying the estimate \eqref{eq:rho_lower_bd} for $\rho(v_2, r)$, we get 
		\begin{align*}
			\norm{I_{3,2,1}} \lesssim \norm{f}_{\mathcal{C}_x^{\tilde{\gamma}}} \opnorm{X}_{p, \beta} (t-s)^{(1+\beta + \tilde{\gamma})H}.
		\end{align*}

		We finally estimate $I_{3,2,3}$. Using the bounds \eqref{eq:P_C_1} and \eqref{eq:rho_RS}, 
		\begin{align*}
			\abs{I_{3,2,3}} \lesssim \norm{f}_{\mathcal{C}_x^{\tilde{\gamma}}}\left|X_{v_1, v_2}'\right| \int_s^u \rho^{\tilde{\gamma}}(v_2, r) \, \rho(v_2, \mathrm{d} r)
			\lesssim \norm{f}_{\mathcal{C}_x^{\tilde{\gamma}}}\left|X_{v_1, v_2}'\right| \rho^{\tilde{\gamma} + 1}(v_2, u).
		\end{align*}
		Therefore, 
		\begin{align*}
			\norm{I_{3,2,3}}_{L^p_{\omega}} \lesssim  \norm{f}_{\mathcal{C}_x^{\tilde{\gamma}}} \opnorm{X}_{p, \beta} (t-s)^{(1+\beta+\tilde{\gamma})H},
		\end{align*}
		yielding the estimate \eqref{eq:est_I_3_2}, which in turn gives us \eqref{eq:young_ssl_cond_2}, completing the proof.
	\end{proof}
	\subsection{Pathwise Uniqueness}\label{subsec:pw_uniq}
	Thanks to Prop.~\ref{prop:fbm_young}, we are ready to prove our main result in the case $H \in (1/2,1)$.
	\begin{theorem}[Pathwise Uniqueness in the Young regime]\label{thm:pathwise_uniqueness_young}
		Let $T>0$, $d_1,\,d_2 \in \mbN$, $B^H$ be an fBm with Hurst parameter $H \in (1/2, 1)$, $\gamma \in (\frac{1}{2H},1)$ and $\sigma \in \mcC^{\gamma}(\mbR^{d_1};\mcM(d_1,d_2))$ be such that $\sigma \sigma^{\top}$ is uniformly elliptic, i.e. there exists an $l \in [1, \infty)$ such that 
		$l^{-1} \leq \sigma(x) \sigma^{\top}(x) \leq l$ for all $x \in \mathbb{R}^{d_1}$.
		Then, pathwise uniqueness holds 
		for \eqref{eq:main_young} in the sense of Def.~\ref{def:pathwise_uniq_young}.
	\end{theorem}
	Before giving the proof of Thm.~\ref{thm:pathwise_uniqueness_young}, we 
	prepare the following lemma.
	\begin{lemma}\label{lem:sol_is_controlled}
		In the setting of Thm.~\ref{thm:pathwise_uniqueness_young}, 
		let $(\Omega,\mcF,(\mcF_t)_{t\in [0,T]},\mbP,X, B^H)$ be a weak-Young solution to \eqref{eq:main_young} (see Def.~\ref{def:weak_young_sol}). Then, we have 
		$(X, \sigma(X)) \in \mathfrak{D}^{\gamma, l}_{B^H}$ as given by Def.~\ref{def:controlled_path}. 
	\end{lemma}
	\begin{proof}
		Let $\alpha < H$, but sufficiently close to $H$.
		We recall the seminorm $\llbracket \,\cdot\, \rrbracket_{\mathcal{C}^{\alpha}_{[s, t]}}$ from \eqref{eq:holder_seminorm}.
		Since 
		\begin{align*}
			X_{s, t} = \int_s^t \sigma(X_r) \, \mathrm{d} B^H_r,
		\end{align*}
		the remainder estimate of Young integration, estimate \eqref{eq:young_fund_est}, gives 
		\begin{align}\label{eq:X_remainder_estimate}
			\abs{X_{s,t} - \sigma(X_s) B^H_{s, t}}
			\lesssim_{\alpha, \gamma} \norm{\sigma}_{\mathcal{C}^{\gamma}_x} \llbracket X \rrbracket_{\mathcal{C}^{\alpha}_{[s, t]}}^{\gamma} \left\|B^H\right\|_{\mathcal{C}^{\alpha}_{[s, t]}} 
			(t-s)^{(1 + \gamma) \alpha}.
		\end{align}
		By \cite[Lem.~10.7]{friz_victoir_2010}, we have
		\begin{align}\label{eq:young_sol_a_priori_bd}
			\llbracket X \rrbracket_{\mathcal{C}^{\alpha}_T}
			\lesssim_{\alpha, \gamma, T} (1 + \norm{\sigma}_{\mathcal{C}^{\gamma}_x} \norm{B^H}_{\mathcal{C}^{\alpha}_T})^{\frac{1}{\alpha}}. 
		\end{align}
		Now, due to Gaussianity of $B^H$ and as a consequence of the Kolmogorov continuity criterion, for all $p < \infty$, 
		there exists a $C\coloneqq C(p, \alpha, T, \norm{\sigma}_{\mathcal{C}_x^{\gamma}})>0$ such that
		\begin{align}\label{eq:X_holder_moment}
			\left\|\llbracket X \rrbracket_{\mathcal{C}_T^{\alpha}}\right\|_{L^p_{\omega}} \leq C.
		\end{align}
		Combining \eqref{eq:X_remainder_estimate} and \eqref{eq:X_holder_moment}, we get 
		$\norm{X_{s,t}}_{L^p_\omega} \lesssim (t-s)^H$.
		In particular, we have $\norm{\sigma(X)_{s, t}}_{L^p_\omega} \lesssim (t-s)^{\gamma H}$.
		To get an estimate of the remainder 
		$R_{s,t} = X_{s, t} - \sigma(X_s) B_{s, t}^H$,
		it suffices to apply the sewing lemma in $L^p_{\omega}$ with germ 
		$A_{s, t} \defby \sigma(X_s) B_{s, t}^H$,
		since 
		\begin{equation*}
			\norm{\delta A_{s, u, t}}_{L^p_{\omega}} = \left\|\sigma(X)_{s, u} B^H_{u, t}\right\|_{L^p_{\omega}}
			\lesssim (t-s)^{(1 + \gamma) H}. 
		\end{equation*}
	\end{proof}
	We are now ready to give the proof of Thm.~\ref{thm:pathwise_uniqueness_young}.
	\begin{proof}[Proof of Thm.~\ref{thm:pathwise_uniqueness_young}]
		Let $\alpha \in (1/2,H)$ be arbitrarily close to $H$ and 
		$(X, B^H)$ and $(Y, B^H)$ be two weak solutions defined on the same filtered probability space $(\Omega,\mcF,(\mcF_t),\mbP)$  and driven by the same $(\mathcal{F}_t)$-fBm, $B^H$. 
		Since the requirement on $\gamma \in (\frac{1}{2H},1)$ is given by an open interval it is always possible to find a $\gamma' \in (\frac{1}{2H},\gamma)$. We then let $(\sigma_n)_{n=1}^{\infty}$ be a smooth approximation to $\sigma$ converging in $\mcC^{\gamma'}(\mbR^{d_1};\mcM(d_1,d_2))$ (which exists by the compact embedding between H\"older--Besov spaces).
		The remainder estimate of Young integration \eqref{eq:young_fund_est} gives 
		\begin{align}\label{eq:young_fundamental}
			\abs[\Big]{\int_s^t \sigma(Z_r) \, \mathrm{d} B^H_r - \sigma(Z_s) B^H_{s,t}}
			\lesssim_{\alpha, \gamma'} \norm{\sigma}_{\mathcal{C}_x^{\gamma'}} \norm{Z}_{\mathcal{C}_T^{\alpha}}^{\gamma'} 
			\left\|B^H\right\|_{\mathcal{C}^{\alpha}_T} (t-s)^{(1+\gamma') \alpha}
		\end{align}
		for $Z \in \{X, Y\}$.
		
		From the estimate \eqref{eq:young_fundamental} and convergence of the $\sigma_n$ in $\mcC^{\gamma'}(\mbR^{d_1};\mcM(d_1,d_2))$, it follows that
		\begin{align*}
			\lim_{n \to \infty} \int_0^{\cdot} \sigma_n(Z_r) \, \mathrm{d} B^H_r 
			= \int_0^{\cdot} \sigma(Z_r) \, \mathrm{d} B^H_r  
		\end{align*}
		uniformly in $[0, T]$. As the computation \eqref{eq:diff_linearisation} shows, 
		we have 
		\begin{align*}
			\int_0^t \left(\sigma_n(X_r) - \sigma_n(Y_r) \right) \, \mathrm{d} B^H_r 
			= \sum_k \int_0^t (X^{k}_r - Y^{k}_r) \, \mathrm{d} V^k_n(r),
		\end{align*}
		where 
		\begin{align*}
			V^k_n(r) \defby 
			\int_0^1 \int_0^r \partial_k \sigma_n(\theta X_u + (1-\theta) Y_u) \, \mathrm{d} B^H_u \, \mathrm{d} \theta.
		\end{align*}
		Hence, we obtain 
		\begin{align}\label{eq:X_V_n}
			X_t - Y_t = \lim_{n \to \infty}
			\sum_k \int_0^t (X^{ k}_r - Y^{k}_r) \, \mathrm{d} V^k_n(r).
		\end{align}
		Now our task is to prove the convergence of $V_n$. 
		By Lem.~\ref{lem:sol_is_controlled}, for $(s,t) \in [0,T ]^2_{\leq}$ we have
		\begin{align*}
			\theta X_{s, t} + (1-\theta) Y_{s,t} 
			= (\theta \sigma(X_s) + (1-\theta) \sigma(Y_s)) B_{s, t} 
			+ R^{\theta}_{s, t}
		\end{align*}
		with 
		\begin{align*}
			\norm{R_{s, t}^{\theta}}_{L^p_{\omega}} \lesssim_{p, \gamma, x, \sigma} (t-s)^{(1 + \gamma) H}.
		\end{align*}
		One technical problem is that 
		the Gubinelli derivative $\theta \sigma(X_s) + (1-\theta) \sigma(Y_s)$ might be degenerate, that is,  
		the matrix  
		\begin{align*}
			\left[\theta \sigma(X_s) + (1-\theta) \sigma(Y_s)\right] \left[\theta \sigma(X_s) + (1-\theta) \sigma(Y_s)\right]^{\top}
		\end{align*}
		could have an eigenvalue $0$. 
		(In the end, we will show that $X = Y$, hence it must be non-degenerate; but a priori we do not know this.)
		Because of this technical problem, we proceed as follows.
		
		{\bfseries Step 1.} We first assume that 
		\begin{align}\label{eq:sigma_low_oscillation}
			\sup_{x, \,y \in \mathbb{R}^d} \abs{\sigma(x) - \sigma(y)} \leq \frac{l^{-2}}{4},
		\end{align}
		where $l$ is the assumed ellipticity constant of $\sigma$. This condition implies that 
		for all $x,y\in \mbR^{d_1}$ and $\theta \in (0,1)$ we have
		\begin{align*}
			\left[\theta \sigma(x) + (1-\theta) \sigma(y)\right]& \left[\theta \sigma(x) + (1-\theta) \sigma(y)\right]^{\top}  \\
			&= [\sigma(x) + (1-\theta) (\sigma(y) - \sigma(x))] [\sigma(x) + (1-\theta) (\sigma(y) - \sigma(x))]^{\top} \\
			&\geq l^{-1} - \frac{l^{-1}}{2} - \frac{l^{-2}}{4} 
			\geq \frac{l^{-1}}{4}.
		\end{align*}
		Therefore, in combination of Lem.~\ref{lem:sol_is_controlled}, we see that
		\begin{align*}
			(\theta X + (1- \theta) Y, \theta \sigma(X) + (1-\theta) \sigma(Y)) \in \mathfrak{D}^{\gamma, l/4}_{B^H}.
		\end{align*}
		Prop.~\ref{prop:fbm_young} and standard embedding of H\"older--Besov spaces, yield
		\begin{equation*}
			\norm[\Big]{\int_s^t \partial_k (\sigma_n -\sigma_m)(\theta X_r^1 + (1-\theta) X_r^2) 
				\, \mathrm{d} B^H_r}_{L^p_{\omega}} \\
			\lesssim_{p, \gamma', x, \sigma} 
			\norm{\sigma_n -\sigma_m}_{\mathcal{C}_x^{\gamma'}} (t-s)^{\gamma' H}.
		\end{equation*}
		By Kolmogorov's continuity theorem, we see that 
		there exists a process $V^k$ such that for any $\beta < \gamma H$, 
		\begin{align*}
			\lim_{n \to \infty} \norm{\norm{V^k - V^k_n}_{\mathcal{C}^{\beta}_T}}_{L^p_{\omega}} = 0.
		\end{align*}
		Since $\gamma' H > 1/2$, we may take $\beta > 1/2$. Therefore, recalling \eqref{eq:X_V_n} we observe 
		\begin{align*}
			X_t - Y_t = \sum_k \int_0^t (X^{k}_r - Y^{k}_r) \, \mathrm{d} V^k_r \quad \text{a.s.}
		\end{align*}
		As this shows that $X - Y$ solves the linear Young differential equation 
		\begin{align*}
			\mathrm{d} x_t = \sum_k x^k_t \, \mathrm{d} V^k_t, \quad x_0 = 0,
		\end{align*}
		we see that $X - Y = 0$ a.s.
		
		{\bfseries Step 2.} Now we do not assume \eqref{eq:sigma_low_oscillation}. The new ingredient is a stopping time argument.
		We choose $\epsilon > 0$ so that 
		\begin{align*}
			\sup_{x, y: \abs{x-y} \leq \epsilon} \abs{\sigma(x) - \sigma(y)} \leq \frac{l^{-2}}{5},
		\end{align*}
		and we set $T^{(0)} \defby 0$ and inductively
		\begin{align*}
			T^{(i)} \defby\inf \set{t \geq T^{(i-1)} : \abs{X_t - X_{T^{(i-1)}}} \geq \epsilon/2 \text{ or } \abs{Y_t - Y_{T^{(i-1)}}} \geq \epsilon/2}.
		\end{align*}
		Using the notation $\llbracket X \rrbracket_{\cC^{\alpha}_T}$ from \eqref{eq:holder_seminorm}, if $T^{(i)} \leq T$, 
		\begin{equation*}
			\frac{\epsilon}{2} = \max\left\{\abs{X_{T^{(i)}} - X_{T^{(i-1)}}}, \abs{Y_{T^{(i)}} - Y_{T^{(i-1)}}}\right\} 
			\leq \max\left\{ \llbracket X \rrbracket_{\mathcal{C}_T^{\alpha}}, \llbracket Y \rrbracket_{\mathcal{C}_T^{\alpha}} \right\} 
			(T^{(i)} - T^{(i-1)})^{\alpha}.
		\end{equation*}
		Hence, the a priori estimate \eqref{eq:young_sol_a_priori_bd} implies that 
		\begin{align}\label{eq:T_i_unif_est}
			T^{(i)} - T^{(i-1)} \gtrsim_{\alpha, \gamma, T} \varepsilon^{-\frac{1}{\alpha}}  (1 + \norm{\sigma}_{\mathcal{C}^{\gamma}_x} \|B^H\|_{\mathcal{C}^{\alpha}_T})^{\frac{1}{\alpha^2}}			
		\end{align}
		uniformly over $i$, as long as $T^{(i)} \leq T$.
		
		To see that $X = Y$ a.s. up to time $T^{(1)}$, 
		let $\sigma^{(1)}$ be a $\gamma$-Hölder map such that $\sigma^{(1)} = \sigma$ in an $\epsilon$-neighbourhood of 
		the initial condition $x$ and such that
		\begin{align*}
			\sup_{x, y \in \mathbb{R}^{d_1}} \abs{\sigma^{(1)}(x) - \sigma^{(1)}(y)} \leq \frac{l^{-2}}{4}.
		\end{align*}  
		We set 
		\begin{align*}
			X^{(1)}_t \defby x + \int_0^t \sigma^{(1)}(X_r) \, \mathrm{d} r,
			\quad 
			Y^{(1)}_t \defby x + \int_0^t \sigma^{(1)}(Y_r) \, \mathrm{d} r
		\end{align*}
		and 
		\begin{align*}
			V^{(1), k}_n(t) \defby \int_0^t \int_0^1 \partial_k \sigma_n(\theta X^{(1)}_r + (1-\theta) Y^{(1)}_r) \, \mathrm{d} \theta \, \mathrm{d} B^H_r.
		\end{align*}
		Up to time $T^{(1)}$, we have 
		\begin{align*}
			X_t - Y_t = \lim_{n \to \infty}
			\sum_k \int_0^t (X^{ k}_r - Y^{k}_r) \, \mathrm{d} V^{(1), k}_n(r).
		\end{align*}
		Due to our choice of $\sigma^{(1)}$, the argument of Step 1 shows that $V^{(1), k}_n$ converges to some $V^{(1), k}$ 
		in $\mathcal{C}_T^{\beta}$, and that $X = Y$ up to time $T^{(1)}$.
		
		{\bfseries Step 3.}
		Obviously we want to repeat the operation of Step 2, but now there is a small problem that $X_{T^{(1)}}$ is random. 
		For this sake, let $(x_m)_{m \in \mathbb{N}}$ be a countable dense set of $\mathbb{R}^{d_1}$, and 
		let $\sigma^m$ be a $\gamma$-Hölder map such that 
		$\sigma^m = \sigma$ in an $\epsilon$-neighbourhood of $x_m$ and such that 
		\begin{align}\label{eq:sigma_m_low_osc}
			\sup_{x, y \in \mathbb{R}^{d_1}} \abs{\sigma^m(x) - \sigma^m(y)} \leq \frac{l^{-2}}{4}.
		\end{align}  
		For each $m$ and $Z \in \{X, Y\}$, 
		we define the approximation,
		\begin{equation*}
			\Sigma^{Z, m}(t)  \defby 
			\begin{cases}
				\sigma(Z_t) \indic_{\{t \leq T^{(1)}\}} +\sigma^m(Z_t) \indic_{\{t > T^{(1)}\}}, & \text{if } \abs{x_m - X_{T^{(1)}}} =  \abs{x_m - Y_{T^{(1)}}}\leq \epsilon/2,\\
				\sigma(Z_{\min\{t, T^{(1)}\}}), & \text{otherwise },
			\end{cases}
		\end{equation*}
		along with the cut-off equations,
		\begin{align*}
			X^{(2, m)}_t \defby x + \int_0^t \Sigma^{X, m}(r) \, \mathrm{d} B^H_r,\qquad Y^{(2, m)}_t \defby x + \int_0^t \Sigma^{Y, m}(r) \, \mathrm{d} B^H_r
		\end{align*}
		and approximate linearisation,
		\begin{align*}
			V^{(2, m), k}_n (t) 
			\defby \int_0^t \int_0^1 \partial_k \sigma_n(\theta X^{(2, m)}_r + (1-\theta) Y^{(2, m)}_r) \, \mathrm{d} \theta \, \mathrm{d} B^H_r.
		\end{align*}
		Notice that $X^{(2, m)}$, $Y^{(2, m)}$ are both adapted.
		We claim that, for $Z \in \{X, Y\}$, 
		\begin{align}\label{eq:Z_remainder_estimate}
			\abs{Z_{s, t}^{(2, m)} - \Sigma^{Z, m}(s) B^H_{s, t}} \lesssim 
			(\norm{\sigma}_{\mathcal{C}_x^{\gamma}} + \norm{\sigma^m}_{\mathcal{C}_x^{\gamma}}) 
			\norm{Z}_{\mathcal{C}_T^{\alpha}}^{\gamma} \left\| B^H \right\|_{\mathcal{C}_T^{\alpha}} (t-s)^{(1+ \gamma) \alpha}.
		\end{align}
		Indeed, by the remainder estimate of Young's integral \eqref{eq:young_fund_est}, 
		the estimate \eqref{eq:Z_remainder_estimate} is obvious if $t \leq T^{(1)}$ or if $s \geq T^{(1)}$. Suppose that $s < T^{(1)} < t$.
		If $\abs{x_m - X_{T^{(1)}}} \leq \epsilon/2$, then 
		\begin{align*}
			\abs{Z_{s, T^{(1)}}^{(2, m)} - \sigma(Z_s) B^H_{s, T^{(1)}}} &\lesssim 
			\norm{\sigma}_{\mathcal{C}_x^{\gamma}} 
			\norm{Z}_{\mathcal{C}_T^{\alpha}}^{\gamma} \left\| B^H \right\|_{\mathcal{C}_T^{\alpha}} (t-s)^{(1+ \gamma) \alpha}, \\
			\abs{Z_{T^{(1)}, t}^{(2, m)} - \sigma^m(Z_{T^{(1)}}) B^H_{T^{(1)}, t}} &\lesssim 
			\norm{\sigma^m}_{\mathcal{C}_x^{\gamma}} 
			\norm{Z}_{\mathcal{C}_T^{\alpha}}^{\gamma} \left\| B^H \right\|_{\mathcal{C}_T^{\alpha}} (t-s)^{(1+ \gamma) \alpha}.
		\end{align*}
		Since $\abs{x_m - X_{T^{(1)}}} = \abs{x_m - Y_{T^{(1)}}} \leq \epsilon/2$, we have 
		\begin{align*}
			\abs{\sigma^m(Z_{T^{(1)}}) - \sigma(Z_s)} = \abs{\sigma(Z_{T^{(1)}}) - \sigma(Z_s)} 
			\leq \norm{\sigma}_{\mathcal{C}_x^{\gamma}} \norm{Z}_{\mathcal{C}_T^{\alpha}}^{\gamma} (T^{(1)} - s)^{\gamma \alpha},
		\end{align*} 
		and the estimate \eqref{eq:Z_remainder_estimate} follows.
		The case where $\abs{x_m - X_{T^{(1)}}} > \epsilon/2$ is similar.
		
		Now we check that the Gubinelli derivative $\Sigma^{\theta, m} \defby \theta \Sigma^{X,m} + (1-\theta) \Sigma^{Y, m}$ is non-degenerate. Indeed, 
		for $t \leq T^{(1)}$ we have 
		\begin{align*}
			\Sigma^{\theta, m}(t) \Sigma^{\theta, m}(t)^{\top} = \sigma(X_t) \sigma(X_t)^{\top} \geq l^{-1},
		\end{align*}
		and for $t > T^{(1)}$ the condition \eqref{eq:sigma_m_low_osc} implies that 
		$\Sigma^{\theta, m}(t) \Sigma^{\theta, m}(t)^{\top} \geq l^{-1}/4$. 
		Therefore, Prop.~\ref{prop:fbm_young} and the Kolmogorov continuity theorem show that 
		$V^{(2, m)}_n$ converges to some $V^{(2, m)}$ in $L^p_{\omega}$. By the diagonalization argument, we may suppose that 
		almost surely for every $m \in \mathbb{N}$ and $\delta \in (0, \gamma H)$ we have 
		\begin{align*}
			\lim_{n \to \infty} \norm{V^{(2, m)} - V^{(2, m)}_n}_{\mathcal{C}_T^{\gamma H - \delta}} = 0.
		\end{align*}

		Since $(x_m)$ is dense, we can find a random $m = m(\omega)$ so that $\abs{x_m - X_{T^{(1)}}} < \epsilon/2$. We then have 
		\begin{align*}
			X_t - Y_t = \sum_k \int_0^t (X^k_r - Y^k_r) \, \mathrm{d} V^{(2, m), k}(r)
		\end{align*}
		up to $t \leq T^{(2)}$, hence $X = Y$ on $[0, T^{(2)}]$. 
		It is now clear that this algorithm can be continued, and at some point we must have $T^{(i)} \geq T$ due to \eqref{eq:T_i_unif_est}.
	\end{proof}
	\section{Rough Case: $H \in (1/3, 1/2]$}\label{sec:rough}
	In this section we prove Thm.~\ref{thm:main} for the regime $H\in (1/3,1/2]$. As discussed at the opening of Sec.~\ref{sec:rough_integration}, for Hurst parameters below $1/2$ we can no longer appeal to Young integration theory to make sense of  $\int_0^t \sigma(X_s)\dd B^H_s$ in \eqref{eq:main_sde}. Here we recall the fact that for $H\in (1/3,1/2]$ we can $\mbP$-a.s. canonically lift $B^H$ to a geometric rough path $(B^H,\mbB^H)$ (see Lem.~\ref{lem:canonical_lift}) as well as the definition of $\msD^\beta_{B^H}$; the space of $\beta$-regular, rough paths, controlled by $B^H$ (see Def.~\ref{def:controlled_path}). These preparations are enough to give a pathwise definition of solutions to rough differential equations (RDEs).
	\begin{definition}[Solutions to RDEs]\label{def:rough_de_sol}
		Given $T>0$, $d_1,\,d_2 \in \mbN$, $H \in (1/3,1/2]$, $(B^H(\omega),\mbB^H(\omega))$ an almost sure realisation of the canonical, geometric rough path lift for the fBm with Hurst parameter $H$ (see Lem.~\ref{lem:canonical_lift}) and a map $\sigma :\mbR^{d_1}\to \mcM(d_1,d_2)$, we say that a path $(y,y')$ controlled by $(B^H(\omega),\, \mbB^H(\omega))$ is a \emph{solution to the rough differential equation}
		\begin{equation}\label{eq:rough_eq_pw}
			\mathrm{d} y_t = \sigma(y_t) \, \mathrm{d} B_t^H(\omega), \quad y_0 = x,
		\end{equation}
		on $[0,T]$, if the path $(\sigma(y),\nabla \sigma(y)y') \in \msD^{\gamma \alpha}_{B^H(\omega)}$ for some $\alpha \in (1/3,H)$ and $\gamma >\frac{1-\alpha}{\alpha}$ and if for every $t\in [0,T]$ one has
		\begin{align*}
			y_t = x + \int_0^t \sigma(y_r) \, \mathrm{d} B_r^H(\omega),
		\end{align*}
		where the integral is understood as a rough integral as described in Sec.~\ref{sec:rough_integration}.
	\end{definition}
	Note that the assumption $(\sigma(y),\nabla \sigma(y)y') \in \msD^{\gamma\alpha}_{B^H(\omega)}$ for $\gamma>\frac{1-\alpha}{\alpha}>1$ is the minimum required to make sense of the integral in \eqref{eq:rough_eq_pw} as a rough integral. 
	If $(y, y')$ is a solution to \eqref{eq:rough_eq_pw}, the remainder estimate of rough integrals \eqref{eq:rough_fund_est} shows that 
	the map $t \mapsto \int_0^t \sigma(y_r) \, \mathrm{d} B^H_r(\omega)$ is again controlled by $B^H(\omega)$ 
	with Gubinelli derivative $\sigma(y)$.
	Due to this observation, it in fact holds that
	$(y,\sigma(y))\in \msD^{2\alpha}_{B^H(\omega)}$, 
	hence appealing to \cite[Prop.~6.4]{FH20} and \emph{true roughness} of $B^H$ one sees that for all solutions in the sense of Def.~\ref{def:rough_de_sol}, it holds that $y'=\sigma(y)$.
	
	As in the Young case, for an fBm on a given probability space $(\Omega,\mcF,\mbP)$, Def.~\ref{def:rough_de_sol} is an $\omega$-wise definition. We therefore give the following probabilistic definitions, which are analogues of Def.~\ref{def:weak_young_sol}, \ref{def:strong_young_sol} and \ref{def:pathwise_uniq_young} for solutions to
	\begin{align}\label{eq:main_sde_rough}
		\mathrm{d} X_t = \sigma(X_t) \, \mathrm{d} B^H_t, \quad X_0 = x.
	\end{align}
	\begin{definition}[Weak-Rough Solutions]\label{def:weak_rough_sol}
		Given $T>0$, say that a tuple
		\begin{equation*}
			(\Omega,\mcF,(\mcF_t)_{t\in [0,T]}, \mbP,B^H,X,X'),
		\end{equation*}
		of a filtered probability space, an fBm of Hurst parameter $H$ and an $(\mathcal{F}_t)$-adapted, controlled rough path, $(X,X')$ is a weak-rough solution to \eqref{eq:main_sde_rough}, if $B^H$ is an $(\mcF_t)_{t\geq 0}$ fBm (in the sense of Def.~\ref{def:adapted_fBm}) and for $\mbP$-a.e $\omega \in \Omega$,  the process $(X(\omega),X'(\omega))$ is a rough solution to \eqref{eq:main_sde_rough}in the sense of Def.~\ref{def:rough_de_sol}.
		
		We say that weak uniqueness holds for \eqref{eq:main_sde_rough} in the rough sense, if for any two weak solutions $(\Omega,\mcF,(\mcF_t)_{t\geq 0}, \mbP,B^H,X,X')$ and $(\tilde{\Omega},\tilde{\mcF},(\tilde{\mcF_t})_{t\geq 0}, \tilde{\mbP},\tilde{B}^H,\tilde{X},\tilde{X}')$ as above, the processes $(X,X'),\, (\tilde{X},\tilde{X}')$ are equal in law, i.e. one has 
		\begin{equation*}
			(X,X')_{\#} \mbP = (\tilde{X},\tilde{X}')_{\#}\tilde{\mbP}  \quad \text{as elements of}\quad  \mcP\left(C([0,T];\mbR^{d_1})\times C\left([0,T];\mcM(d_1,d_2)\right)\right)
		\end{equation*}
	\end{definition}
	\begin{remark}\label{rem:weak_rough_existence_compactness}
		As mentioned in the Young case (Rem.~\ref{rem:weak_young_existence_compactness}) a standard compactness argument (a la \cite[Appendix~B]{matsuda22}) allows one to construct weak-rough solutions in the case $\sigma\in \mcC^{\gamma}(\mbR^{d_1};\mcM(d_1,d_2))$ for $\gamma>\frac{1-H}{H}$. Note that in contrast to Def.~\ref{def:rough_de_sol} this is explicitly an assumption on the coefficient $\sigma$ which in turn ensures the required regularity of the controlled path $(\sigma(X),\nabla \sigma(X),X')$ for some $\alpha \in (1/3,H)$ sufficiently close to $H$.
		
		In the regime $H\in (1/3,1/2)$ our main result (Thm.~\ref{thm:main}) obtains pathwise uniqueness for $\sigma$  that is $\gamma$-H\"older continuous for $\gamma >\frac{1-H}{H}$ and uniformly elliptic so that we obtain pathwise uniqueness for these weak solutions (under the assumption of uniform ellipticity) and hence strong existence due to Prop.~\ref{prop:yamada_watanabe}.
	\end{remark}
	\begin{definition}[Strong-Rough Solutions]\label{def:strong_rough_sol}
		Given $T>0$, a probability space, $(\Omega,\mcF,\mbP)$ carrying an fBm $B^H$, we say that $(X,X')$ is a strong-rough solution to \eqref{eq:main_sde_rough} on $[0,T]$ if for $\mbP$-a.e. $\omega\in \Omega$, the process $(X(\omega),X'(\omega)$ is a rough solution to \eqref{eq:main_sde_rough} in the sense of Def.~\ref{def:rough_de_sol}, on $[0,T]$ and $X$ is adapted to the natural filtration generated by $B^H$.
	\end{definition}
	\begin{definition}[Pathwise Uniqueness of Rough Solutions]\label{def:pathwise_uniq_rough}
		We say that \emph{pathwise-uniqueness} holds for \eqref{eq:main_sde_rough} in the rough sense, if for any two, weak-rough solutions, $(X,X')$ and $(\tilde{X},\tilde{X}')$ (in the sense of Def.~\ref{def:weak_rough_sol}), on the same filtered probability space $(\Omega,\mcF,(\mcF_t)_{t\in [0,T]},\mbP)$, it holds that $(X(\omega),X'(\omega))=(\tilde{X}(\omega),\tilde{X}'(\omega))$ for $\mbP$-a.e. $\omega \in \Omega$.
	\end{definition}
	The goal of this section is to prove pathwise uniqueness for \eqref{eq:main_sde_rough} under the assumption that $\sigma$ is uniformly elliptic and 
	$\gamma > (1-H)/H$. 
	Namely, we prove Thm.~\ref{thm:main} in the rough case, $H \in (1/3, 1/2]$, which is restated as Thm.~\ref{thm:pathwise_uniqueness_rough} below.
	%
	\subsection{Rough Integral with Irregular Integrand}\label{subsec:rough_int_irreg}
	As in Sec.~\ref{sec:young}, the key to prove pathwise uniqueness is to obtain a sharp stochastic estimate on 
	stochastic integral of the form \eqref{eq:int_f_young}.
	Differently from the Young case, however, we also require suitable estimates on iterated integrals of the form \eqref{eq:B_G_lift}. The iterated integrals are treated in Sec.~\ref{subsec:rough_iterated} below.
	
	To treat the first order integrals we apply the stochastic sewing lemma (Lem.~\ref{lem:ssl}). 
	We will need the following technical estimate, whose proof will be given in App.~\ref{sec:gaussian}.
	\begin{lemma}\label{lem:integral_along_tilde_B}
		Let $T>0$, $H \in (1/3, 1/2]$ and $B^H$ be an fBm adapted to a given filtration $(\mathcal{F}_t)_{t\in [0,T]}$. 
		We also fix a $\tau \in [0, T]$, define $\tilde{B}^{H, \tau}$ by  \eqref{eq:def_tilde_B} and let $\xi$ be an $\mathcal{F}_{\tau}$-measurable random variable valued in a measurable space $E$. Then, given a map $F \from E \times \mathbb{R}^{d_2} \to \mathbb{R}$ such that  
		$\sup_{x} \norm{F(x, \cdot)}_{C^{2}_y} < \infty$, 
		one has, for any $\gamma \in (\frac{1}{2H} - 1, 1)$ and $(s, t) \in [\tau, T]_{\leq}^2$ that
		\begin{multline*}
			\Big\Vert \int_s^t F(\xi, B^{H}_{\tau, r}) \, \mathrm{d} \tilde{B}^{H, \tau}_r 
			- \frac{F(\xi,  B_{\tau, s}^H) + F(\xi, B_{\tau, t}^H)}{2} \tilde{B}^{H, \tau}_{s, t} \Big\Vert_{L^p_{\omega}} \\
			\lesssim_{p, \gamma} \norm{\norm{F(\xi, \cdot)}_{\mathcal{C}^{\gamma}_y}}_{L^{2p}_{\omega}} (t-s)^{(1+\gamma)H}.
		\end{multline*}
		The integral with respect to $\tilde{B}^{H, \tau}$ is understood as rough integral with respect to the geometric rough path 
		constructed in Lem.~\ref{lem:tilde_B_rp}.
	\end{lemma}
	\begin{proof}
		See App.~\ref{subsec:proof_lem:integral_along_tilde_B}.
	\end{proof}
	It is worth mentioning that the proof of Lem.~\ref{lem:integral_along_tilde_B} is based on the fully shifted stochastic sewing lemma, 
	in the spirit of \cite[Prop.~3.5]{matsuda22}.
	The rest of this subsection is devoted to proving the following proposition, which is the analogue of Prop.~\ref{prop:fbm_young}.
	We recall, again, the definition of the stochastically controlled, path space, $\mathfrak{D}^{1, l}_{B^H}$, and associated norm, $\opnorm{\,\cdot\,}_{p, \beta}$, from Def.~\ref{def:control_prob}.
	\begin{proposition}\label{prop:fbm_rough}
		Let $T>0$, $d_1,\,d_2 \in \mbN$, $B^H$ be an $\mbR^{d_2}$-fBm with Hurst parameter $H \in (1/3, 1/2]$, adapted to a given filtration $(\mcF_t)_{t\in [0,T]}$, $p \in [2, \infty)$, $l \in [1, \infty)$, $(X, X') \in \mathfrak{D}^{1, l}_{B^H}$, 
		$f \in C^3(\mathbb{R}^{d_1}, \mathbb{R})$ and $\tilde{\gamma} \in \left( \frac{1-H}{H}-1, 1\right)$. 
		Then, for $(s, t) \in [0, T\wedge 1]_{\leq}^2$ it holds that
		\begin{equation*}
			\norm[\Big]{\int_s^t f(X_r) \, \mathrm{d} B_r^H - f(X_s) B_{s, t}^H}_{L^p_{\omega}} 
			\lesssim_{p, l, \tilde{\gamma}, T}  
			\norm{f}_{\mathcal{C}_x^{\tilde{\gamma}}} (1 + \opnorm{X}_{2p, 1})  (t-s)^{(1+\tilde{\gamma}) H}.
		\end{equation*}
	\end{proposition}
	\begin{proof}
		For concision, throughout the proof we simply write $B \defby B^H$. Let us set
		\begin{align}\label{eq:germ_rough_integral}
			A_{s,t} \defby \int_s^t f(X_s + X_s' B_{s, r}) \, \mathrm{d} B_r.
		\end{align}
		The integral is understood as rough integral: the path
		$ r \mapsto f(X_s + X_s' B_{s, r})$
		is controlled by $B$ with Gubinelli derivative $\nabla f(X_s + X_s' B_{s, r}) X'_s$.
		In particular, the estimate \eqref{eq:rough_fund_est} yields
		\begin{align*}
			\abs{A_{s, t} - f(X_s) B_{s, t} - \nabla f(X_s) X'_s \mathbb{B}_{s, t}}
			\lesssim (t-s)^{3 \alpha}
		\end{align*}
		for $\alpha \in ( 1/3, H)$. Hence, the uniqueness part of the sewing lemma (see Rem.~\ref{rem:sewing_uniq}) implies that
		\begin{align*}
			\int_s^t f(X_r) \, \mathrm{d} B_r = 
			\lim_{\abs{\pi} \to 0 } \sum_{[s', t'] \in \pi} A_{s', t'},
		\end{align*}
		where $\pi$ is a partition of $[s, t]$.
		As in Prop.~\ref{prop:fbm_young}, we apply the stochastic sewing to 
		$A_{s, t}$. This time, Lê's original version \cite{le20} is sufficient (that is $\alpha = 0$ and $v = s$ 
		in Lem.~\ref{lem:ssl}).
		Precisely, we prove the estimates 
		\begin{align*}
			\norm{\delta A_{s, u, t}}_{L^p_{\omega}} \leq \Gamma_1 (t-s)^{\frac{1}{2} + \varepsilon_1}, 
			\quad \norm{\mathbb{E}[\delta A_{s, u, t} \vert \mathcal{F}_s]}_{L^p_{\omega}} \leq \Gamma_2 (t-s)^{1 + \varepsilon_2}
		\end{align*}
		for $s < u < t$ with appropriate $\varepsilon_1, \varepsilon_2 > 0$ and proportional constants $\Gamma_1, \Gamma_2$.
		We split the rest of the proof into three steps.
		
		{\bfseries Step 1.} We estimate $\norm{\delta A_{s, u, t}}_{L^p_{\omega}}$. 
		Writing $Y \defby Y^{H, u}$ and $\tilde{B} \defby \tilde{B}^{H, u}$ as in \eqref{eq:def_Y} and \eqref{eq:def_tilde_B} respectively, 
		we observe 
		\begin{align}
			\delta A_{s, u, t} \notag
			=& \int_u^t \left( f(X_s + X_s' B_{s, r}) - f(X_u + X_u' B_{u, r})\right) \, \mathrm{d} B_r \\
			=& \int_u^t \left( f(X_s + X_s' B_{s, r}) - f(X_u + X_u' B_{u, r})\right) \dot{Y_r} \, \mathrm{d}r 
			\label{eq:f_integral_along_Y}\\
			&+ \int_u^t \left( f(X_s + X_s' B_{s, r}) - f(X_u + X_u' B_{u, r})\right) \, \mathrm{d} \tilde{B}_r,
			\label{eq:f_integral_along_tilde_B}
		\end{align}
		The integral \eqref{eq:f_integral_along_Y}, by the triangle inequality, is bounded by 
		\begin{align*}
			\norm{f}_{\mathcal{C}_x^{\tilde{\gamma}}} \int_u^t \abs{X_s + X_s' B_{s, r} - (X_u + X_u' B_{u, r})}^{\tilde{\gamma}} 
			\abs{\dot{Y}_r} \, \mathrm{d} r.
		\end{align*}
		Defining the remainder $R$ as in Def.~\ref{def:control_prob}, we have 
		\begin{align*}
			X_s + X_s' B_{s, r} - (X_u + X_u' B_{u, r}) 
			= -R_{s, u} - X'_{s, u} B_{u, r},
		\end{align*}
		and since, due to Lem.~\ref{lem:Y_dot} we have $\norm{\dot{Y}_r}_{L^{2p}_{\omega}} \lesssim_{p} (r - u)^{H-1}$, it follows that
		\begin{multline*}
			\Bigg\|\int_u^t \left( f(X_s + X_s' B_{s, r}) - f(X_u + X_u' B_{s, r})\right) \dot{Y_r} \, \mathrm{d}r\,\Bigg\|_{L^p_{\omega}} \\
			\leq \norm{f}_{\mathcal{C}_x^{\tilde{\gamma}}} \int_u^t \norm{\abs{R_{s, u} + X'_{s, u} B_{u, r}}^{\tilde{\gamma}}}_{L^{2p}_{\omega}} 
			\norm{\dot{Y}_r}_{L^{2p}_{\omega}} \, \mathrm{d}r 
			\lesssim_p \norm{f}_{\mathcal{C}_x^{\tilde{\gamma}}}
			\opnorm{X}_{2p, 1}^{\tilde{\gamma}}  (t-s)^{(1 + 2\tilde{\gamma})H}.
		\end{multline*}
		To estimate \eqref{eq:f_integral_along_tilde_B},  
		for $v \in \{s, u\}$ we apply Lem.~\ref{lem:integral_along_tilde_B} with 
		\begin{align*}
			\tau = u, \quad \xi = (X_v + X_v' B_{v, u}, X_v'), \quad F((x, x'), y) = f(x + x' y)
		\end{align*}
		to obtain the estimate
		\begin{multline*}
			\norm[\Big]{\int_u^t  f(X_v + X_v' B_{v, r})  \, \mathrm{d} \tilde{B}_r - 
				\frac{f(X_v + X'_v B_{v, u}) + f(X_v + X'_v B_{v, t})}{2} \tilde{B}_{u, t}}_{L^p_{\omega}}  \\
			\lesssim_{p, l, \tilde{\gamma}} \norm{f}_{\mathcal{C}_x^{\tilde{\gamma}}} (t-s)^{(1 + \tilde{\gamma}) H}.
		\end{multline*}
		Since 
		\begin{multline*}
			\abs[\Big]{\frac{f(X_s + X'_s B_{s, u}) + f(X_s + X'_s B_{s, t})}{2} 
				- \frac{f(X_u) + f(X_u + X'_u B_{u, t})}{2} } \\
			\lesssim \norm{f}_{\mathcal{C}_x^{\tilde{\gamma}}}
			\big( \abs{X_{s,u} - X_s' B_{s, u}}^{\tilde{\gamma}} + \abs{X'_{s,u} B_{u, t}}^{\tilde{\gamma}}\big),
		\end{multline*}
		we see that the $L^p_{\omega}$-norm of the integral \eqref{eq:f_integral_along_tilde_B} 
		is bounded by (up to constant) 
		\begin{align*}
			\norm{f}_{\mathcal{C}_x^{\tilde{\gamma}}} (t-s)^{(1+\tilde{\gamma})H}
			+\norm{f}_{\mathcal{C}_x^{\tilde{\gamma}}} \opnorm{X}_{2p, 1}^{\tilde{\gamma}}  (t-s)^{(1+2 \tilde{\gamma})H}.
		\end{align*}
		Hence, 
		\begin{align}\label{eq:fbm_rough_int_ssl_1}
			\norm{\delta A_{s, u, t}}_{L^p_{\omega}} 
			\lesssim_{p, l, \tilde{\gamma}} 
			\norm{f}_{\mathcal{C}_x^{\tilde{\gamma}}} (t-s)^{(1+\tilde{\gamma})H}
			+\norm{f}_{\mathcal{C}_x^{\tilde{\gamma}}} \opnorm{X}_{2p, 1}^{\tilde{\gamma}}  (t-s)^{(1+2 \tilde{\gamma})H}.
		\end{align}

		{\bfseries Step 2.} Next we estimate $\mathbb{E}[\delta A_{s, u, t} \vert \mathcal{F}_s]$, 
		which is essentially done in Prop.~\ref{prop:fbm_young}.
		We define $Y$ and $\tilde{B}$ as in Step 1, and we write $\mathbb{E}_u[\,\cdot\,] \defby \mathbb{E}[\,\cdot\, \vert \mathcal{F}_u]$. We have
		\begin{align*}
			\mathbb{E}_{u} \left[\delta A_{s, u, t} \right]=& \mathbb{E}_{u}\left[ \int_u^t \left( f(X_s + X_s' B_{s, r}) - f(X_u + X_u' B_{s, r})\right) \dot{Y_r} \, \mathrm{d}r\right] \\
			&+ \mathbb{E}_u \left[\int_u^t \left( f(X_s + X_s' B_{s, r}) - f(X_u + X_u' B_{s, r})\right) \, \mathrm{d} \tilde{B}_r\right]. 
		\end{align*}
		For $v \in \{s, u\}$, we observe
		\begin{equation*}
			\mathbb{E}_{u} \left[\int_u^t  f(X_v + X_v' B_{v, r})  \dot{Y_r} \, \mathrm{d}r \right]\\
			= \int_u^t \cP_{\rho^2(u, r) X'_{v}(X'_{v})^{\top}} f
			(X_{v} + X'_{v}(B_{v, u} + Y_{u, r})) \dot{Y}_r \, \mathrm{d} r
		\end{equation*}
		with $\rho^2(u, r) \defby \rho^2_H(u, r)$ defined by \eqref{eq:def_rho}.
		As in the proof of Prop.~\ref{prop:fbm_young} (especially \eqref{eq:g_along_B_expectation}), Lem.~\ref{lem:gaussian_integral} implies that 
		\begin{multline*}
			\mathbb{E}_u \left[\int_u^t f(X_v + X_v' B_{v, r})  \, \mathrm{d} \tilde{B}_r \right]\\
			=
			\frac{1}{2} X'_{v} \int_u^t 
			\nabla \cP_{\rho^2(u, r) X'_{v}(X'_{v})^{\top}} f
			(X_{v} + X'_{v}(B_{v, u} + Y_{u, r})) \, \rho^2(u, \mathrm{d} r).
		\end{multline*}
		Hence, $\mathbb{E}_u \left[\delta A_{s, u, t}\right] = J_1 + J_2$, where
		\begin{multline*}
			J_1 \defby \int_u^t \big( \cP_{\rho^2(u, r) X'_{s}(X'_{s})^{\top}} 
			f(X_{s} + X'_{s}(B_{s, u} + Y_{u, r})) \\
			- \cP_{\rho^2(u, r) X'_{u}(X'_{u})^{\top}}f(X_{u} + X'_{u}Y_{u, r}) \big) \dot{Y}_r \, \mathrm{d} r,
		\end{multline*}
		\begin{multline*}
			J_2 \defby 
			\frac{1}{2}  \int_u^t   
			\big (
			X'_{s} \nabla \cP_{\rho^2(u, r) X'_{s}(X'_{s})^{\top}} f
			(X_{s} + X'_{s}(B_{s, u} + Y_{u, r})) \\ 
			- X'_{u} \nabla \cP_{\rho^2(u, r) X'_{u}(X'_{u})^{\top}} f
			(X_{u} + X'_{u} Y_{u, r}) \big) \, \rho^2(u, \mathrm{d} r).
		\end{multline*}
		Then the rest of the argument is identical to the proof of Prop.~\ref{prop:fbm_young}. ($J_1$ corresponds to $I_{3,1}$ and $J_2$ to $I_{3,2}$.) 
		In particular, we obtain 
		\begin{align}\label{eq:fbm_rough_int_ssl_2}
			\norm{\mathbb{E}_s[ \delta A_{s, u, t}]}_{L^p_{\omega}}
			\lesssim_{p, l, \tilde{\gamma}, T} \norm{f}_{\mathcal{C}^{\tilde{\gamma}}} \opnorm{X}_{2p, 1} (t-s)^{(2 + \tilde{\gamma}) H}.
		\end{align}

		{\bfseries Step 3.} Thanks to the estimates \eqref{eq:fbm_rough_int_ssl_1} and \eqref{eq:fbm_rough_int_ssl_2},  
		we are in the setting of Lem.~\ref{lem:ssl} with $\alpha = 0$ and $v=s$. 
		Therefore, the remainder estimate of Lem.~\ref{lem:ssl} yields
		\begin{equation}\label{eq:fbm_rough_int_vs_A}
			\norm[\Big]{\int_s^t f(X_r) \, \mathrm{d} B_r - A_{s, t}}_{L^p_{\omega}} 
			\lesssim_{p, l, \tilde{\gamma}, T} \norm{f}_{\mathcal{C}_x^{\tilde{\gamma}}} (1 + \opnorm{X}_{2p,1}) (t-s)^{(1 + \tilde{\gamma}) H}.   
		\end{equation}
		On the other hand, we have 
		\begin{multline*}
			A_{s, t} - \frac{f(X_s) + f(X_s + X_s' B_{s, t})}{2} B_{s, t} \\ 
			= \int_s^t \left( f(X_s + X_s' B_{s, r}) - \frac{f(X_s) + f\left(X_s + X_s' B_{s, t}\right)}{2} \right) \dot{Y}_r^{s} \, \mathrm{d} r \\
			+ \int_s^t  f(X_s + X_s' B_{s, r}) \, \mathrm{d} \tilde{B}_r^s -  \frac{f(X_s) + f(X_s + X_s' B_{s, r})}{2} \tilde{B}_{s, t}^s.
		\end{multline*}
		By Lem.~\ref{lem:Y_dot}, the first term is bounded by $\norm{f}_{\mathcal{C}^{\tilde{\gamma}}_x} (t-s)^{2H}$ up to constant, while 
		by Lem.~\ref{lem:integral_along_tilde_B} the second term is bounded by  $\norm{f}_{\mathcal{C}^{\tilde{\gamma}}_x} (t-s)^{(1+\tilde{\gamma})H}$ 
		up to constant. Therefore, since $\tilde{\gamma}<1$ and $|t-s|\leq 1$,
		\begin{align}\label{eq:A_s_t_str}
			\norm[\Big]{A_{s, t} - \frac{f(X_s) + f(X_s + X_s' B_{s, t})}{2} B_{s, t}}_{L^p_{\omega}} \lesssim_{p, l, \tilde{\gamma}, T} 
			\norm{f}_{\mathcal{C}_x^{\tilde{\gamma}}} (t-s)^{(1+\tilde{\gamma})H}.
		\end{align}
		Finally, we observe 
		\begin{equation}\label{eq:str_vs_ito}
			\norm[\Big]{\frac{f(X_s) + f(X_s + X_s' B_{s, t})}{2} B_{s, t} - f(X_s) B_{s, t}}_{L^p_{\omega}}
			\lesssim_{p, l, \tilde{\gamma}, T} \norm{f}_{\mathcal{C}_x^{\tilde{\gamma}}} (t-s)^{(1 + \tilde{\gamma})H}. 
		\end{equation}
		Combining \eqref{eq:fbm_rough_int_vs_A}, \eqref{eq:A_s_t_str} and \eqref{eq:str_vs_ito}, we obtain the desired estimate.
	\end{proof}
	\subsection{Iterated Integrals with Rough Integrands}\label{subsec:rough_iterated}
	Referring back to the strategy described for the rough case in Sec.~\ref{subsec:proof_overview}, as compared to the Young case, we also require sufficient control on iterated integrals of the form described by \eqref{eq:B_G_lift}. In this section (over the following two propositions) we treat the case when the $(G^k)_{k=1}^{d_2}$ invoked therein is a generic path of the form
	\begin{equation*}
		G^k_{t} \coloneqq \int_0^t f(X_r)\dd B^H_r,\quad k =1 ,\ldots, d_2,
	\end{equation*}
	for $f$ as in Prop.~\ref{prop:fbm_rough}. In Prop.~\ref{prop:rough_iterated_integral}, we treat iterated integrals of the form $\int G^k \otimes \, \mathrm{d} B^H$ 
	(which in turn gives the estimate on $\int B^H \otimes \, \mathrm{d} G^k$ by integration by parts), 
	and in Prop.~\ref{prop:rough_itr_int_2} we treat iterated integrals of the form $\int G^k \otimes \, \mathrm{d} G^l$.
	These two propositions as well as Prop.~\ref{prop:fbm_rough} will be used in Sec.~\ref{subsec:pw_uniq_rough} below 
	to construct the rough path lift \eqref{eq:B_G_lift} by proving the convergence of rough paths lifts associated to smooth approximations of $\sigma$.

	Since we already have access to Prop.~\ref{prop:fbm_rough}, the estimates on iterated integrals 
	follows from the deterministic sewing lemma (Lem.~\ref{lem:det_sewing}).
	We recall again the notations $\mathfrak{D}^{1, l}_{B^H}$ and $\opnorm{\,\cdot\,}_{p, 1}$ from Def.~\ref{def:control_prob}.
	\begin{proposition}\label{prop:rough_iterated_integral}
		Let $H \in (1/3, 1/2]$, $p \in [2, \infty)$, $l \in [1, \infty)$, $(X, X') \in \mathfrak{D}^{1, l}_{B^H}$, 
		$f \in C^3(\mathbb{R}^{d_1})$ and $\tilde{\gamma} \in\left(\frac{1-H}{H} -1, 1\right)$. 
		Then, for $(s, t) \in [0, T]_{\leq}^2$ and $i,j \in \{1, 2, \ldots, d_2\}$ we have 
		\begin{equation*}
			\norm[\Big]{\int_s^t f(X_r) B^{H, i}_{s, r} \, \mathrm{d} B^{H, j}_r - f(X_s) \mathbb{B}^{H, i, j}_{s, t}}_{L^p_{\omega}} 
			\lesssim_{p, l, \tilde{\gamma}, T}
			\norm{f}_{\mathcal{C}_x^{\tilde{\gamma}}} (1 + \opnorm{X}_{4p, 1}) (t-s)^{(2+\tilde{\gamma})H}.
		\end{equation*}
	\end{proposition}
	\begin{proof}
		For concision, we write $B \defby B^H$ and $\mathbb{B} \defby \mathbb{B}^H$.
		We fix a $\tau \leq s$, and we set 
		\begin{equation*}
			\bar{A}_{s, t} \defby B^i_{\tau, s} \int_s^t f(X_r) \, \mathrm{d} B^j_r + f(X_s) \mathbb{B}_{s, t}^{i, j}.
		\end{equation*}
		{\bfseries Step 1.} We show that $\bar{A}_{s, t}$ is a correct approximation to the iterated integral 
		$\int_s^t f(X_r) B_{\tau, r}^i \, \mathrm{d} B_r^j$.
		For this sake, we set 
		\begin{equation*}
			\bar{A}'_{s, t} \defby f(X_s) B^i_{\tau, s} B^j_{s, t} \\
			+ f(X_s) \mathbb{B}^{i,j}_{s, t}
			+ \sum_{k, l}  B^i_{\tau, s} \partial_k f (X_s) (X'_s)^{k l} \mathbb{B}_{s, t}^{l j}.
		\end{equation*}
		By definition of the rough integral, we have 
		\begin{align*}
			\int_s^t  f(X_r) B^i_{\tau, r} \, \mathrm{d} B^j_{r}
			= \lim_{\substack{\pi \text{ is a partition of } [s, t], \\ \abs{\pi} \to 0}} \sum_{[u, v] \in \pi} \bar{A}'_{u, v}.
		\end{align*}
		In view of the uniqueness part of the sewing lemma (see Rem.~\ref{rem:sewing_uniq}), 
		our goal is to prove $\abs{\bar{A}_{s, t} - \bar{A}'_{s, t}} \lesssim_{f, X, X', \omega} (t-s)^{3H}$.
		We compute 
		\begin{align*}
			\bar{A}_{s, t} - \bar{A}'_{s, t}
			= B^i_{\tau, s} \Big( \int_s^t f(X_r) \, \mathrm{d} B_r -  f(X_s)  B^j_{s, t} - \sum_{k, l} \partial_k f (X_s) (X'_s)^{k l} \mathbb{B}_{s, t}^{l j}\Big)
		\end{align*}
		By the remainder estimate of the rough integral \eqref{eq:rough_fund_est}, we have 
		\begin{align*}
			\abs[\Big]{\int_s^t f(X_r) \, \mathrm{d} B_r -  f(X_s)  B^j_{s, t} - \sum_{k, l} \partial_k f (X_s) (X'_s)^{k l} \mathbb{B}_{s, t}^{l j}}
			\lesssim_{f, X, X', \omega} (t-s)^{3H},
		\end{align*}
		whence $\bar{A}$ is a correct approximation.
		
		{\bfseries Step 2.} This time we do not need the stochastic sewing; the deterministic sewing (Lem.~\ref{lem:det_sewing}) 
		in $L^p_{\omega}$ suffices. We compute
		\begin{align*}
			\delta \bar{A}_{s, u, t}  = - B^i_{s, u} \int_u^t f(X_r) \, \mathrm{d} B^j_r + f(X_s) (\mathbb{B}_{s, t}^{i, j} - \mathbb{B}_{s, u}^{i, j}) - f(X_u) \mathbb{B}_{u, t}^{i, j}.
		\end{align*}
		Using Chen's relation 
		$\mathbb{B}^{i, j}_{u, t} = \mathbb{B}^{i, j}_{s, t} - \mathbb{B}^{i, j}_{s, u} - B^i_{s, u} B^j_{u, t}$,
		we obtain
		\begin{align*}
			\delta A_{s, u, t} = - B^i_{s, u} \Big( \int_u^t f(X_r) \, \mathrm{d} B_r^j - f(X_u) B^j_{u, t} \Big)
			- f(X)_{s, u} (\mathbb{B}^{i, j}_{s, t} - \mathbb{B}^{i, j}_{s, u}).
		\end{align*}
		Since 
		\begin{equation*}
			\norm{B^i_{s, u}}_{L^p_{\omega}} \lesssim_p (t-s)^H, \,\, \norm{\mathbb{B}^{i, j}_{s, t}}_{L^p_{\omega}} \lesssim_p (t-s)^{2H}, 
			\,\,
			\norm{f(X)_{s, u}}_{L^p_{\omega}} \lesssim_{p, \tilde{\gamma}} \norm{f}_{\mathcal{C}_x^{\tilde{\gamma}}} \opnorm{X}_{p \tilde{\gamma}, 1}^{\tilde{\gamma}} (t-s)^{\tilde{\gamma} H},
		\end{equation*}
		and by Prop.~\ref{prop:fbm_rough} 
		\begin{align*}
			\norm[\Big]{\int_u^t f(X_r) \, \mathrm{d} B_r^j - f(X_u) B^j_{u, t}}_{L^{2p}_{\omega}}
			\lesssim_{p, l, \tilde{\gamma}, T} \norm{f}_{\mathcal{C}_x^{\tilde{\gamma}}} (1 + \opnorm{X}_{4p, 1}) (t-s)^{(1+\tilde{\gamma})H},
		\end{align*}
		using Cauchy--Schwarz inequality we obtain 
		\begin{align*}
			\norm{\delta A_{s, u, t}}_{L^p_{\omega}} \lesssim_{p, l, \tilde{\gamma}, T}  
			\norm{f}_{\mathcal{C}_x^{\tilde{\gamma}}} (1 + \opnorm{X}_{4p, 1}) (t-s)^{(2+\tilde{\gamma})H}.
		\end{align*}
		As $(2 + \tilde{\gamma}) H > 1$, the sewing lemma in $L^p_{\omega}$ gives  
		\begin{equation*}
			\norm[\Big]{\int_s^t f(X_r) B_{\tau, r}^i \, \mathrm{d} B_r^j - \bar{A}_{s, t}}_{L^p_{\omega}}  
			\lesssim_{p, l, \tilde{\gamma}, T}
			\norm{f}_{\mathcal{C}_x^{\tilde{\gamma}}} (1 + \opnorm{X}_{4p, 1}) (t-s)^{(2+\tilde{\gamma})H}.
		\end{equation*}
		It remains to set $\tau = s$.
	\end{proof}
	\begin{proposition}\label{prop:rough_itr_int_2}
		Let $H \in (1/3, 1/2]$, $p \in [2, \infty)$, $l \in [1, \infty)$, $(X, X'), (Y, Y') \in \mathfrak{D}^{1, l}_{B^H}$, 
		$f, g\in C^3(\mathbb{R}^d)$ and $\tilde{\gamma} \in \left(\frac{1-H}{H} - 1, 1\right)$. 
		Then, for every $(s, t) \in [0, T]_{\leq}^2$ and $i, j \in \{1, 2, \ldots, d_2\}$ 
		we have 
		\begin{multline*}
			\norm[\Big]{\int_s^t \Big(\int_s^{r_2} g(Y_{r_1}) \, \mathrm{d} B_{r_1}^{H, i} \Big) f(X_{r_2}) \, \mathrm{d} B^{H, j}_{r_2} 
				- g(Y_s) f(X_s) \mathbb{B}_{s, t}^{H, i, j}}_{L^p_{\omega}}  \\
			\lesssim_{p, l, \tilde{\gamma}, T}\norm{f}_{\mathcal{C}_x^{\tilde{\gamma}}}
			\norm{g}_{\mathcal{C}_x^{\tilde{\gamma}}}  
			(1 + \opnorm{X}_{8p, 1})(1 + \opnorm{Y}_{8p, 1}) (t-s)^{(2+\tilde{\gamma})H}.
		\end{multline*}
	\end{proposition}
	\begin{proof}
		For concision, throughout the proof we write $B \defby B^H$ and $\mathbb{B} \defby \mathbb{B}^H$.
		We fix a $\tau \in [0, s]$. To estimate the iterated integral 
		\begin{align*}
			\int_s^t \Big(\int_{\tau}^{r_2} g(Y_{r_1}) \, \mathrm{d} B_{r_1}^i \Big) f(X_{r_2}) \, \mathrm{d} B^j_{r_2}
		\end{align*}
		our choice of the germ is 
		\begin{equation*}
			\tilde{A}_{s, t} \defby \Big( \int_{\tau}^s g(Y_r) \, \mathrm{d} B^i_r \Big)  \int_s^t f(X_r) \, \mathrm{d} B^j_r 
			+ g(Y_s) \int_s^t B^i_{s, r} f(X_r) \, \mathrm{d} B^j_r.
		\end{equation*}
		{\bfseries Step 1.} To see that $\tilde{A}_{s, t}$ is a correct approximation, we set 
		\begin{multline*}
			\tilde{A}'_{s, t} \defby \Big( \int_{\tau}^s g(Y_r) \, \mathrm{d} B_r^i  \Big) f(X_s) B^j_{s, t} 
			+ g(Y_s) f(X_s) \mathbb{B}^{i, j}_{s, t} \\
			+ \sum_{k, l} \Big( \int_{\tau}^s g(Y_r) \, \mathrm{d} B^i_r \Big) \partial_k f (Y_s) (Y'_s)^{k l} \mathbb{B}_{s, t}^{l, j}.
		\end{multline*}
		By definition of the rough integral, we have 
		\begin{align*}
			\int_s^t \Big(\int_{\tau}^{r_2} g(Y_{r_1}) \, \mathrm{d} B_{r_1}^i \Big) f(X_{r_2}) \, \mathrm{d} B^j_{r_2}
			= \lim_{\substack{\pi \text{ is a partition of } [s, t], \\ \abs{\pi} \to 0}} \sum_{[u, v] \in \pi} \tilde{A}'_{u, v}.
		\end{align*}
		In view of the uniqueness part of the sewing lemma (see Rem.~\ref{rem:sewing_uniq}), 
		our goal is to prove $\abs{\tilde{A}_{s, t} - \tilde{A}'_{s, t}} \lesssim (t-s)^{3H}$. We observe 
		\begin{multline*}
			\tilde{A}_{s, t} - \tilde{A}'_{s, t} 
			= \Big( \int_{\tau}^s g(Y_r) \, \mathrm{d} B_r^i \Big)
			\Big( \int_s^t f(X_r) \, \mathrm{d} B^j_r - f(X_s) B^j_{s, t} -  \sum_{k, l} \partial_k f (Y_s) (Y'_s)^{k l} \mathbb{B}_{s, t}^{l, j}\Big) \\
			+ g(Y_s) \Big( \int_s^t B^i_{s, r} f(X_r) \, \mathrm{d} B^j_r - f(X_s) \mathbb{B}_{s, r}^{i, j} \Big).
		\end{multline*}
		By the remainder estimate of the rough integral \eqref{eq:rough_fund_est}, we have 
		\begin{gather*}
			\abs[\Big]{\int_s^t f(X_r) \, \mathrm{d} B^j_r - f(X_s) B^j_{s, t} -  \sum_{k, l} \partial_k f (Y_s) (Y'_s)^{k l} \mathbb{B}_{s, t}^{l, j}} \lesssim (t-s)^{3H}, \\
			\abs[\Big]{\int_s^t B^i_{s, r} f(X_r) \, \mathrm{d} B^j_r - f(X_s) \mathbb{B}_{s, r}^{i, j}} \lesssim (t-s)^{3H}. 
		\end{gather*}
		Hence, $\tilde{A}_{s, t}$ is a correct approximation. 
		
		{\bfseries Step 2.} Again we do not need the stochastic sewing; the deterministic sewing lemma (Lem.~\ref{lem:det_sewing}) in $L^p_{\omega}$ suffices. We compute 
		\begin{equation*}
			\delta \tilde{A}_{s, u, t} = - \Big( \int_s^u g(Y_r) \, \mathrm{d} B^i_r - g(Y_s) B^i_{s, u} \Big) \int_u^t f(X_r) \, \mathrm{d} B^j_r 
			- g(X)_{s, u} \int_u^t f(X_r) B^i_{u, r} \, \mathrm{d} B^j_r.
		\end{equation*}
		By Prop.~\ref{prop:fbm_rough} and Prop.~\ref{prop:rough_iterated_integral}, 
		\begin{align*}
			\norm[\Big]{\int_s^u g(Y_r) \, \mathrm{d} B^i_r - g(Y_s) B^i_{s, u}}_{L^{2p}_{\omega}} &\lesssim \norm{g}_{\mathcal{C}_x^{\tilde{\gamma}}} (1 + \opnorm{X}_{4p, 1}) (t-s)^{(1+\tilde{\gamma})H}, \\
			\norm[\Big]{\int_u^t f(X_r) \, \mathrm{d} B^j_r}_{L^{2p}_{\omega}} & \lesssim \norm{f}_{\mathcal{C}_x^{\tilde{\gamma}}} (1 + \opnorm{Y}_{4p, 1}) (t-s)^{H}, \\
			\norm{g(X)_{s, u}}_{L^{2p}_{\omega}} &\lesssim \norm{g}_{\mathcal{C}_x^{\tilde{\gamma}}} \opnorm{X}_{4p, 1} (t-s)^{\tilde{\gamma} H}, \\
			\norm[\Big]{\int_u^t f(X_r) B^i_{u, r} \, \mathrm{d} B^j_r}_{L^{2p}_{\omega}} &\lesssim \norm{f}_{\mathcal{C}_x^{\tilde{\gamma}}} 
			(1 + \opnorm{Y}_{8p, 1}) (t-s)^{2H}.
		\end{align*}
		Since $(2+\tilde{\gamma})H > 1$, we are in the setting of the sewing lemma (Lem.~\ref{lem:det_sewing}). 
		The remainder estimate yields 
		\begin{multline*}
			\norm[\Big]{\int_s^t \Big(\int_s^{r_2} g(Y_{r_1}) \, \mathrm{d} B_{r_1}^i \Big) f(X_{r_2}) \, \mathrm{d} B^j_{r_2} - \tilde{A}_{s, t}}_{L^p_{\omega}} \\  
			\lesssim \norm{g}_{\mathcal{C}_x^{\tilde{\gamma}}} \norm{f}_{\mathcal{C}_x^{\tilde{\gamma}}} (1 + \opnorm{X}_{8p, 1}) (1 + \opnorm{Y}_{8p, 1}) (t-s)^{(2+\tilde{\gamma})H}.
		\end{multline*}
		Setting $\tau = s$, we get 
		\begin{multline*}
			\norm[\Big]{\int_s^t \Big(\int_s^{r_2} g(Y_{r_1}) \, \mathrm{d} B_{r_1}^i \Big) f(X_{r_2}) \, \mathrm{d} B^j_{r_2} - g(Y_s) \int_s^t f(X_r) B^i_{s, r} \, \mathrm{d} B^j_r}_{L^p_{\omega}} \\  
			\lesssim \norm{g}_{\mathcal{C}_x^{\tilde{\gamma}}} \norm{f}_{\mathcal{C}_x^{\tilde{\gamma}}} (1 + \opnorm{X}_{8p, 1}) (1 + \opnorm{Y}_{8p, 1}) (t-s)^{(2+\tilde{\gamma})H}.
		\end{multline*}
		By Prop.~\ref{prop:rough_iterated_integral}, 
		\begin{equation*}
			\norm[\Big]{\int_s^t f(X_r) B^i_{s, r} \, \mathrm{d} B^j_r - f(X_s) \mathbb{B}^{i, j}_{s, t}}_{L^p_{\omega}} \lesssim 
			\norm{f}_{\mathcal{C}_x^{\tilde{\gamma}}} (1 + \opnorm{Y}_{4p, 1}) (t-s)^{(2+\tilde{\gamma})H},
		\end{equation*}
		and the proof is complete.
	\end{proof}
	\subsection{Pathwise Uniqueness}\label{subsec:pw_uniq_rough}
	Our final result is the following, which is a restatement of Thm.~\ref{thm:main} in the regime $(1/3,1/2]$.
	\begin{theorem}[Pathwise Uniqueness in the Rough Regime]\label{thm:pathwise_uniqueness_rough}
		Let $T>0$, $d_1,\,d_2 \in \mbN$,  $B^H$ be an fBm with Hurst parameter $H \in (1/3, 1/2]$, $\gamma \in ((1-H)/H, 2)$ and $\sigma \in \mcC^{\gamma}(\mbR^{d_1};\mcM(d_1,d_2))$ be such that $\sigma \sigma^{\top}$ is uniformly elliptic, i.e. there exists an $l \in [1, \infty)$ such that 
		$l^{-1} \leq \sigma(x) \sigma^{\top}(x) \leq l$ for all $x \in \mathbb{R}^{d_1}$.
		Then, pathwise uniqueness holds for \eqref{eq:main_sde_rough} in the sense of Def.~\ref{def:pathwise_uniq_rough}.
	\end{theorem}
	Before giving the proof of Thm.~\ref{thm:pathwise_uniqueness_rough}, we require some preparations. Let $\gamma \in (\frac{1-H}{H}, 2)$ be as above and fix an $\alpha \in (1/3, H)$ such that 
	$\gamma + \gamma \alpha > 1$.
	Let $\rho \in C^\infty(\mbR^{d_1};\mcM(d_1,d_2))$ (which we will eventually think of as a smooth approximation to $\sigma$)
	and $(X, X'), (Y, Y') \in \mathscr{D}^{\gamma \alpha}_{B^H}$ (see Def.~\ref{def:controlled_path}). 
	Then, we set 
	\begin{align}\label{eq:G_k_def}
		G_t^k \defby G[\rho]^k_t \defby \int_0^t \Big( \int_0^1 \partial_k \rho(\theta X_r + (1-\theta) Y_r) \, \mathrm{d} \theta \Big) \, \mathrm{d} B_r^H,\quad \text{for}\, k=1,\ldots,d.
	\end{align}
	The integral is understood as a rough integral; since $\rho$ was assumed to be smooth for now, the path 
	\begin{align*}
		t \mapsto \int_0^1 \partial_k \rho(\theta X_t + (1-\theta) Y_t) \, \mathrm{d} \theta
	\end{align*}
	is controlled by $B^H$ with Gubinelli derivative 
	\begin{align*}
		t \mapsto \int_0^1 \nabla \partial_k \rho(\theta X_t + (1-\theta) Y_t) (\theta X_t' + (1-\theta) Y_t') \, \mathrm{d} \theta.
	\end{align*}
	We furthermore set for $k=1,\ldots,d_1$,
	\begin{align}\label{eq:curl_G_k}
		\mathcal{G}_{s, t}^k \defby \mathcal{G}[\rho]_{s, t}^k 
		= \int_s^t \Big( \int_0^1 \partial_k \rho(\theta X_t + (1-\theta) Y_t) \, \mathrm{d} \theta \Big) B_{s, r}^H \otimes \, \mathrm{d} B_r^H,
	\end{align}
	and define the vector of paths
	$\mathbf{G}^k \defby (G^k, \mathcal{G}^k)$. Note that $(\mathbf{G}^k)_{k=1}^d$ does not satisfy Chen's relation (see \eqref{eq:modified_chen} below) so it is not strictly speaking a rough path, however, the same analytic tools may be applied to it. 	By the remainder estimate of rough integration \eqref{eq:rough_fund_est} and assumed regularity of $\rho$, we may define the following finite quantity,
	\begin{align}\label{eq:G_norm}
		\|\mathbf{G}^k\|_{\alpha,\gamma} \defby \sup_{0 \leq s < t \leq T} \left( 
		\frac{\abs{G^k_{s, t}}}{(t-s)^{\alpha}} + \frac{\abs{\mathcal{G}^k_{s, t}}}{(t-s)^{\gamma \alpha}} \right)
		< \infty.
	\end{align}
	Finally, we observe the following \emph{modified} Chen's relation
	\begin{equation}\label{eq:modified_chen}
		\mathcal{G}_{s, t}^k = \mathcal{G}_{s, u}^k + \mathcal{G}_{u, t}^k + B_{s, u}^H \otimes G_{u, t}^k.
	\end{equation}

	More abstractly, we can consider any pair $\mathbf{G} = (G, \mathcal{G})$ satisfying the analytic condition \eqref{eq:G_norm} and 
	the algebraic condition \eqref{eq:modified_chen}.
	Let $Z$ be a path controlled by $B^H$.
	It is not difficult to see (simply by repeating the arguments in the usual rough path setting, e.g. \cite[Thm.~4.10]{FH20}) that 
	the integral 
	\begin{align}\label{eq:rough_int_along_G}
		\int_s^t Z_r \, \mathrm{d} \mathbf{G}_r 
		\defby \lim_{\abs{\pi} \to 0} \sum_{[u,v] \in \pi} (Z_u G_{u, v} + Z'_u \mathcal{G}_{u, v})
	\end{align}
	exists, and we have the quantitative remainder estimate
	\begin{align}\label{eq:mod_rp_remainder}
		\abs[\Big]{\int_s^t Z_r  \, \mathrm{d} \mathbf{G}_r - Z_s G_{s, t} - Z_s' \mathcal{G}_{s, t}} 
		\lesssim_{\gamma, \alpha} \norm{Z}_{\mathscr{D}_{B^H}^{\gamma \alpha}} \norm{\mathbf{G}}_{\alpha, \gamma} (t-s)^{ (1+ \gamma) \alpha}.
	\end{align}
	Furthermore, as in \cite[Thm.~4.17]{FH20}, we have the following stability estimate: by defining the metric 
	\begin{align}\label{eq:mod_rp_metric}
		d_{\alpha,\gamma}(\mathbf{G}, \bar{\mathbf{G}}) 
		\defby \sup_{0 \leq s < t \leq T} \left( \frac{\abs{G_{s, t} - \bar{G}_{s, t}}}{(t-s)^{\alpha}} +  \frac{\abs{\mathcal{G}_{s, t} - \bar{\mathcal{G}}_{s, t}}}{(t-s)^{\gamma \alpha}}\right), 
	\end{align}
	and provided that 
	$\norm{\mathbf{G}}_{\alpha, \gamma} + \norm{\bar{\mathbf{G}}}_{\alpha, \gamma} \leq M$
	for some $M \geq 1$,
	we have 
	\begin{align}\label{eq:modified_rp_stab}
		\abs[\Big]{\int_s^t Z_r \, \mathrm{d} \mathbf{G}_r - \int_s^t Z_r \, \mathrm{d} \bar{\mathbf{G}}_r} 
		\lesssim_{\gamma, \alpha, M} \norm{Z}_{\mathscr{D}^{\gamma \alpha}_{B^H}} 
		d_{\alpha}(\mathbf{G}, \bar{\mathbf{G}}) (t-s)^{\alpha}.
	\end{align}

	With the above in place and continuing to treat $\rho$ as smooth, we obtain the following linearisation,  as laid out in Sec.~\ref{subsec:proof_overview}.
	\begin{lemma}\label{lem:linearization_rough}
		Let $\gamma \in ((1-H)/H, 2)$, $\alpha \in (1/3, H)$ be such that $\gamma + \gamma \alpha > 1$, $\rho \in C^\infty(\mbR^{d_1};\mcM(d_1,d_2))$,  $(X, X'), (Y, Y') \in \mathscr{D}^{\gamma \alpha}_{B^H}$
		and for $k=1, 2, \ldots, d_1$ let us set $\mathbf{G}^k[\rho] = (G^k[\rho], \mathcal{G}^k[\rho])$ as in \eqref{eq:G_k_def} and \eqref{eq:curl_G_k}. Then it holds that 
		\begin{align*}
			\int_0^t \left(\rho(X_r) - \rho(Y_r)\right) \, \mathrm{d} B_r^H 
			= \sum_{k=1}^{d_1} \int_0^t (X^k_r - Y^k_r) \, \mathrm{d} \mathbf{G}_r^k[\rho],
		\end{align*}
		where the integral with respect to $\mathbf{G}^k[\rho]$ is understood in the same sense as \eqref{eq:rough_int_along_G}.
	\end{lemma}
	\begin{proof}
		The proof is completely pathwise.
		We set 
		\begin{align*}
			A^1_{s, t} &\defby 
			\left( \rho(X_s) - \rho(Y_s) \right) B^H_{s, t}  
			+ \left(\nabla \rho(X_s) X_s' -\nabla \rho(Y_s) Y_s'\right) \mathbb{B}^H_{s, t},  \\
			A^2_{s, t} &\defby \sum_{k=1}^{d_1} (X^k_s - Y^k_s) G_{s, t}^k + \left( (X'_s)^k - (Y'_s)^k \right) \mathcal{G}^k_{s, t},
		\end{align*}  
		where as usual $X^k$ denotes the $k^{\text{th}}$ element of $X \in \mbR^{d_1}$ and $(X')^k$ denotes the $k^{\text{th}}$ row of the matrix $X' \in \mcM(d_1,d_2)$, similarly for $Y,\,Y'$.
		In addition, the second term in $A^1_{s, t}$ should be understood as a $d_1$-dimensional vector whose $i^{\text{th}}$ component is given by 
		\begin{align*}
			\sum_{j, k, l}  \left(\partial_k\rho^{i j}(X_s) (X'_s)^{k l} - \partial_k\rho^{i j}(Y_s) (Y'_s)^{k l}\right) \mathbb{B}_{s, t}^{H, l j}.
		\end{align*}
		Note that 
		\begin{align*}
			\int_0^t \big(\,\rho(X_r) - \rho(Y_r)\big) \, \mathrm{d} B^H_r 
			&= \lim_{\substack{\pi \text{ is a partition of } [0, t], \\\abs{\pi} \to 0}} \sum_{[u, v] \in \pi} A^1_{u, v},  \\
			\sum_{k=1}^{d_1} \int_0^t \big(X^k_r - Y^k_r\big) \, \mathrm{d} \mathbf{G}_r^k[\rho]
			&= \lim_{\substack{\pi \text{ is a partition of } [0, t], \\\abs{\pi} \to 0}} \sum_{[u, v] \in \pi} A^2_{u, v}.
		\end{align*}
		Therefore, by the uniqueness part of the deterministic sewing lemma (Rem.~\ref{rem:sewing_uniq}), it suffices to show $\abs{A^1_{s, t} - A^2_{s, t}} \lesssim (t-s)^{3H}$. 
		It follows from the mean value theorem that we have both
		\begin{align}\label{eq:rho_mv}
			\rho(X_s) - \rho(Y_s) = \sum_{k=1}^{d_1} \int_0^1 \partial_k \rho(\theta X_s + (1-\theta) Y_s) (X^k_s - Y^k_s) \, \mathrm{d} \theta
		\end{align}
		and 
		\begin{multline}
			\label{eq:nabla_rho_mv}
			\nabla \rho(X_s) X_s' - \nabla \rho(Y_s) Y_s' 
			= \sum_{k=1}^{d_1} 
			\Big( \int_0^1 (X^k_s - Y^k_s) \partial_k \nabla \rho(\theta X_s + (1-\theta) Y_s) \\
			\times (\theta X_s' + (1-\theta) Y_s') \, \mathrm{d} \theta 
			+ \int_0^1 (X_s' - Y_s')^k \partial_k \rho(\theta X_s + (1-\theta) Y_s)  \, \mathrm{d} \theta \Big).
		\end{multline}
		By \eqref{eq:rho_mv} and \eqref{eq:nabla_rho_mv}, we compute 
		\begin{align*}
			A^2_{s, t} - A^1_{s, t} 
			= \sum_{k=1}^{d_1} \left( (X^k_s - Y^k_s) \Delta^{1, k}_{s, t} + 
			\left( (X'_s)^k - (Y'_s)^k \right) \Delta^{2, k}_{s, t} \right),
		\end{align*}
		where 
		\begin{align*}
			\Delta^{1, k}_{s, t} \defby &\int_s^t \Big( \int_0^1 \partial_k \rho(\theta X_r + (1-\theta) Y_r) \, \mathrm{d} \theta \Big) \, \mathrm{d} B^H_r
			- \left(  \int_0^1 \partial_k \rho(\theta X_s + (1-\theta) Y_s) \, \mathrm{d} \theta \right) B^H_{s, t}  \\
			&- \Big(  \int_0^1 \partial_k \nabla \rho(\theta X_s + (1-\theta) Y_s) (\theta X_s' + (1-\theta) Y_s') \, \mathrm{d} \theta \Big) \mathbb{B}^H_{s, t}
		\end{align*}
		and 
		\begin{multline*}
			\Delta^{2, k}_{s, t} \defby \int_s^t \Big(  \int_0^1 \partial_k \rho(\theta X_r + (1-\theta) Y_r) \, \mathrm{d} \theta \Big) B^H_{s,r} \otimes \, \mathrm{d} B^H_r \\
			- \Big(  \int_0^1 \partial_k \rho(\theta X_s + (1-\theta) Y_s) \, \mathrm{d} \theta \Big) \mathbb{B}^H_{s, t}.
		\end{multline*}
		By the remainder estimate of the rough integral \eqref{eq:rough_fund_est} and assumed regularity of $\rho$,  
		both $\Delta^{1, k}_{s, t}$ and $\Delta^{2, k}_{s, t}$ are bounded by $(t-s)^{3H}$ up to constant, 
		and the proof is complete.
	\end{proof}
	We are now ready to prove Thm.~\ref{thm:pathwise_uniqueness_rough}, the key step being 
	to use Prop.~\ref{prop:fbm_rough}, \ref{prop:rough_iterated_integral} \& \ref{prop:rough_itr_int_2} to prove the convergence of 
	the rough path lifts from $(B^H, \int \partial_k \sigma_n (X_r) \otimes \, \mathrm{d} B^H_r)$ as in \eqref{eq:B_G_lift}, where 
	the sequence $(\sigma_n)$ of smooth maps approximates $\sigma$ in a suitable topology, specified below.
	%
	\begin{proof}[Proof of Thm.~\ref{thm:pathwise_uniqueness_rough}]
		We assume that \eqref{eq:sigma_low_oscillation} holds, as the general case is handled by stopping time arguments as exhibited in the proof of Thm.~\ref{thm:pathwise_uniqueness_young}.
		Let $(X, B^H)$ and $(Y, B^H)$ be two weak solutions to \eqref{eq:main_sde_rough} defined on the same filtered probability space $(\Omega,\mcF,(\mcF_t)_{t\in [0,T]},\mbP)$  and driven by the same $(\mathcal{F}_t)$-fBm, $B^H$. 
		Since the requirement on $\gamma \in ((1-H/H,2))$ is given by an open interval it is always possible to find a $\gamma' \in ((1-H)/H,\gamma)$. We then let $(\sigma_n)_{n=1}^{\infty}$ be a smooth approximation to $\sigma$ converging in $\mcC^{\gamma'}(\mbR^{d_1};\mcM(d_1,d_2))$ (which exists by the compact embedding between H\"older--Besov spaces).
		By \cite[Lem~7.2]{FH20} or \cite[Eq.~(1.8)]{dareiotisgerencser22}, 
		we have 
		\begin{equation*}
			\lim_{n \to \infty} \norm{((\sigma - \sigma_n)(X_r), \nabla (\sigma - \sigma_n)(X_r) X'_r)}_{\mathscr{D}^{\gamma \alpha}_{B^H}} = 0
		\end{equation*}
		and similarly for $Y$.
		In view of \eqref{eq:rough_fund_est}, we then have 
		\begin{align*}
			X_t - Y_t = \lim_{n \to \infty} \int_0^t \left(\sigma_n(X_r) - \sigma_n(Y_r)\right) \, \mathrm{d} B^H_r,
		\end{align*}
		and so by Lem.~\ref{lem:linearization_rough}, 
		\begin{align*}
			X_t - Y_t = \lim_{n \to \infty} \sum_{k=1}^{d_1} \int_0^t (X^k_r - Y^k_r) \, \mathrm{d} \mathbf{G}^k_r[\sigma_n].
		\end{align*}
		By repeating the argument of Lem.~\ref{lem:sol_is_controlled}, we can show that 
		$(X, \sigma(X)) \in \mathfrak{D}^{1, l}_{B^H}$ and similarly for $Y$.
		Furthermore, by the condition \eqref{eq:sigma_low_oscillation}, as checked in Step 1 of the proof of Thm.~\ref{thm:pathwise_uniqueness_young}, we have 
		\begin{align*}
			\theta X + (1-\theta) Y \in \mathfrak{D}^{1, l/4}_{B^H}, \quad \forall \theta \in [0, 1].
		\end{align*}
		Therefore, by Prop.~\ref{prop:fbm_rough} \& \ref{prop:rough_iterated_integral} (with $f = \partial_k \sigma_n - \partial_k \sigma_m$), 
		and the Kolmogorov criterion for rough paths 
		(\cite[Thm.~3.1]{FH20}), the sequence of the lifts $\mathbf{G}^k[\sigma_n]$ converges in $L^p(\Omega;\mbR)$ to some limit 
		$\mathbf{G}^k[\sigma] = (G^k[\sigma], \mathcal{G}^k[\sigma])$ in the metric $d_{\alpha, \gamma'}$ defined by \eqref{eq:mod_rp_metric}. 
		Namely, for all  $k = 1, 2, \ldots, d_1$ and $p \in [1, \infty)$,
		\begin{equation*}
			\lim_{n \to \infty} \norm{d_{\alpha, \gamma'}(\mathbf{G}^k[\sigma_n], \mathbf{G}^k[\sigma])}_{L^p_\omega} = 0
		\end{equation*}
		The stability estimate \eqref{eq:modified_rp_stab} yields 
		\begin{align*}
			X_t - Y_t = \sum_{k=1}^{d_1} \int_0^t (X^k_r - Y^k_r) \, \mathrm{d} \mathbf{G}^k_r[\sigma] \quad \text{a.s.}
		\end{align*}

		The remainder estimate \eqref{eq:mod_rp_remainder} shows that $X-Y$ is controlled by $(G^k)_{k=1}^{d_1}$ with 
		Gubinelli derivative $X - Y$. 
		Moreover, similarly to $\mathbf{G}^k[\sigma]$, by Prop.~\ref{prop:rough_itr_int_2} (with $f = \partial_k \sigma_n - \partial_k \sigma_m$, $g = \partial_l \sigma_n$ and $f = \partial_k \sigma_n$, $g = \partial_l \sigma_n - \partial_l \sigma_m$) and the Kolmogorov criterion for rough paths, the sequence of the rough paths 
		\begin{align*}
			\hat{\mathbf{G}}[\sigma_n] \defby \Big((G^k[\sigma_n])_{k=1}^{d_1}, \Big(\int G^k[\sigma_n] \otimes \, \mathrm{d} G^l[\sigma_n] \Big)_{k,l=1}^{d_1}\Big) 
		\end{align*}
		converges to some limit $\hat{\mathbf{G}}[\sigma]$.  
		We claim that 
		\begin{align}\label{eq:mod_rp_vs_rp}
			\int_0^t (X^k_r - Y^k_r) \, \mathrm{d} \mathbf{G}^k_r[\sigma]
			= \int_0^t (X^k_r - Y^k_r) \, \mathrm{d} \hat{\mathbf{G}}^k_r[\sigma] \quad \text{a.s.}
		\end{align}
		We note that the integral in the right-hand side is defined as rough integral with respect to $\hat{\mathbf{G}}[\sigma]$, which is well-defined since 
		$X-Y$ is controlled by $(G^k)_{k=1}^{d_1}$ with Gubinelli derivative $X-Y$.
		To see \eqref{eq:mod_rp_vs_rp}, we again rely on the uniqueness part of the sewing lemma (Rem.~\ref{rem:sewing_uniq}). By definition we have
		\begin{align*}
			\int_0^t (X^k_r - Y^k_r) \, \mathrm{d} \mathbf{G}^k_r
			- \int_0^t (X^k_r - Y^k_r) \, \mathrm{d} \hat{\mathbf{G}}^k_r
			= \lim_{\substack{\pi \text{ is a partition of } [0, t], \\ \abs{\pi} \to 0}} 
			\sum_{[u, v] \in \pi} \Delta_{u, v},
		\end{align*}
		where 
		\begin{align*}
			\Delta_{s, t} \defby &\,(\sigma(X_s) - \sigma(Y_s)) \mathcal{G}^k_{s, t}[\sigma] 
			- \sum_{l=1}^{d_1} (X^l_s - Y^l_s) \int_s^t G^l_{s, r}[\sigma] \otimes \, \mathrm{d} G_r^k[\sigma] \\ 
			= &\sum_{l=1}^{d_1} (X^l_s - Y^l_s)  \left( \left( \int_0^1 \partial_l \sigma(\theta X_s + (1-\theta) Y_s) \, \mathrm{d} \theta \right) \mathcal{G}^k_{s, t}[\sigma]- \int_s^t G^l_{s, r}[\sigma] \otimes \, \mathrm{d} G_r^k[\sigma] \right).
		\end{align*}
		By Prop.~\ref{prop:rough_iterated_integral} \& \ref{prop:rough_itr_int_2}, 
		\begin{equation*}
			\norm[\Big]{\Big( \int_0^1 \partial_l \sigma(\theta X_s + (1-\theta) Y_s) \, \mathrm{d} \theta \Big) \mathcal{G}^k_{s, t} - \int_s^t G^l_{s, r} \otimes \, \mathrm{d} G_r^k\,}_{L^p_{\omega}} 
			\lesssim (t-s)^{(1+\gamma') H},
		\end{equation*}
		and as $(1+\gamma)H > 1$ the identity \eqref{eq:mod_rp_vs_rp} is established.
		
		Hence, $X-Y$ solves the linear rough differential equation 
		\begin{align*}
			\mathrm{d} Z_t = \sum_{k=1}^{d_1} Z_t^k \, \mathrm{d} \hat{\mathbf{G}}_t^k, \quad Z_0 = 0,  
		\end{align*}
		and its uniqueness implies that for $\mbP$-a.e. $\omega$ we have  $X(\omega) = Y(\omega)$, which concludes the proof.
	\end{proof}
	%
	\begin{appendix}
		\section{Properties of Heat Semi-Groups}
		We collect some standard, useful results concerning heat semi-groups with positive definite conductances.
		\begin{lemma}\label{lem:semigroup_regularisation}
			Let $c \in (1, \infty)$ and $\Gamma \in \cM_+$ such that $c^{-1} \leq \Gamma \leq c$. 
			Then,
			for $\alpha, \beta \in \R$ with $\beta \geq 0 \vee \alpha$ and $t \in (0, T)$ one has
			\begin{equation*}
				\norm{\cP_{t\Gamma } f}_{\cC^{\beta}_x} \lesssim_{c, \alpha, \beta} (t \wedge 1)^{-\frac{\beta-\alpha}{2}}
				\norm{f}_{\cC^{\alpha}_x}.
			\end{equation*}
		\end{lemma}
		\begin{proof}
			We write
			\begin{equation}\label{eq:Lambda}
				\Lambda_{\Gamma} f(x) \defby f(\sqrt{\Gamma} x).
			\end{equation}
			Then $\cP_{t\Gamma} f = \Lambda_{\Gamma^{-1}} \cP_{t} \Lambda_{\Gamma} f$.
			Since
			\begin{equation*}
				\norm{\Lambda_{\Gamma}f }_{\mathcal{C}_x^{\gamma}} \lesssim_{c, \gamma} \norm{f}_{\mathcal{C}_x^{\gamma}}
				\lesssim_{c, \gamma} \norm{\Lambda_{\Gamma}f }_{\mathcal{C}_x^{\gamma}},
			\end{equation*}
			we can assume that $\Gamma = I_d$. Then the claim follows from \cite[Lem.~2.3]{dareiotis_gerenscer_le_23_euler}, note that the inclusion of $t=1$ in the definition of the inhomogeneous H\"older-Besov norm makes no difference to the arguments given therein..
		\end{proof}
		\begin{lemma}\label{lem:semigroup_difference}
			Let $c \in (1, \infty)$ and let $\Gamma_i \in \cM_+$ $(i = 1, 2)$ be such that $c^{-1} \leq \Gamma_i \leq c$. Then,
			for $\alpha, \beta \in \R$ with $\beta \geq \alpha \vee 0$ and $t \in (0, T)$ we have
			\begin{equation*}
				\norm{(\cP_{t \Gamma_1} - \cP_{t\Gamma_2 }) f}_{\cC_x^{\beta}}
				\lesssim_{c, \alpha, \beta} (t \wedge 1)^{-\frac{\beta-\alpha}{2}} \abs{\Gamma_1 - \Gamma_2}
				\norm{f}_{\cC_x^{\alpha}}.
			\end{equation*}
		\end{lemma}
		\begin{proof}
			We first claim
			\begin{equation}\label{eq:semigroup_diffrence_L_infty}
				\norm{(\cP_{\Gamma_1 t} - \cP_{\Gamma_2 t}) f}_{L^{\infty}}
				\lesssim_{c} \abs{\Gamma_1 - \Gamma_2} \norm{f}_{L^{\infty}}.
			\end{equation}
			Indeed,
			we have
			\begin{equation*}
				(\cP_{\Gamma_1 t} - \cP_{\Gamma_2 t}) f(x) = \int_{\R^{d_1}} \left(p_{\Gamma_1 t}(y) - p_{\Gamma_2 t}(y)\right) f(x - y) \, \mathrm{d} y.
			\end{equation*}
			By \cite[Prop.~2.7]{dareiotis_gerenscer_le_23_euler},
			\begin{equation}\label{eq:heat_kernel_difference}
				\abs{p_{\Gamma_1 t}(y) - p_{\Gamma_2 t}(y)} \lesssim_c
				\abs{\Gamma_1 - \Gamma_2} (p_{\Gamma_1 t/2}(y) + p_{\Gamma_2 t/2}(y)),
			\end{equation}
			which yields the estimate \eqref{eq:semigroup_diffrence_L_infty}.
			
			Next we prove the estimate \eqref{eq:semigroup_difference} for $\alpha = \beta \notin \N$.
			The proof uses the characterization of Besov spaces by Littlewood--Paley blocks \cite{bahouri_chemin_danchin_11_fourier}.
			The case for $\alpha = \beta = 0$ is proved by \eqref{eq:semigroup_diffrence_L_infty}.
			Let $\{\Delta_j\}_{j=-1}^{\infty}$ be Littlewood--Paley blocks.
			Since $\Delta_j$ and $\cP_{\Gamma_i t}$ commute, the estimate \eqref{eq:semigroup_diffrence_L_infty} yields
			\begin{equation*}
				\norm{\Delta_j (\cP_{\Gamma_1 t} - \cP_{\Gamma_2 t})f}_{L^{\infty}}
				\lesssim_c \abs{\Gamma_1 - \Gamma_2} \norm{\Delta_j f}_{L^{\infty}}.
			\end{equation*}
			Therefore,
			\begin{equation*}
				\norm{(\cP_{\Gamma_1 t} - \cP_{\Gamma_2 t}) f}_{\mathcal{C}^{\beta}_x}
				\lesssim_{c} \abs{\Gamma_1 - \Gamma_2} 
				\norm{f}_{\mathcal{C}^{\beta}_x}.
			\end{equation*}

			Next we prove the estimate \eqref{eq:semigroup_difference} for $\alpha = \beta \in \N$.
			Since we only need the case $\alpha = \beta =1$, we only prove that case. We have
			\begin{equation*}
				\abs{(\cP_{\Gamma_1 t} - \cP_{\Gamma_2 t})f(x) - (\cP_{\Gamma_1 t} - \cP_{\Gamma_2 t})f(y)} 
				\leq \int_{\R^{d_1}} \abs{p_{\Gamma_1 t}(z) - p_{\Gamma_2 t}(z)} \abs{f(x - z) - f(y - z)} \, \mathrm{d} z.
			\end{equation*}
			By \eqref{eq:heat_kernel_difference},
			\begin{multline*}
				\abs{(\cP_{\Gamma_1 t} - \cP_{\Gamma_2 t})f(x) - (\cP_{\Gamma_1 t} - \cP_{\Gamma_2 t})f(y)} \\ 
				\lesssim_c \abs{\Gamma_1 - \Gamma_2} \int_{\R^{d_1}} \abs{p_{\Gamma_1 t/2}(z) + p_{\Gamma_2 t/2}(z)} \abs{f(x - z) - f(y - z)} \, \mathrm{d} z, 
			\end{multline*}
			which is then bounded by  
			$2 \abs{\Gamma_1 - \Gamma_2} \norm{f}_{C^1} \abs{x - y}$.
			
			Finally, we prove the general case. Since
			\begin{equation*}
				\cP_{\Gamma_1 t} - \cP_{\Gamma_2 t} = (\cP_{\Gamma_1 t/2} - \cP_{\Gamma_2 t/2}) \cP_{\Gamma_1 t /2}
				+ \cP_{\Gamma_2 t / 2} (\cP_{\Gamma_1 t/2} - \cP_{\Gamma_2 t/2}),
			\end{equation*}
			we have
			\begin{equation*}
				\norm{(\cP_{\Gamma_1 t} - \cP_{\Gamma_2 t}) f}_{\mathcal{C}^{\beta}_x} 
				\leq \norm{ (\cP_{\Gamma_1 t/2} - \cP_{\Gamma_2 t/2}) \cP_{\Gamma_1 t /2} f}_{\mathcal{C}^{\beta}_x}
				+ \norm{ \cP_{\Gamma_2 t / 2} (\cP_{\Gamma_1 t/2} - \cP_{\Gamma_2 t/2}) f}_{\mathcal{C}^{\beta}_x}.
			\end{equation*}
			As for the first term, the estimate \eqref{eq:semigroup_difference} with $\alpha = \beta$ and Lem.~\ref{lem:semigroup_regularisation}
			imply
			\begin{equation*}
				\norm{ (\cP_{\Gamma_1 t/2} - \cP_{\Gamma_2 t/2}) \cP_{\Gamma_1 t /2} f}_{\mathcal{C}^{\beta}_x}
				\lesssim_c \abs{\Gamma_1 - \Gamma_2} \norm{ \cP_{\Gamma_1 t /2} f}_{\mathcal{C}^{\beta}_x} 
				\lesssim_{c, \alpha, \beta} \abs{\Gamma_1 - \Gamma_2} (t \wedge 1)^{-\frac{\beta - \alpha}{2}} \norm{f}_{\mathcal{C}_x^{\alpha}}.
			\end{equation*}
			The estimate of the second term is similar.
		\end{proof}
		\section{Gaussian and Integral Computations}\label{sec:gaussian}
		We compile the proofs of Lem.~\ref{lem:Y_dot}, \ref{lem:tilde_B_rp}, \ref{lem:gaussian_integral} and \ref{lem:integral_along_tilde_B}. These are almost all proved by elementary estimates and Gaussian analysis, however, since, the authors were not able to find them presented coherently and precisely in our setting elsewhere, they are given with full proofs here.
		\subsection{Proofs of Lem.~\ref{lem:Y_dot} and Lem.~\ref{lem:tilde_B_rp}}\label{subsec:gaussian_sec2}
		\begin{proof}[Proof of Lem.~\ref{lem:Y_dot}]
			For notational concision, we write $Y^s \defby Y^{H, s}$.
			Since $(s, \infty) \ni t \mapsto K_H(t, r)$ is smooth for any $r \in [0, s]$, the process $(s,\infty) \ni t\mapsto Y^s_t$ is smooth, hence we can legitimately take its derivative in this region, which is all that the lemma requires. The estimate on $Y^{s}_{u, t}$ readily follows from that on $\dot{Y}^s_t$, through the 
			identity $Y^{s}_{u, t} = \int_u^t \dot{Y}^s_r \, \mathrm{d} r$.
			Therefore, we focus on the estimate on $\dot{Y}^s_t$.
			
			To this end, we use the Mandelbrot--Van Ness representation \cite{mandelbrot68}: 
			\begin{equation*}
				B^H_t = \int_{-\infty}^t \bar{K}_H(t, r) \, \mathrm{d} \bar{W}_r, \quad \tilde{K}_H(t, r) 
				\defby \bar{c}_H ((t-r)^{H-\frac{1}{2}} - (-r)_+^{H-\frac{1}{2}}),
			\end{equation*}
			where $\bar{W}$ is a two-sided Brownian motion and $\bar{c}_H$ is a suitably chosen constant. We set 
			\begin{equation*}
				\mathcal{F}_t \defby \sigma(B^H_r : r \in [0, t]), \quad \bar{\mathcal{F}}_t \defby \sigma(\bar{W}_r : r \in (-\infty, t]),
			\end{equation*}
			and $\bar{Y}^s_t \defby \mathbb{E}[B^H_t \vert  \bar{\mathcal{F}}_s]$. 
			Since $\mathcal{F}_t \subseteq \bar{\mathcal{F}}_t$ and $Y^s_t = \mathbb{E}[\bar{Y}^s_t \vert \bar{\mathcal{F}}_s]$, we have
			$\dot{Y}^s_t = \mathbb{E}[\dot{\bar{Y}}^s_t \vert \bar{\mathcal{F}}_s]$.
			Therefore, it suffices to estimate $\dot{\bar{Y}}^s_t$, and we observe that 
			\begin{equation*}
				\norm{\dot{\bar{Y}}^s_t}_{L^p_{\omega}}
				\lesssim_p \norm{\dot{\bar{Y}}^s_t}_{L^2_{\omega}}
				= \Big( \int_{-\infty}^s \abs{\partial_t \bar{K}_H(t, r)}^2 \, \mathrm{d} r \Big)^{\frac{1}{2}}
				\lesssim (t-s)^{H-1}. \qedhere
			\end{equation*}
		\end{proof}
		
		\begin{proof}[Proof of Lem.~\ref{lem:tilde_B_rp}]
			Since $\mathbb{B}^{H, s}$ satisfies Chen's relation \eqref{eq:chen_relation} and $Y^{H, s}$ is smooth, it is easy to see that 
			$\tilde{\mathbb{B}}^{H, s}$ also satisfies Chen's relation. 
			To see the regularity estimate, by Kolmogorov's continuity criterion for rough paths \cite[Thm.~3.1]{FH20}, it suffices to show that 
			\begin{align}\label{eq:tilde_bold_B_bound}
				\norm{\tilde{\mathbb{B}}^{H, s}_{u, t}}_{L^p_{\omega}} 
				\lesssim_p (t-u)^{2 H}.
			\end{align}
			But, since we already know the estimates on $\mathbb{B}^{H, s}$ and on $\dot{Y}^{H, s}$ from Lem.~\ref{lem:Y_dot}, the bound \eqref{eq:tilde_bold_B_bound} easily follows. 
			
			To see that $(\tilde{B}^{H, s}, \tilde{\mathbb{B}}^{H, s})$ is a geometric rough path of order $\alpha$, by \cite[Prop.~2.8]{FH20} it suffices to show that 
			the rough path is weakly geometric, i.e. 
			\begin{align*}
				\tilde{\mathbb{B}}^{H, s, i, j}_{u, t} + \tilde{\mathbb{B}}^{H, s, j, i}_{u, t} =  \tilde{B}^{H, s, i}_{u, t} \tilde{B}^{H, s, j}_{u, t}, 
			\end{align*}
			but this algebraic identity is easy to verify, since $(B^{H}, \mathbb{B}^H)$ is geometric.
		\end{proof}
		
		\subsection{Proof of Lem.~\ref{lem:gaussian_integral}}\label{subsec:int_tilde_B}
		Before giving the proof of Lem.~\ref{lem:gaussian_integral}, we require some preparatory lemmas; Lem.~\ref{lem:g_times_iterated_integral}, \ref{lem:rough_vs_stratonovich} \& \ref{lem:rough_vs_strat_2}.
		\begin{lemma}\label{lem:g_times_iterated_integral}
			Let $H \in (1/3, 1/2]$ and $0 \leq s \leq u < t$, and we define $\tilde{B}^{H, s}$ by \eqref{eq:def_tilde_B} and $\tilde{\mathbb{B}}^{H, s}$ by \eqref{eq:tilde_bold_B}.
			Let $g \from \mathbb{R}^{d_2} \to \mathbb{R}$ be such that $g(\tilde{B}^{H, s}_u) \in L^2_{\omega}$. 
			Then we have 
			\begin{align*}
				\mathbb{E}\left[g(\tilde{B}_u^{H, s}) \tilde{\mathbb{B}}^{H, s, i, j}_{u, t} \right] 
				= \frac{1}{2} \mathbb{E}[g(\tilde{B}^{H, s}_u) \tilde{B}_{u, t}^{H, s, i} \tilde{B}_{u, t}^{H, s, j}].
			\end{align*}
		\end{lemma}
		\begin{proof}
			For notational concision we write $\tilde{B} \defby \tilde{B}^{H, s}$ and $\tilde{\mathbb{B}} \defby \tilde{\mathbb{B}}^{H, s}$.
			If $i = j$, the identity is trivial, since the rough path of $\tilde{B}$ is geometric. 
			We therefore assume below that $i \neq j$.
			Since $i \neq j$, $(B^{H, i}, Y^{H, s, i})$ and $(B^{H, j}, Y^{H, s, j})$ are independent, and the random variable $\tilde{\mathbb{B}}^{H, s, i, j}$ 
			defined by \eqref{eq:tilde_bold_B} belongs to the second order Wiener chaos. (For the basic properties of the Wiener chaos, see \cite[Chapter~1.1]{nualart06}.)
			The orthogonality in the Wiener chaos implies that  
			\begin{equation*}
				\mathbb{E}[g(\tilde{B}_u)  \tilde{\mathbb{B}}_{u, t}^{i, j}   ] 
				= \mathbb{E}[(\tilde{B}^i_u \tilde{B}^j_u)^2]^{-1} 
				\mathbb{E}[g(\tilde{B}_u) \tilde{B}^i_u \tilde{B}^j_u]
				\mathbb{E}[ \tilde{B}^i_u \tilde{B}^j_u \tilde{\mathbb{B}}^{i, j}_{u, t} ].
			\end{equation*}
			Since $(\tilde{B}^i, \tilde{B}^j) \dequal (\tilde{B}^j, \tilde{B}^i)$, we have 
			\begin{gather*}
				\mathbb{E}[ \tilde{B}^i_u \tilde{B}^j_u \tilde{\mathbb{B}}^{i, j}_{u, t} ]
				= \mathbb{E}[ \tilde{B}^i_u \tilde{B}^j_u \tilde{\mathbb{B}}^{j, i}_{u, t} ], \\
				\mathbb{E}[ \tilde{B}^i_u \tilde{B}^j_u \tilde{\mathbb{B}}^{i, j}_{u, t} ]
				= \frac{1}{2} \mathbb{E}[ \tilde{B}^i_u \tilde{B}^j_u 
				(\tilde{\mathbb{B}}^{i, j}_{u, t} 
				+ \tilde{\mathbb{B}}^{j, i}_{u, t})] 
				= \frac{1}{2} \mathbb{E}[\tilde{B}^i_u \tilde{B}^j_u 
				\tilde{B}^i_{u, t} \tilde{B}^j_{u, t}],
			\end{gather*}
			where  we use the fact that $(\tilde{B}, \tilde{\mathbb{B}})$ is geometric. 
			Therefore, 
			\begin{equation*}
				\mathbb{E}\Big[g( \tilde{B}_u)  \tilde{\mathbb{B}}^{i, j}_{u, t}   \Big] 
				= \frac{1}{2}
				\mathbb{E}[(\tilde{B}^i_u \tilde{B}^j_u)^2]^{-1} 
				\mathbb{E}[g(\tilde{B}_u) \tilde{B}^i_u \tilde{B}^j_u]
				\mathbb{E}[ \tilde{B}^i_u \tilde{B}^j_u \tilde{B}^i_{u, t} \tilde{B}^j_{u, t} ] 
				= \frac{1}{2} \mathbb{E}[g(\tilde{B}_u)  \tilde{B}^i_{u, t} \tilde{B}^j_{u, t}],
			\end{equation*}
			and the proof is complete. 
		\end{proof}
		\begin{lemma}\label{lem:rough_vs_stratonovich}
			Let $H \in (1/3, 1/2]$ and $\tau \geq 0$, and we define $\tilde{B}^{H, \tau}$ by \eqref{eq:def_tilde_B} 
			and define $\tilde{\mathbb{B}}^{H, \tau}$ by \eqref{eq:tilde_bold_B}.
			Let $g \from \mathbb{R}^{d_2} \to \mathbb{R}$ be a bounded measurable function. 
			%
			We set 
			\begin{align*}
				A_{s, t}^1 \defby g(\tilde{B}^{H, \tau}_s) \tilde{B}^{H, \tau, i}_{s, t} \tilde{B}^{H, \tau, j}_{s, t}, 
				\quad A_{s, t}^2 \defby g(\tilde{B}^{H, \tau}_s) \tilde{\mathbb{B}}^{H, \tau, i, j}_{s, t}.
			\end{align*}
			We write $(\mathcal{F}_t)_{t \geq 0}$ for the filtration generated by $(B^H_t)_{t \geq 0}$.
			Then, for $p \in [2, \infty)$ and $\tau \leq v < s < t$ with $t-s \leq s-v$, we have 
			\begin{align*}
				\norm{\mathbb{E}[A^1_{s,t} - 2 A^2_{s, t} \vert \mathcal{F}_v]}_{L^p_{\omega}}
				\lesssim_{p} \norm{g}_{L^{\infty}_x} \Big( \frac{t-s}{s-v} \Big)^{1 - H} (t-s)^{2 H}.
			\end{align*}
		\end{lemma}
		\begin{proof}
			For simplicity, we drop scripts on $H$. Recalling the Volterra representation \eqref{eq:fbm_volterra}, we set 
			\begin{align}\label{eq:Y_tau_v}
				Y^{\tau, v}_t \defby \int_{\tau}^v K_H(t, r) \, \mathrm{d} W_r = Y^{v}_t - Y^{\tau}_t. 
			\end{align}
			Note that $Y^{\tau, v} = \mathbb{E}[\tilde{B}^{\tau} \vert \mathcal{F}_v]$ 
			and 
			$\tilde{B}^{\tau} - Y^{\tau, v} = \tilde{B}^{v}$
			is independent of $\mathcal{F}_v$.
			We have 
			\begin{align*}
				A^1_{s, t} 
				= g(Y_s^{\tau, v} + \tilde{B}^v_s) [Y_{s, t}^{\tau, v, i} Y_{s, t}^{\tau, v, i} 
				+ Y_{s, t}^{\tau, v, i} \tilde{B}^{v, j}_{s, t} + \tilde{B}^{v, i}_{s, t} Y_{s, t}^{\tau, v, j} 
				+ \tilde{B}^{v, i}_{s, t} \tilde{B}^{v, j}_{s, t}].
			\end{align*}
			To estimate the first three terms, observe by Lem.~\ref{lem:Y_dot} that 
			\begin{align}\label{eq:Y_tau_v_estimate}
				\norm{Y_{s, t}^{\tau, v}}_{L^p_{\omega}} \leq  \norm{Y^{v}_{s, t}}_{L^p_{\omega}}
				\lesssim_{p} (s-v)^{H - 1} (t-s),   
			\end{align}
			and by \eqref{eq:rho_lower_bd} that $\norm{\tilde{B}^v_{s, t}}_{L^p_{\omega}} \lesssim_p (t-s)^H$.
			Hence, we obtain 
			\begin{align*}
				\norm{A^1_{s, t} - g(Y_s^{\tau, v} + \tilde{B}^v_s) \tilde{B}^{v, i}_{s, t} \tilde{B}^{v, j}_{s, t}}_{L^p_{\omega}} 
				\lesssim_{p} \norm{g}_{L^{\infty}_x} \Big( \frac{t-s}{s-v} \Big)^{1-H} (t-s)^{2H}.
			\end{align*}
			Similarly, $A^2_{s,t}$ equals to 
			\begin{equation*}
				g(Y_s^{\tau, v} + \tilde{B}_s^v) 
				\Big[ \int_s^t Y_{s, r}^{\tau, v, i} \dot{Y}_r^{\tau, v, j} \, \mathrm{d} r + 
				\int_s^t Y_{s, r}^{\tau, v, i} \, \mathrm{d} \tilde{B}^{v, j}_r 
				+ \int_s^t \tilde{B}^{ v, i}_{s, r} \dot{Y}_r^{\tau, v, j} \, \mathrm{d} r
				+ \tilde{\mathbb{B}}^{v, i, j}_{s, t}
				\Big]
			\end{equation*}
			and 
			\begin{align*}
				\norm[\Big]{A^2_{s, t} - g(Y^{\tau, v}_s + \tilde{B}_s^v)  \tilde{\mathbb{B}}^{v, i, j}_{s, t}}_{L^p_{\omega}} 
				\lesssim_p \norm{g}_{L^{\infty}_x} \Big( \frac{t-s}{s-v} \Big)^{1-H} (t-s)^{2 H}.
			\end{align*}
			To complete the proof, it remains to observe
			\begin{align}\label{eq:tilde_B_iterated_integral}
				\mathbb{E}[g(Y_s + \tilde{B}_s)  \tilde{\mathbb{B}}^{v, i, j}_{s, t}  \vert \mathcal{F}_v ]
				= \frac{1}{2} \mathbb{E}[g(Y_s + \tilde{B}_s) \tilde{B}^i_{s, t} \tilde{B}_{s, t}^j 
				\vert \mathcal{F}_v],
			\end{align}
			which follows from Lem.~\ref{lem:g_times_iterated_integral}.
		\end{proof}
		\begin{lemma}\label{lem:rough_vs_strat_2}
			Let $H \in (1/3, 1)$ and $\tau \geq 0$, and we define $\tilde{B}^{H, \tau}$ by \eqref{eq:def_tilde_B}. 
			Let $g \in C^2(\mathbb{R}^{d_2})$. We then have 
			\begin{align*}
				\int_s^t g(\tilde{B}^{H, \tau}_r) \, \mathrm{d} \tilde{B}^{H, \tau}_r = 
				\lim_{\substack{\pi \text{ is a partition of }[s, t], \\ \abs{\pi} \to 0 }} 
				\sum_{[s', t'] \in \pi} \frac{g(\tilde{B}^{H, \tau}_{s'}) + g(\tilde{B}^{H, \tau}_{t'})}{2} \tilde{B}^{H, \tau}_{s', t'},
			\end{align*} 
			where the convergence is in $L^p_{\omega}$ for any $p \in [1, \infty)$.
		\end{lemma}
		\begin{proof}
			For simplicity, we drop scripts on $H$.
			When $H > 1/2$, the claim follows from the remainder estimate of the Young integral \eqref{eq:young_fund_est}. 
			Hence, we assume that $H \leq 1/2$.
			
			We set 
			\begin{equation*}
				A^1_{s, t} \defby g(\tilde{B}^{\tau}_s) \tilde{B}^{\tau}_{s, t} 
				+ \nabla g(\tilde{B}^{\tau}_s) \tilde{\mathbb{B}}^{\tau}_{s, t},
				\quad 
				A^2_{s, t} \defby \frac{g(\tilde{B}^{\tau}_s) + g(\tilde{B}^{\tau}_t)}{2} \tilde{B}^{\tau}_{s, t}.
			\end{equation*}
			The integral $\int_s^t g(\tilde{B}^{\tau}_r) \, \mathrm{d} \tilde{B}^{\tau}_r$ is understood as rough integral: 
			\begin{align}\label{eq:g_integral_by_A_1}
				\int_s^t g(\tilde{B}^{\tau}_r) \, \mathrm{d} \tilde{B}^{\tau}_r
				= \lim_{\substack{\pi \text{ is a partition of } [s, t], \\ \abs{\pi} \to 0}} \sum_{[s', t'] \in \pi} A^1_{s', t'}\quad \text{in } L^p_{\omega}.
			\end{align}

			The proof relies on the uniqueness part of the deterministic and stochastic sewing lemma (Rem.~\ref{rem:sewing_uniq}). 
			We have 
			\begin{align}
				A^2_{s, t} - A^1_{s, t} 
				&= \frac{g(\tilde{B}^{\tau}_t) - g(\tilde{B}^{\tau}_s)}{2} 
				\tilde{B}^{\tau}_{s, t} 
				- \nabla g(\tilde{B}^{\tau}_s) \tilde{\mathbb{B}}^{\tau}_{s, t} \notag\\
				&= \frac{1}{2} \nabla g(\tilde{B}^{\tau}_s) \tilde{B}^{\tau}_{s,t} \otimes \tilde{B}^{\tau}_{s, t} 
				- \nabla g(\tilde{B}^{\tau}_s) \tilde{\mathbb{B}}^{\tau}_{s, t}
				+ R_{s, t}, \label{eq:A^2_minus_A^1_g}
			\end{align}
			where 
			$\abs{R_{s, t}} \lesssim_{\norm{g}_{C^2_x}} \abs{\tilde{B}^{\tau}_{s, t}}^3$.
			Setting $A^3_{s, t} \defby A^2_{s, t} - A^1_{s,t} - R_{s, t}$, we therefore see that 
			\begin{align}\label{eq:A_1_A_2_A_3}
				\lim_{\abs{\pi} \to 0} \sum_{[s', t'] \in \pi} (A^2_{s', t'} - A^1_{s', t'}) = 
				\lim_{\abs{\pi} \to 0} \sum_{[s', t'] \in \pi}  A^3_{s', t'}  \quad \text{in } L^p_{\omega}.
			\end{align}
			To estimate the right-hand side, we will apply the uniqueness part of Lem.~\ref{lem:ssl}.
			Using \eqref{eq:A^2_minus_A^1_g} and Lem.~\ref{lem:rough_vs_stratonovich}, we obtain 
			\begin{align*}
				\norm{A^3_{s,t}}_{L^p_{\omega}} \lesssim_{p, \norm{g}_{C^1_x}} (t-s)^{2H},
				\quad 
				\norm{\mathbb{E}[A^3_{s, t} \vert \mathcal{F}_v]}_{L^p_{\omega}} 
				\lesssim_{p, \norm{g}_{C^1_x}}  \Big( \frac{t-s}{s-v} \Big)^{1-H}(t-s)^{2H}.
			\end{align*}
			Hence, applying the uniqueness part of Lem.~\ref{lem:ssl}, the claim is proven in view of \eqref{eq:g_integral_by_A_1} 
			and \eqref{eq:A_1_A_2_A_3}.
		\end{proof}
		\begin{proof}[Proof of Lem.~\ref{lem:gaussian_integral}]
			For notational concision we write $\tilde{B}^s \defby \tilde{B}^{H, s}$.
			We first note that, since $\tilde{B}^s$ is Gaussian, the process (recall $\rho_H$ from \eqref{eq:def_rho})
			\begin{equation*}
				g (a \tilde{B}_r^s) - c_r \tilde{B}^{s, i}_r, \quad \text{where}\quad c_r
				\defby \rho_H^{-2}(s, r) \mathbb{E} [g (
				a \tilde{B}_r^s) \tilde{B}^{s, i}_r], 
			\end{equation*}
			is orthogonal to $\tilde{B}^{s, i}$ in $L^2_{\omega}$. Therefore, also applying Lem.~\ref{lem:rough_vs_strat_2}, we have
			\begin{align*}
				\mathbb{E} \Big[ \int_u^t g (a \tilde{B}^s_r)
				\, \mathrm{d} \tilde{B}^{s, i}_r \Big] 
				&= \lim_{\substack{\pi \text{ is a partition of } [u, t], \\ \abs{\pi} \to 0}} \sum_{[s', t'] \in \pi} \mathbb{E} \Big[\frac{g(a \tilde{B}_{s'}^s) + g(a \tilde{B}_{t'}^s)}{2} \tilde{B}^{s, i}_{s', t'} \Big] \\
				&= \lim_{\substack{\pi \text{ is a partition of } [u, t], \\ \abs{\pi} \to 0}} \sum_{[s', t'] \in \pi}\mathbb{E} \Big[\frac{c_{s'} \tilde{B}^{s, i}_{s'} + c_{t'} \tilde{B}^{s, i}_{t'}}{2} \tilde{B}^{s, i}_{s', t'} \Big]\\
				&= \frac{1}{2} \mathbb{E}\Big[ \int_u^t c_r 
				\, \mathrm{d} (\tilde{B}^{s, i}_r)^2 \Big] 
				= \frac{1}{2}  \int_u^t c_r \rho_H^2(s, \mathrm{d} r),
			\end{align*}
			and we obtain
			\begin{equation}\label{eq:g_tilde_B_1}
				\mathbb{E} \left[ \int_u^t g (a \tilde{B}^s_r)
				\, \mathrm{d} \tilde{B}^{s, i}_r \right] =   
				\frac{1}{2} \int_u^t \mathbb{E} [g (a
				\tilde{B}_r^s) \tilde{B}_r^{s, i}] \frac{\rho_H^2(s, \mathrm{d} r)}{\rho_H^2(s, r)}.
			\end{equation}
			Gaussian integration by parts yields
			\begin{equation}
				\mathbb{E} [g (a \tilde{B}_r^s) \tilde{B}_r^{s, i}]
				= \rho_H^2(s, r)
				\sum_{j=1}^{d_1} a^{ji} \mathbb{E}[\partial_j  g (
				a \tilde{B}^s_r ) ] 
				= \rho_H^2(s, r)
				\sum_{j=1}^{d_1} a^{ji} [\partial_j \mathcal{P}_{\rho^2_H(s, r) a a^{\mathrm{T}}} g](0). \label{eq:g_tilde_B_2}
			\end{equation}
			Combining \eqref{eq:g_tilde_B_1} and \eqref{eq:g_tilde_B_2}, we obtain the desired identity.
		\end{proof}
		
		\subsection{Proof of Lem.~\ref{lem:integral_along_tilde_B}}\label{subsec:proof_lem:integral_along_tilde_B}
		To prove Lem.~\ref{lem:integral_along_tilde_B},  we follow the strategy of  \cite[Prop.~3.5]{matsuda22}.
		However, unlike \cite{matsuda22} where the computation is based on Mandelbrot-Van Ness representation \cite{mandelbrot68}, 
		in this article we work with the Volterra representation, which requires some technical preparations; given by Lem.~\ref{lem:K_derivative_K} \& \ref{lem:rho_and_tilde_B} below.
		\begin{lemma}\label{lem:K_derivative_K}
			Let $H \in (1/3, 1/2)$ and we define $K_H$ by \eqref{eq:K_H_less_12}. 
			Then, for $0 \leq v < s \leq t$, we have 
			\begin{align*}
				\abs[\Big]{\int_0^v K_H(s, r) \partial_t K_H(t, r) \, \mathrm{d} r } \lesssim (s-v)^{2H - 1}.
			\end{align*}
		\end{lemma}
		\begin{proof}
			Recalling \eqref{eq:K_H_less_12} and observing 
			\begin{equation*}
				\frac{\partial}{\partial t} K_H(t, s) = c_H \Big( H - \frac{1}{2} \Big) \Big( \frac{t}{s} \Big)^{H- \frac{1}{2}} 
				(t-s)^{H-\frac{3}{2}},
			\end{equation*}
			it suffices to prove 
			\begin{gather}\label{eq:derivative_K_est_1}
				\int_0^v \Big( \frac{s}{r} \Big)^{H-\frac{1}{2}} (s-r)^{H-\frac{1}{2}} \Big( \frac{t}{r} \Big)^{H-\frac{1}{2}} (t-r)^{H-\frac{3}{2}} \, \mathrm{d} r
				\lesssim (s-v)^{2H - 1}, \\
				\int_0^v \Big( \frac{t}{r} \Big)^{H - \frac{1}{2}} (t-r)^{H - \frac{3}{2}} r^{\frac{1}{2} - H} 
				\Big( \int_r^s u^{H - \frac{3}{2}} (u-r)^{H - \frac{1}{2}} \, \mathrm{d} u \Big) \, \mathrm{d} r
				\lesssim (s-v)^{2H - 1}. \label{eq:derivative_K_est_2}
			\end{gather}
			The estimate \eqref{eq:derivative_K_est_1} is straightforward. Indeed, noting that $r/t \leq 1$, $r/s \leq 1$ and $s-r \leq t-r$, the 
			left-hand side is bounded by 
			$(s-v)^{2H - 1}$ up to constant.
			
			To prove \eqref{eq:derivative_K_est_2}, using again $r/t \leq 1$ and also for $r \in [0, v]$
			\begin{align*}
				(t-r)^{H - \frac{3}{2}} = (t-r)^{2H - 1} (t-r)^{-H-\frac{1}{2}} \leq (s-v)^{2H - 1} (t-r)^{-H-\frac{1}{2}},
			\end{align*}
			the left-hand side of \eqref{eq:derivative_K_est_2} is bounded by $(s-v)^{2H-1} I$, where
			\begin{align*}
				I \defby \int_0^v (t-r)^{-H-\frac{1}{2}} r^{\frac{1}{2} - H} \Big( \int_r^t u^{H-\frac{3}{2}} (u-r)^{H-\frac{1}{2}} \, \mathrm{d} u \Big)  \, \mathrm{d} r.
			\end{align*}
			Therefore, it suffices to show $I \lesssim 1$.
			Applying the bound 
			\begin{align*}
				u^{H - \frac{3}{2}} = u^{-1} u^{H - \frac{1}{2}} \leq r^{-1} (u-r)^{H - \frac{1}{2}}, \quad u \in [r, t],
			\end{align*}
			we get 
			\begin{align*}
				I &\leq \int_0^v (t-r)^{-H-\frac{1}{2}} r^{\frac{1}{2} - H} r^{-1} \Big( \int_r^t  (u-r)^{2 H-1} \, \mathrm{d} u \Big)  \, \mathrm{d} r \lesssim \int_0^v (t-r)^{H - \frac{1}{2}} r^{-\frac{1}{2} - H} \, \mathrm{d} r \\
				&\leq \int_0^t (t-r)^{H - \frac{1}{2}} r^{-\frac{1}{2} - H} \, \mathrm{d} r 
				= \int_0^1 (1 - s)^{H - \frac{1}{2}} s^{-\frac{1}{2} - H} \, \mathrm{d} s. \qedhere 
			\end{align*}
		\end{proof}
		\begin{lemma}\label{lem:rho_and_tilde_B}
			Let $H \in (1/3, 1/2]$, and we define $\tilde{B}^{H, v}$ by \eqref{eq:def_tilde_B} and $\rho_H(v, t)$ by \eqref{eq:def_rho}. 
			Let $v < s < t$. 
			\begin{enumerate}
				\item For $i \in \{1, 2, \ldots, d_2\}$, we have 
				\begin{align*}
					\left|\mathbb{E}\left[\tilde{B}_s^{H, v, i} \tilde{B}_{s, t}^{H, v, i}\right]\right| + \left|\mathbb{E}\left[\tilde{B}_t^{H, v, i} \tilde{B}_{s, t}^{H, v, i}\right]\right| \lesssim (t-s)^{2H} + (s-v)^{2H - 1} (t-s).
				\end{align*}
				\item We have $\abs{\rho_H^2(v, t) - \rho_H^2(v, s)} \lesssim (t-s)^{2H} + (s-v)^{2H - 1} (t-s) $.
			\end{enumerate}
		\end{lemma}
		\begin{proof}
			When $H = 1/2$ the claim is trivial, so we suppose that $H < 1/2$.
			
			\textbf{Proof of 1.} By \eqref{eq:rho_lower_bd}, we have 
			$\mathbb{E}[\abs{B^{H, v}_{s, t}}^2] \lesssim (t-s)^{2H}$.
			Hence, it suffices to estimate $\mathbb{E}[\tilde{B}^{H, v, i}_s \tilde{B}^{H, v, i}_{s, t}]$.
			Recall $Y^{H, v}$ from \eqref{eq:def_Y}. Since $Y^{H, v}$ and $\tilde{B}^{H, v}$ are independent and $B^H = Y^{H, v} + \tilde{B}^{H, v}$, 
			we have 
			\begin{align*}
				\mathbb{E}\left[B_s^{H, i} B_{s, t}^{H, i}\right] = \mathbb{E}\left[Y_s^{H, v, i} Y_{s, t}^{H, v, i}\right] 
				+ \mathbb{E}\left[\tilde{B}_s^{H, v, i} \tilde{B}_{s, t}^{H, v, i}\right].
			\end{align*}  
			As for $B^H$, the computation is explicit, and it is easy to see that the left-hand side is bounded by $(t-s)^{2H}$ up to constant.
			Therefore, it suffices to prove $\abs{\mathbb{E}[Y_s^{H, v, i} Y_{s, t}^{H, v, i}]} \lesssim (s-v)^{2H - 1} (t-s)$. 
			However, this follows from the bound 
			\begin{align*}
				\left|\mathbb{E}\left[Y^{H, v, i}_s \dot{Y}^{H, v, i}_t\right]\right| \lesssim (s-v)^{2H - 1},
			\end{align*}
			which in turn follows from Lem.~\ref{lem:K_derivative_K}.
			
			\textbf{Proof of 2.} We have 
			\begin{align*}
				\rho^2_H(v, t) - \rho^2_H(v, s) +  \mathbb{E}\Big[\Big(Y^{H, v, i}_t\Big)^2\Big] - \mathbb{E}\Big[\Big(Y^{H, v, i}_s\Big)^2\Big] 
				= \mathbb{E}\Big[\Big(B_t^{H, i}\Big)^2\Big] - \mathbb{E}\Big[\Big(B_s^{H, i}\Big)^2\Big].  
			\end{align*}
			Since 
			\begin{align*}
				0 \leq \mathbb{E}\Big[\Big(B_t^{H, i}\Big)^2\Big] - \mathbb{E}\Big[\Big(B_s^{H, i}\Big)^2\Big]
				= t^{2H} - s^{2H} \leq (t-s)^{2H},
			\end{align*}
			it suffices to prove $\abs{\mathbb{E}[(Y^{H, v, i}_t)^2] - \mathbb{E}[(Y^{H, v, i}_s)^2]} 
			\lesssim  (s-v)^{2H - 1} (t-s)$. However, this follows from the bound 
			\begin{align*}
				\abs[\Big]{\frac{\mathrm{d}}{\mathrm{d} t} \mathbb{E}\Big[\Big(Y^{H, v, i}_t\Big)^2\Big]} = 2 \Big|\mathbb{E}\Big[Y^{H, v, i}_t \dot{Y}^{H, v, i}_t\Big]\Big| \lesssim (t-v)^{2H - 1},
			\end{align*}
			which follows from Lem.~\ref{lem:K_derivative_K}.
		\end{proof}
		The following is the final preparatory lemma before we are able to give the proof of Lem.~\ref{lem:integral_along_tilde_B}.
		\begin{lemma}\label{lem:integral_stratonovich}
			Let $H \in (1/3, 1/2]$ and $B^H$ be an $(\mathcal{F}_t)$-fBm. 
			We fix a $\tau \in [0, T]$ and define  $\tilde{B}^{H, \tau}$ by  \eqref{eq:def_tilde_B}.
			Let $\xi$ be an $\mathcal{F}_{\tau}$-measurable random variable valued in some measurable space $E$.
			Suppose that $F \from E \times \mathbb{R}^{d_2} \to \mathbb{R}$ satisfies 
			\begin{align*}
				\abs{F(x, y_1) - F(x, y_2)} \leq M(x) 
				\abs{y_1 - y_2}^{\gamma}
			\end{align*}  
			for all $x \in E$ and $y_1, y_2 \in \mathbb{R}^{d_2}$,
			with $\gamma \in (\frac{1}{2H} - 1, 1]$. 
			We set 
			\begin{align*}
				A_{s, t} \defby (F(\xi, B^H_s) + F(\xi, B^H_t)) \tilde{B}_{s, t}^{H, \tau, i}.
			\end{align*}
			For $\tau \leq v < s < u < t$ with $t-s \leq s-v$, we then have 
			\begin{align}
				\label{eq:delta_A_stratonovich}
				\norm{\delta A_{s, u, t}}_{L^p_{\omega}} 
				&\lesssim_{p, \gamma} \norm{M(\xi)}_{L^{2p}_{\omega}} (t-s)^{(1 + \gamma) H},\\
				\label{eq:delta_A_conditional_stratonovich}
				\norm{\mathbb{E}[\delta A_{s, u, t} \vert \mathcal{F}_v]}_{L^p_{\omega}} 
				&\lesssim_{p, \gamma} \norm{M(\xi)}_{L^{2p}_{\omega}} \Big( \frac{t-s}{s-v} \Big)^{(3-\gamma)H} 
				(t-s)^{(1 + \gamma) H}.
			\end{align}
		\end{lemma}
		\begin{proof}
			We essentially repeat the argument of \cite[Prop.~3.5]{matsuda22}.
			For simplicity, we drop scripts on $H$.
			We first consider the estimate \eqref{eq:delta_A_stratonovich}. 
			We compute 
			\begin{align*}
				\delta A_{s, u, t} 
				= F(\xi, B)_{u, t} \tilde{B}^{\tau, i}_{s, u} 
				- F(\xi, B)_{s, u} \tilde{B}^{\tau, i}_{u, t},
			\end{align*}
			and 
			\begin{align*}
				\abs{\delta A_{s, u, t}} 
				\leq M(\xi) \big[\abs{B_{u, t}}^{\gamma} \abs{\tilde{B}^{\tau}_{s, u}}
				+ \abs{B_{s, u}}^{\gamma} \abs{\tilde{B}^{\tau}_{u, t}} \big].
			\end{align*}
			In view of \eqref{eq:rho_lower_bd}, it is easy to see the estimate \eqref{eq:delta_A_stratonovich}.
			
			We turn to the estimate \eqref{eq:delta_A_conditional_stratonovich}.
			We recall $Y^{v}$ from \eqref{eq:def_Y} and $Y^{\tau, v}$ from \eqref{eq:Y_tau_v}, and we recall the decomposition 
			$B = Y^v + \tilde{B}^v$ and $\tilde{B}^{\tau} = Y^{\tau, v} + \tilde{B}^v$, 
			where $Y^v$ and $Y^{\tau, v}$ are measurable with respect to $\mathcal{F}_v$ and $\tilde{B}^v$ is independent of $\mathcal{F}_v$.
			We set 
			\begin{equation*}
				a_0(s) \defby \mathbb{E}[F(x, y_s + \tilde{B}^{v}_s)] \vert_{x = \xi, y = Y^v}, \quad
				a_i(s) \defby  \rho^{-2}(v, s)
				\mathbb{E}[F(x, y_s + \tilde{B}^v_s) \tilde{B}^{v, i}_s] \vert_{x = \xi, y = Y^v}
			\end{equation*}
			with $\rho = \rho_H$ defined by \eqref{eq:def_rho}.
			As in \cite[Prop.~3.5]{matsuda22}, especially \cite[Eq.~(3.8) and (3.11)]{matsuda22}, we have 
			$\mathbb{E}[\delta A_{s, u, t} \vert \cF_v]
			= D^0_{s, u, t} + D^i_{s, u, t}$,
			where
			\begin{align*}
				D^0_{s, u, t} &\defby (a_0(t) - a_0(u))  Y_{s, u}^{\tau, v, i} + (a_0(s) - a_0(u))  Y_{u, t}^{\tau, v, i}, \\
				D^i_{s, u, t} &\defby (a_i(t) - a_i(u)) \mathbb{E}[\tilde{B}_t^{v, i} \tilde{B}_{s, t}^{v, i}]
				+ (a_i(s) - a_i(u)) \mathbb{E}[\tilde{B}_s^{v, i} \tilde{B}_{s, t}^{v, i}] \\
				&\phantom{\defby} -(a_i(s) - a_i(u)) \mathbb{E}[\tilde{B}_s^{v, i} \tilde{B}_{s, u}^{v, i}]
				-(a_i(t) - a_i(u)) \mathbb{E}[\tilde{B}_t^{v, i} \tilde{B}_{u, t}^{v, i}].
			\end{align*}

			The estimate as in \cite[Eq.~(3.9)]{matsuda22} gives
			\begin{equation*}
				\abs{D^0_{s, u, t}} 
				\lesssim_{\gamma} M(\xi) \big[ \rho^{(\gamma-1)}(v, s) \abs{Y^{\tau, v}_{s,u}}\abs{Y^{\tau, v}_{u,t}} 
				+ \rho^{\gamma -1}(v, s)[\rho(v, t) - \rho(v, s)] (\abs{Y_{s, u}^{\tau, v}} + \abs{Y^{\tau, v}_{u, t}}) \big].
			\end{equation*}
			To estimate $\rho(v, t) - \rho(v, s)$, we observe that 
			\begin{align*}
				\rho(v, t) - \rho(v, s) &= (\rho(v, t) + \rho(v, s))^{-1} [\rho^2(v, t) - \rho^2(v, s)] \\
				&\leq 2 \rho^{-1}(v, s) [\rho^2(v, t) - \rho^2(v, s)] \\
				&\lesssim \rho^{-1}(v, s) [(t-s)^{2H} + (s-v)^{2H - 1} (t-s)],
			\end{align*}
			where the last inequality follows from Lem.~\ref{lem:rho_and_tilde_B}.
			Using the bound on $\rho$ from \eqref{eq:rho_lower_bd} and that on $Y^{\tau, v}$ from \eqref{eq:Y_tau_v_estimate}, 
			we obtain
			\begin{equation*}
				\norm{D^0_{s, u, t}}_{L^p_{\omega}} 
				\lesssim_{p, \gamma} \norm{M(\xi)}_{L^{2p}_{\omega}} [(s - v)^{(\gamma + 1)H - 2} (t - s)^2
				+ (s - v)^{(\gamma - 1) H - 1} (t - s)^{1 + 2H}]. 
			\end{equation*}
			Since 
			$(s-v)^{(\gamma + 1)H -2} (t-s)^2 \leq (s-v)^{(\gamma - 1)H - 1} (t-s)^{1 + 2H}$,
			we have 
			\begin{align*}
				\norm{D^0_{s, u, t}}_{L^p_{\omega}} 
				\lesssim_{p, \gamma} \norm{M(\xi)}_{L^{2p}_{\omega}} 
				(s - v)^{(\gamma - 1) H - 1} (t - s)^{1 + 2H}. 
			\end{align*}

			To estimate $D^i_{s, u, t}$, as in \cite{matsuda22} (see the two estimates before (3.12) therein), we have 
			\begin{align*}
				\abs{a_i(t) - a_i(u)} \lesssim_{\gamma} M(\xi) \rho^{\gamma-1}(v, u) (\abs{Y^v_{u, t}} + \rho(v, t) - \rho(v, u)).  
			\end{align*}
			Therefore, estimating $Y^v$ and $\rho(v, t) - \rho(v, u)$ as before, we obtain 
			\begin{align*}
				\norm{a_i(t) - a_i(u)}_{L^p_{\omega}} 
				&\lesssim_{p, \gamma} \norm{M(\xi)}_{L^{2p}_{\omega}} 
				[\rho^{\gamma-2}(v, u) (u-v)^{H - 1} (t-u) + \rho^{\gamma-3}(v, u) (t-u)^{2H} ] \\
				&\lesssim_{p, \gamma} \norm{M(\xi)}_{L^{2p}_{\omega}} (u-v)^{(\gamma - 3)H} (t-u)^{2H}.
			\end{align*}
			Combining the estimate of $\mathbb{E}[\tilde{B}^{v, i}_u \tilde{B}^{v, i}_{u, t}]$ from Lem.~\ref{lem:rho_and_tilde_B}-(i), we obtain
			\begin{equation*}
				\norm{D^i_{s, u, t}}_{L^p_{\omega}} \lesssim_{p, \gamma} \norm{M(\xi)}_{L^{2p}_{\omega}}
				(s - v)^{(\gamma - 3) H} (t - s)^{4H}.
			\end{equation*}
			Now the estimate \eqref{eq:delta_A_conditional_stratonovich} is proven.
		\end{proof}
		\begin{proof}[Proof of Lem.~\ref{lem:integral_along_tilde_B}]
			In view of Lem.~\ref{lem:ssl} and Lem.~\ref{lem:integral_stratonovich}, 
			it suffices to show 
			\begin{align}\label{eq:f_integral_by_A_2}
				\int_s^t F(\xi, B^H_{\tau, r}) \, \mathrm{d} \tilde{B}^{H, \tau}_r 
				= \lim_{\abs{\pi} \to 0} \sum_{[s', t'] \in \pi} A_{s', t'},
			\end{align}
			where 
			$A_{s, t} \defby 2^{-1}(F(\xi, B^H_{\tau, s}) + F(\xi, B^H_{\tau, t})) \tilde{B}^{H, \tau}_{s, t}$.
			But the proof goes in the same way as Lem.~\ref{lem:rough_vs_strat_2}, so we omit the proof.
		\end{proof}
		%
	\end{appendix}

\bibliography{ref}{}
\bibliographystyle{alpha}

\end{document}